\documentclass[11pt,a4paper]{article}

\usepackage{amsmath}
\usepackage{amsfonts}
\usepackage{amsthm}
\usepackage{amssymb}
\usepackage{url}
\usepackage[usenames,dvips]{color}
\usepackage{stmaryrd}
\usepackage[pdftex]{graphicx}

\DeclareMathAlphabet{\mathpzc}{OT1}{pzc}{m}{it}

\newtheorem{theorem}{Theorem}[section]
\newtheorem{lemma}[theorem]{Lemma}

\newtheorem{remark}[theorem]{Remark}

\newtheorem{proposition}[theorem]{Proposition}

\newcommand{\RR}{\mathbb{R}}
\newcommand{\NN}{\mathbb{N}}
\newcommand{\CC}{\mathbb{C}}

\newcommand{\PP}{\mathcal{P}}

\newcommand{\II}{\mathcal{I}}
\newcommand{\BB}{\mathcal{B}}

\newcommand{\X}{X^+}
\newcommand{\Y}{X^-}
\newcommand{\x}{x_h}
\newcommand{\y}{y_h}
\newcommand{\OO}{\mathcal{O}}

\newcommand{\FF}{\mathcal{F}}

\newcommand{\VV}{\mathcal{V}}

\newcommand{\JJ}{\mathcal{J}}

\newcommand{\eps}{\varepsilon}
\newcommand{\ga}{\gamma}

\newcommand{\la}{\lambda}
\newcommand{\al}{\alpha}

\newcommand{\be}{\begin{equation}}
\newcommand{\ee}{\end{equation}}

\newcommand{\D}{\Omega}

\begin{document}

\title{Regularization of sliding global bifurcations derived from the   local fold singularity of Filippov systems.}
\author{Carles Bonet Rev\'{e}s\thanks{\tt carles.bonet@upc.edu} \ and  Tere M. Seara\thanks{\tt tere.m-seara@upc.edu}}
\maketitle


\begin{abstract}
In this paper we study the Sotomayor-Teixeira regularization of a general visible fold singularity of a Filippov system. 
Extending Geometric Fenichel Theory beyond the fold with asymptotic methods, we determine there the deviation of the orbits
of the regularized system  from the generalized solutions of the Filippov one. 
This result is applied to the regularization of some global sliding bifurcations as the Grazing-Sliding of periodic orbits 
or the Sliding Homoclinic to a Saddle, as well as to some classical problems in dry friction. 

Roughly speaking, we see that locally, and also globally, the regularization of the bifurcations preserve the topological features of the sliding ones.      
\end{abstract}
\section{Introduction}

In recent years there has been an increasing research in piecewise differentiable vector fields. 
This kind of systems model many phenomena in control theory, in mechanical friction and impacts, in hysteresis in electrical circuits and plasticity, 
etc... See \cite{diBernardoBCK08} for a general scope of the matter. 
In a piecewise differentiable vector field the phase space is divided into several regions where the system takes different smooth forms. 
The degree of discontinuity in the edge between two adjacent regions, usually called switching manifold, is used to classify them. 
Vector fields  with jump discontinuities are usually named Filippov Systems.

In Filippov systems the derivatives of the state variables are no longer uniquely determined as at the switching manifold they can take 
values in a whole interval. 
For the study of these systems, it has been generalized the concept of differential equation to a more general differential inclusion.
The theory developed for these systems has succeeded to proof, under general conditions, 
theorems related to the existence and uniqueness of solutions (\cite{Kunze00b}). 
Moreover, over the switching manifold, using the Filippov convention (\cite{Filippov88}), 
one can define a vector field made up from a  certain linear convex combination of two adjacent equations.
Although other possible conventions can be more suitable in some cases, as the Utkin's equivalent control (\cite{Utkin}),  
in this paper we restrict ourselves to the Filippov convention.

The non-smooth mathematical models are often a discontinuous idealization of regular phenomena where the phase space is divided into regions 
where the variables have different orders of behavior (slow-fast regions, for example). 
It is natural to ask if the generalized solutions of these discontinuous models are close to the solutions of the corresponding real regular ones. 
A natural question is whether a discontinuous system can be embedded in a set of parametric regular systems in such a manner that the discontinuous 
one will be, in some sense, their limit.
But as noted in \cite{Utkin}, not only there is not an unambiguous regularization technique but different regularization techniques can lead 
to different ways of defining the edge solutions. 
The way  chosen will depend on their suitability to model the problem. 
For example in the case of dry friction systems that we deal with in section \ref{dryfriction}, 
the regularization should be different if we use the stiction friction model or the Coulomb model, in spite of both models are identical 
outside the switching manifold.

In this paper we work with Filippov systems in the plane and we use the regularization method proposed by Sotomayor and Teixeira \cite{SotomayorT95}, 
based in replacing the two adjacent fields by an $\eps$-parametric field built as a linear convex combination of them in a $\eps$-neighbourhood 
of the switching manifold. The regularized system so obtained is a slow-fast system on the plane.

It is known \cite{TeixeiraBuzziSilva,TeixeiraS12} that, under general conditions, in some compact regions near the switching manifold  
(the so-called  sliding and escaping zones which do not contain 
the tangency points between the vector fields and the switching manifold) 
the regularized system has, for small values of the parameter $\eps$, a normally hyperbolic invariant manifold 
(attracting near the sliding region or repelling near the scaping one) which is $\eps$-close to the switching manifold. 
Furthermore, the flow of the regularized vector field reduced to  this invariant manifold 
tends to the Filippov flow.

Therefore, the  results in  \cite{TeixeiraBuzziSilva,TeixeiraS12} give a  partial positive answer to the main question of this paper: 
the solutions of the regularized vector field are well approximated by the Filipov ones  in these regions. 
This result can be proved in several ways but for ours aims we stress the methods issued from the geometrical theory of singular 
perturbation of N.Fenichel and others \cite{Fenichel79,Jones94,Kaper99}.
 
But as one approaches to a boundary of the sliding (or scaping) region, that is, a point of tangency of 
one of the vector fields with the switching manifold (called in \cite{GuardiaST11}  fold-regular point) this theory fails 
because the tangency point of the Filippov vector field  creates a fold point in the  slow 
manifold of the regularized vector field and, therefore, the invariant manifold looses its hyperbolicity.
At this stage, the theory needs to be combined with other tools, like asymptotic or blow-up methods, to understand the behavior 
of the manifold near the fold point.

In \cite{KutznesovRG03, GuardiaST11}, a systematic topological classification and normal forms for different types of tangency points 
of Filippov vector fields and their bifurcations is made. 
It is therefore natural to study the regularization of these normal forms to determine in 
which cases the dynamics of the regularized normal forms 
moves towards  the corresponding one in the Filippov system. 
Although  in this paper we only examine in detail the regularization of the normal form of a visible  
tangency, we think that the same approach can be used to study the other tangencies.

With the tools provided by singular perturbation theory and asymptotic expansions, following \cite{MischenkoR80}, 
we analyse how the normally hyperbolic invariant manifold deviates in passing around the fold and we determine 
regions close to  the fold exponentially attracted to this variety. 
Then we conclude that the orbits issuing from these regions, after passing near the tangency, are concentrated in an exponentially 
small neighborhood of the extended invariant manifold provided by Fenichel theory. 
Moreover, the deviation of the invariant manifold is leaded by a distinguished solution of a Riccati equation, a typical result in singular perturbed 
systems around the singular points of the slow manifold (\cite{MischenkoR80, Bonet87, KrupaS97}).
One can then conclude that, also close to a visible fold-regular point, the regularized system behaves closely to the Filippov one.

From the work of Dumortier, Krupa, Roussarie, Szmolian, Wechselberger  (\cite{DumortierR96, Szmolyan01, KrupaS01}) and others, 
the blow-up technique is used as a geometrical alternative to asymptotic methods. 
Nevertheless, we have decided to use these last methods because we only 
need to arrive until the lower half region of the fold and the calculations involved are no too difficult.
Furthermore, the careful analysis needed to control the regions exponentially attracted by the invariant manifold  is made comfortably with these methods.

The qualitative results obtained  in this work do not depend of the degree of smoothness of the regularized system but the quantitative ones do.
In the case that the regularized system is $\mathcal{C}^1$, that is, the contact of the regularized field and the two adjacent fields in the 
boundary of the regularization zone is strictly of order one, 
we  proof the  well known result that the deviation of the invariant manifold is $O(\eps ^{\frac{2}{3}})$.
But we think is worth to derive it in the setting of piecewise differentiable systems and also as a basis to
extend it to the $\mathcal{C}^{p-1}$ contact, where we find that the deviation is $O(\eps ^{\frac{p}{2p-1}})$. 

A crucial result in our work is to see that the invariant manifold attracts a region near the
sliding region which contains points up to a distance of order $\eps ^{\gamma}$, $\gamma<\frac{p}{2p-1}$,  to the tangency 
point. 
Moreover, the fact that the regularization only takes place in an $\eps$-neighborhood of the switching manifold, 
remaining unaltered the adjacent fields outside, makes easier to  analyze global properties 
of the system. 
If the field tangent to the switching manifold has any stable recurrence, such as a (sliding or grazing) periodic orbit or a 
sliding homoclinic orbit to a hyperbolic saddle, 
the exponential flattening to the slow manifold of sliding areas  $\eps ^{\gamma}$-near the fold, 
will ensure recurrence also in the regularized systems, and a return  Poincare map can be determined and computed.
 
All this will allow us to study the existence of global periodic orbits in the regularized system in different settings, 
like in one parameter Filippov families of vector fields having a grazing-sliding bifurcation of periodic orbits or a sliding homoclinic bifurcation.
We will also apply our results to some classical examples as the dry friction models.

The paper is organized as follows.

In section 2 we introduce the notation, the basic concepts of a Filippov vector field in the plane and we present the 
Sotomayor-Teixeira regularization. 
To study the dynamics near a fold-regular points we introduce Poincar\'{e} sections and  a Poincar\'{e} map near the fold.  
The main theorem of the paper is Theorem \ref{thm:main}, where we give the main asymptotic properties of this Poincar\'{e} map.
The proof of this theorem, rather cumbersome, is given in section \ref{sec:proof}. 
The main idea is to use the fact that the regularized vector field and the Filippov one are identical everywhere except 
in a region near the switching manifold which is of order $\eps$. 
So the main part of the proof is to study the behavior of the regularized system, which turns to be a slow-fast system, in this region. 
This study is done using geometric singular perturbation theory, which provides  the existence 
of a normally  attracting invariant manifold $\Lambda _\eps$ of the system.
Once we have this invariant  manifold we need to extend it to see two things: 
on the one hand we have  that this manifold exponentially attracts a region which contains points which are at a distance 
of order $\eps ^{\gamma}$, $\gamma<\frac{p}{2p-1}$,  to the origin 
(see propositions \ref{prop:atractiogranlineal}, \ref{prop:atractiogran}, \ref{prop:atractiogranp}). 
On the other hand, we need to give an asymptotic expression of this invariant manifold when it arrives to the border of the regularized region
(see propositions \ref{varietatconfinadalineal}, \ref{blocouter}, \ref{prop:blocinner},  \ref{blocouterp}, \ref{prop:blocinnerp}). 
This last part is done using asymptotic expansions and matching methods to obtain a suitable inner equation.

Although we study in detail the $\mathcal{C}^1$ regularization of the normal form of the visible fold, in sections 
\ref{sec:spc} and \ref{sec:generalfold} 
we show that the techniques used and the results generalize straightforwardly  to $\mathcal{C}^{p-1}$  regularizations and generic folds. 

Besides a greater complication of the computations, the only delicate issue to study the $\mathcal{C}^{p-1}$ case, 
is the determination of the distinguished solution of the equation
$$
 y '= x + y^p
$$
that appears as a dominant term in the asymptotic development near the fold. 
This equation is well known in the case $p=2$ (see \cite{MischenkoR80}) but, as far as the authors know, the general case has not been done before. 
In propositions \ref{prop:asymptoticsp}, \ref{prop:blocinnerp}
we proof that this solution leads, as in the  $\mathcal{C}^1$ case, the deviation of the invariant manifold, 
which turns out to be $O(\eps ^{\frac{p}{2p-1}})$. 

Once we have our main local result in Theorem \ref{thm:main}, in Theorem \ref{thm:po} we analyze the existence of periodic orbits in 
the regularized system assuming that the Filippov vector field has some global recurrence which typically occurs near a grazing sliding bifurcation.
Finally, Theorem \ref{thm:sella-node} studies the possible global bifurcations of periodic orbits in the regularization of  a one parameter Filippov 
vector field undergoing a grazing-sliding bifurcation. As expected, we  see that the grazing-sliding bifurcation of a hyperbolic attracting  periodic 
orbit leads to a structurally stable periodic orbit in the regularized system and the   grazing-sliding bifurcation of a hyperbolic repelling  
periodic orbit creates a saddle-node bifurcation of periodic orbits  in the regularized system.

Also in section \ref{dryfriction} we consider the three basic models of dry friction in single degree of freedom systems, 
following the formulation described in  \cite{Leine00b, Leine00}. 
We see that only the Stribeck model fulfills our hypotheses to directly conclude the existence of 
attracting periodic orbits of the regularized system. 
Nevertheless, in Theorem \ref{thm:centre},  we will see that our methods will be able to ensure the existence of 
periodic orbits also in the Coulomb model, 
in spite of the neutral character of the tangent orbit (it belongs to a centre). 
The exponential concentration of the regularized field to a neighborhood of the Fenichel variety combined with the  
return that provides the centre will guaranty that the unique orbit of the non-smooth system tangent to the border of the 
regularization zone is semi-stable, that is, attracts all 
the regularization strip.

However, this regularization does not apply for the Stiction model as the mechanical analysis in the switching manifold gives an 
equation different from  the Filippov one. 
It is clear that a different  regularization will be needed as the phase portrait of the slip Stiction model equations 
is identical to Coulomb and therefore the regularized system would tend to the Filippov dynamics. 
This case is beyond the scope of this article and will be studied later.

The last results of the paper deal with the existence of periodic orbits (and homoclinic ones) in the regularized system when the Filippov 
system has a sliding homoclinic orbit to a saddle, creating a pseudo-separatrix connection between  a saddle and a fold (\cite{KutznesovRG03}). 
This is a codimension one phenomena and therefore appears generically in some one-parameter families.
Theorem \ref{thm:separatrix} studies the general case, showing the existence in the regularized system of a so-called homoclinic 
bifurcation where the periodic orbit dies in a homoclinic one and then disappears. Theorem  \ref{thm:hseparatrix} studies the corresponding 
bifurcation in the Hamiltonian case where the existence of a homolinic orbit is generic.

We want to conclude by emphasizing that, eventhough  this work studies the generic case of a generic visible 
fold-regular point in a Filippov vector field 
in the plane, we think  that the methods used here can be useful to study local bifurcations as fold-fold points and also higher dimensional 
Filipov systems.
We also expect to extend these results to the case where the regularized vector field is analytic. The main novelty of this case 
will be that the regularized vector field and the Filippov one and different in the whole phase space but this is just a technical 
problem that will not change the final results.

\section{Hypotheses and main results}\label{sec:hypotheses}

The main goal of this section is to introduce the regularization of a  Filippov vector field in the plane near a visible fold-regular point
and give the main results of the paper.
Therefore, we consider a non-smooth system in $\RR^2$:
\begin{equation}\label{def:Filippov}
Z(x,y)=\left\{\begin{array}{l}
        \X(x,y),\, (x,y)\in\VV^+\\
        \Y(x,y),\, (x,y)\in\VV^-,
       \end{array}\right.
\end{equation}
where: $\VV^+=\{(x,y), y> 0\}$, 
$\VV^-=\{(x,y), y< 0\}$ with a  switching manifold  given by:
$$
\Sigma=\{(x,y), y= 0\}.
$$ 
We assume that the vector fields $\X$ and $\Y$ have an extension to $\Sigma$ which is, at least  
$\CC^{2}$, and we denote their flows by $\phi_{\X}$ and  $\phi_{\Y}$ respectively.

Without loss of generality we can assume that the fold point is at $(0,0)$.
We assume that the vector field $\Y$ is transversal to $\Sigma$ and that $\X$ has a generic fold in $\Sigma$, that is:
\begin{equation}\label{generalform}
\begin{array}{rcl}
\X(0,0) &=& (\X_1(0,0),0), \quad \X_1(0,0) \ne 0 ,\quad \frac{\partial \X_2}{\partial x}(0,0)\ne 0\\
\Y(0,0)&=&(\Y_1(0,0),\Y_2(0,0)), \quad  \Y_2(0,0)\ne 0 .
\end{array}
\end{equation}
We will consider the case where:
\begin{equation}\label{cond:visiblefold}
\Y_2(0,0)>0, \mbox{and} \ \X_2(x,0) <0 \ \mbox{for}\ x<0, \   \X_2(x,0) >0\ \mbox{for}\ x>0.
\end{equation} 
These conditions ensure that $(0,0)$ is a generic visible fold-regular point. 
As $\X _1(0,0)\ne 0$, we will deal with the case
\begin{equation}\label{cond:visiblefold1}
\X_1(0,0)>0,
\end{equation} 
which implies that $\X$ goes ``to the right".  Analogous results are true for $\X_1(0,0)<0$.

The fold point divides, locally, the switching manifold $\Sigma$ in  two regions: 
\begin{equation}\label{def:sliding}
\begin{array}{rcl}
\Sigma^s&=&\{(x,0), x <0\}\ \mbox{the sliding region} \\
\Sigma^c&=&\{(x,y), x > 0\}\ \mbox{the crossing region}
\end{array}
\end{equation}
Also, following \cite{GuardiaST11}, we define 
\begin{equation}
W^s_+(0,0)=\{ \phi_{\X}(t;0,0), \ t < 0 \}, \quad W^u_+(0,0)=\{ \phi_{\X}(t;0,0), \ t > 0\}
\end{equation}
the stable and unstable pseudo-separatrices  in $\VV^+$ of the fold point $(0,0)$. 
Under our hypotheses, the fold point also has a  stable pseudoseparatrix
in $\VV^-$, but it does not play any role in our setting.

As usual in non-smooth vector fields, we consider the flow of a point  $p\not \in \Sigma$ as given by the flows of the vector fields $\X$ or $\Y$, respectively, depending if $p\in \VV^\pm$. 
If the point $p$ belongs to the switching manifold $\Sigma$ in the crossing region  $\Sigma ^c$ we concatenate both flows in a consistent way.
Moreover, with the Filippov convention \cite{Filippov88}, we can define a sliding vector field in the sliding region $\Sigma ^s$, that, in our case, reads:
$$
\dot x= \frac{\X_1 \Y_2-\Y_1\X_2}{\Y_2-\X_2}(x,0), \ x<0 .
$$
This allows us to define a flow in the whole  neighborhood of $(0,0)$
 (see \cite{GuardiaST11}).

Moreover, under  conditions \eqref{generalform}, \eqref{cond:visiblefold} and \eqref{cond:visiblefold1}, we also have, for $x<0$, small enough:
\begin{equation}\label{positivefilipov}
\X_1\Y_2-\Y_1\X_2 >0
\end{equation}
which gives that the Filipov vector field also moves ``to the right".

To study the behavior near the fold, we consider  any value $y_0>0$ and the Poincar\'{e} sections
$$
\Sigma ^-_{y_0} =\{ (x, y_0), \ x<0\}, \quad \Sigma ^+ _{y_0}=\{ (x, y_0), \ x>0\}.
$$

We denote by 
$$
 (x_0^\pm, y_0)=W^{u,s}_+(0,0)\cap \Sigma ^\pm_{y_0} 
$$
and we assume that $y_0$ is small enough in such a way that these intersections are transversal.
\begin{figure}
\begin{center}
\includegraphics[width=10cm]{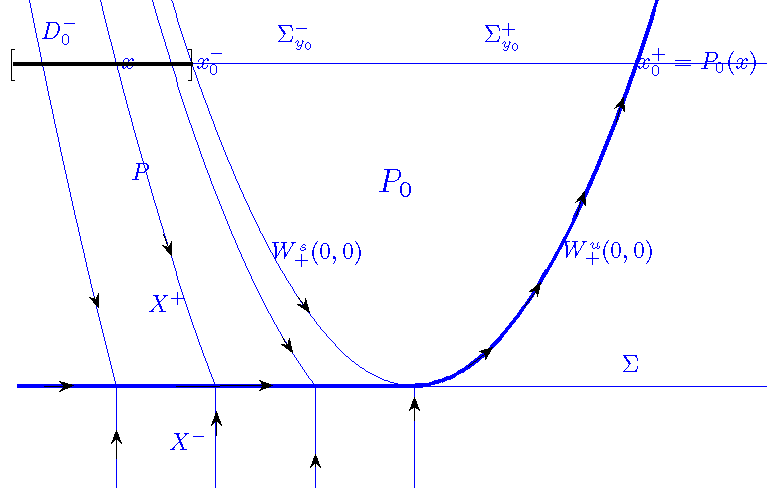}
\end{center}
\caption{The Poincar\'{e} map for the Filippov system.}
\label{fig:P0final}
\end{figure}

We consider the Poincar\'e map:
\begin{equation}\label{Poincare}
\begin{array}{rcl}
P_0:D_0\times \{y_0\}\subset \Sigma  ^-_{y_0}&\to& \Sigma ^+_{y_0}\\
(x,y_0) &\mapsto &(P_0(x), y_0).
\end{array}
\end{equation}
For the Filipov system \eqref{def:Filippov}, all the trajectories of the system beginning 
at $(x,y_0) \in D_0\times \{y_0\}$ with $ x \le x_0^-$
arrive to the sliding region $
\Sigma ^{s}$ (see 
\eqref{def:sliding}),
then slide until they leave the switching manifold $\Sigma$ at the fold $(0,0)$ following its unstable pseudoseparatrix $W^u_+(0,0)$ 
(see figure \ref{fig:P0final}). 
Therefore the map $P_0$ is constant in $D_0^-$:
$$
P_0(x) = x_0^+, \ \forall x \in D_0^-=\{x\in D_0, \quad x\le x_0^-\} .
$$

\subsection{The regularized system near the fold}\label{sec:trsnf}

As the non-smooth system $Z$ in \eqref{def:Filippov} can be written  as:
$$
Z (x,y)= \frac{\X(x,y)+\Y(x,y)}{2}+\Xi(y) \frac{\X(x,y)-\Y(x,y)}{2},
$$
where the function $\Xi$ is the discontinuous function:
$\Xi : \RR \to \RR $, defined by:
$$
\Xi(z)=\left\{\begin{array}{ccc} -1 &\mbox{if} & z<0 \\
1 &\mbox{if} & z>0 \end{array}\right . ,
$$
a classical way to regularize the vector field $Z$ \cite{SotomayorT95} is to consider vector fields $Z_\eps$:
\begin{equation}\label{regularizedvf}
Z_\eps (x,y)= \frac{\X(x,y)+\Y(x,y)}{2}+\varphi(\frac{y}{\eps}) \frac{\X(x,y)-\Y(x,y)}{2},
\end{equation}
where we can  take any increasing smooth function $\varphi$ which approximates the discontinuous function $\Xi$ and verifies:
$$
\varphi(v) =-1, \ \mbox{for} \ v\le -1, \quad \varphi(v) =1, \ \mbox{for} \ v\ge 1.
$$
Let us point out that, with these smooth regularizations, outside the regularized zone $|y|\le \eps$, the regularized vector field
$Z_\eps$ coincides with the non-smooth one $Z$. This would not be the case if we chose an analytic function $\varphi$ in \eqref{regularizedvf}. In that case $Z_\eps$ and $Z$ would be different everywhere and this will  be the study of a future work.
 
In Theorem \ref{thm:main} we will give and asymptotic expansion, for $\eps$ small enough,  of the Poincar\'{e} map 
$$
P_\eps: D_\eps\times \{y_0\}\subset \Sigma ^-_{y_0} \to \Sigma ^+_{y_0},
$$ 
which is the Poincar\'{e} map for the regularized system $Z_\eps$.

We denote $(x_\eps, \eps)$ to  the point where the vector field $\X$ has a tangency with the horizontal line $y=\eps$, that is 
\begin{equation}\label{eq:tangent}
\X_2(x_\eps,\eps)=0
\end{equation}
and by 
$(\bar x_\eps, y_0)$ the intersection of its orbit by $\X$ with $\Sigma _{y_0}^-$, that is
\begin{equation}\label{eq:tangentbar}
(\bar x_\eps,y_0)= \phi _{\X}(t; x_\eps, \eps) \in \Sigma_{y_0}^-
\end{equation}
for some suitable $t<0$  (see figure \ref{fig:Pepsfinal}). Clearly, by \eqref{generalform}, $x_\eps  = O(\eps)$.

It is clear that, for $x\in D_\eps$ such that $x>\bar x_\eps$, one has $P_\eps (x)=P_0(x)$.
Therefore, we will restrict our study of the Poincar\'{e} map $P_\eps$ to the interval
$[k_\eps,\bar x_\eps]\subset D_\eps$, 
where $k_\eps <x_0^-$ is a suitable constant which depends of the global properties of $Z_\eps$.

In $[k_\eps,\bar x_\eps]$, it will be convenient to write  the map $P_\eps= \bar P\circ \PP_\eps \circ  P$ (see figure \ref{fig:Pepsfinal}), where 
\begin{eqnarray*}
P : \Sigma_{y_0} ^-& \to &\Sigma _{\eps}^-\\
\PP_\eps : \Sigma_\eps ^-& \to &\Sigma _{\eps}^+\\
\bar P : \Sigma_\eps ^+& \to &\Sigma _{y_0}^+ .
\end{eqnarray*}

\begin{figure}
\begin{center}
\includegraphics[width=10cm]{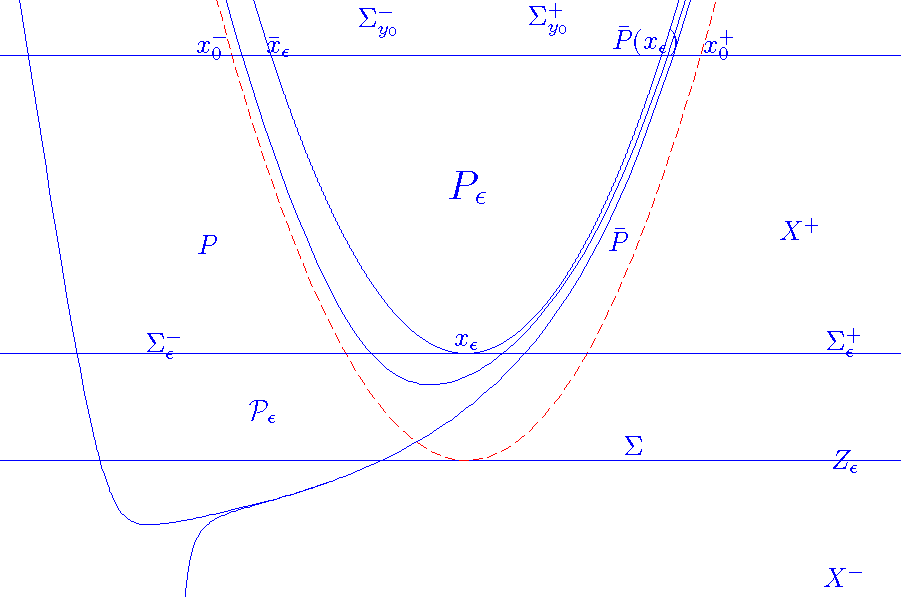}
\end{center}
\caption{The Poincar\'{e} map for the regularized system $Z_\eps$.}
\label{fig:Pepsfinal}
\end{figure}

The map $\PP_\eps$ is defined in the region where the 
regularized system $Z_\eps$ and the original Filipov one $Z$ are different. 
Its study will be one of the main goals of the paper and will be done using Geometric Singular Perturbation Theory in section \ref{sec:proof}.

Clearly $P$ and $\bar P$ are the same for $Z$ and the regularized system $Z_\eps$.
In  fact, they are Poincar\'{e} maps associated to the vector field $\X$.
Their asymptotic expressions for $\eps$ small enough are an easy consequence of next proposition.

\begin{proposition}\label{prop:flowtangency}
Consider the  pseudoseparatrices of the fold $W^{u,s}_+(0,0)$,  and the points  $(x_0^\pm,y_0)=W^{u,s}_+(0,0)\cap \Sigma ^\pm_{y_0}$ 
and assume that these intersections are tranversal, that is $\X_2(x_0^\pm,y_0)\ne 0$.
Denote by  $T^\pm$ the time such that
$\phi _{\X} (T^\pm;0,0) \in \Sigma ^\pm_{y_0}$, where $\phi_{\X}(t;x,y)$ is the flow of the (regular) vector field $	\X$. Consequently
$\phi _{\X}(T^\pm;0,0)=(x^\pm_0, y_0)$.

Then, there exists a neighborhood $U$ of the origin such that, for any 
$(x,y)\in U$, there exist regular functions 
\begin{eqnarray*}
\tau ^\pm : U &\to &\RR \\
(x,y) &\mapsto &\tau ^\pm(x,y)
\end{eqnarray*}
such that, $\phi_{\X}(\tau^\pm (x,y);x,y)\in \Sigma ^\pm_{y_0}$.
Moreover:
\begin{itemize}
\item
$\tau ^\pm(0,0)=T^\pm$
\item
If $(x,y) \in U$, one has
$$
\phi_{\X}(\tau^\pm (x,y); x,y)= \left(x^\pm _0+ \alpha ^\pm y+ \beta ^\pm x^2+ O(x y,y^2), \, y_0\right)
$$
with $\alpha ^+ <0$, $\beta ^+ >0$, $\alpha ^->0$, $\beta ^- <0$.
\end{itemize}
\end{proposition}

\begin{proof}
Let's consider the flow of $\X$, $\phi  _{\X}(t;x,y)$.

The existence of  the functions $\tau ^\pm(x,y)$ is a consequence of the implicit function theorem applied to the equation
$m(t,x,y)=0$, where $m(t,x,y)=\pi _y (\phi _{\X} (t;x,y))-y_0$ near $(T^+,0,0)$ and $(T^-,0,0)$ respectively.

On one hand we have that $m(T^\pm,0,0)=0$ and the transversality of the intersections of $W^u_+(0,0)\cap \Sigma ^+_{y_0}$ and 
$W^s_+(0,0)\cap \Sigma ^-_{y_0}$ gives
$\frac{\partial m}{\partial t}(T^\pm,0,0)= \X_2(x_0^\pm, y_0)\ne 0$.

We compute $\phi_{\X}(t;x,y) $ developing by  Taylor at $(x,y)=(0,0)$: 
\begin{equation}\label{variacional}
\phi_{\X}(t;x,y)= \phi_{\X}(t;0,0) + D\phi_{\X}(t;0,0)\left( \begin{array}{c}x\\y\end{array}\right) + O_2(x,y).
\end{equation}
We observe that 
$D\phi_{\X}(t;0,0)$ is the fundamental matrix of the variational equations:
$$
z'=D\X(\phi_{\X}(t;0,0))z, \quad \mbox{satisfying}\quad D\phi_{\X}(0;0,0)= \mathrm{Id}.
$$ 

We know that $ \phi_{\X }'(t;0,0)$ is a solution of the variational equations and that, by hypotheses \eqref{generalform},  $\phi'_{\X}(0;0,0) = (\X_1(0,0),0)$, therefore, one can take 
$z_1(t) = \frac{1}{\X_1(0,0)}\phi' _{\X}(t;0,0)$
and look for an independent solution $z_2(t)$ of the variational equation in such a way that:
$D\phi_{\X}(t;0,0)= \left(\begin{array}{cc}  z_1(t)& z_2(t)\end{array}\right)$.

By the implicit function theorem we know that:
$$
D\tau ^{\pm}(0,0) = -\frac{1}{\partial _ t m(T^\pm,0,0)} D m(T^\pm,0,0)= 
-\frac{1}{y_0'(T^\pm)}\left(\frac{y_0'(T^\pm)}{\X_1(0,0)}, \pi _y (z_2(T^\pm))\right)
$$
where we have denoted by $(x_0(t),y_0(t))= \phi_{\X}(t;0,0)$.

Now, using \eqref{variacional},  we compute:
$$
\pi_x (\phi_{\X}(\tau ^{\pm},x,y))= x_0(\tau ^\pm)+ \frac{1}{\X_1(0,0)}x_0'(\tau^\pm)x + \pi _x (z_2 (\tau ^\pm))y + O_2(x,y).
$$
Using the Taylor expansion of $\tau^\pm$ and also expanding the above expression for $x_0(t)$ we obtain:
\begin{eqnarray*}
\pi_x (\phi_{\X}(\tau ^{\pm},x,y))&=& x_0(T^\pm)-x_0'(T^\pm)\frac{1}{y_0'(T^\pm)}\left(\frac{y_0'(T^\pm)}{\X_1(0,0)}x+\pi _ y(z_2(T^\pm))y\right)\\
&+&\frac{1}{\X_1(0,0)} x_0'(T^\pm)x + \pi _x (z_2 (T^\pm))y + O_2(x,y)\\
&=&x_0^\pm + \alpha ^\pm y + O_2 (x,y) =
x_0^\pm + \alpha ^\pm y + \beta^\pm x^2+O (xy,y^2) .
\end{eqnarray*}
The signs of the constants $\alpha ^\pm$ and $\beta^\pm$ are a consequence of the fact that the orbits of a vector field on the plane can not intersect.
\end{proof}

From this proposition, it is clear that, if $(x,\eps) \in U$:
\begin{equation}\label{eq:ppg}
P^{-1}(x)= x^-_0 +\alpha ^- \eps+ \beta^- x ^2 +O(\eps x,\eps^2), \quad  
\bar P(x )= x^+_0 + \alpha ^+ \eps +\beta ^+ x^2+ O( \eps x, \eps^2).
\end{equation}

Observe that, the domain of $\bar P$ is $U^+=[x_\eps, k^+]$ where the point $(x_\eps, \eps)$ 
corresponds to the point  \eqref{eq:tangent} where the vector field $\X$ has a tangency with the horizontal line $y=\eps$,
and $k^+$ is a suitable constant independent of $\eps$.  
Analogously, the domain of $P$ is $U^-=[K^-, \bar x_\eps]$, were the point $\bar x_\eps =P^{-1}(x_\eps)$
was defined in \eqref{eq:tangentbar}.
 
As $x_\eps  = O(\eps)$,  using the formulas given in \eqref{eq:ppg}:
\begin{equation}\label{xepsilon}
\begin{array}{rcl}
\bar P(x_\eps)&=& x^+_0+ \alpha ^+ \eps +O(\eps ^2)\\
\bar x_\eps=P^{-1}(x_\eps)&=& x^- _0+\alpha ^- \eps +O(\eps ^2). 
\end{array}
\end{equation}
Summarizing,  one has that 
\begin{eqnarray*}
\bar P: [x_\eps, k_+] &\to &[\bar P(x_\eps), K^+]\\
P:[K^-,\bar x_ \eps] &\to& [k_-, x_\eps ].
\end{eqnarray*}

Section \ref{sec:proof} is devoted to study the Poincar\'{e} map $\PP_\eps$ after the regularization.
Combining the behavior of $\PP_\eps$ with the maps $P$ and $\bar P$ we will obtain the asymptotics for $P_\eps$.
We will consider different functions $\varphi$ with different regularity and we will study how 
the properties of the regularized system depend on this regularity. Moreover, using geometric singular 
perturbation theory and matching asymptotic expansions, we will give asymptotic formulas for the Poincar\'{e} map $P_\eps$. 

There are two significantly different cases:
\begin{itemize}
\item
$\varphi$ is a continuous piecewise linear function:
\begin{equation}\label{caslineal}
\varphi(v)=
\left\{\begin{array}{ccc} -1 &\mbox{if} & v\le-1 \\
v &\mbox{if} & -1<v< 1 \\
1 &\mbox{if} & v\ge 1 .
\end{array}\right.
\end{equation}
\item
$\varphi$ is a $\CC ^{p-1} $ function, $p\ge 2$,  such that:
\begin{equation}\label{diferenciable}
\varphi(v)=\left\{ \begin{array}{ccc} -1 &\mbox{if} & v\le -1 \\
1 &\mbox{if} & v\ge 1,
\end{array} \right. 
\end{equation}
and is $\CC^{\infty}$ for $-1<v<1$.
Therefore, locally, near $v=1$, and for $v\le 1$, it will behave as 
\begin{equation}\label{eq:fipropdeu}
\varphi(v) \simeq 1+ O(v-1)^p .
\end{equation}
\end{itemize}

Next theorem gives the asymptotic behavior of the Poincar\'{e} map $P_\eps $ in terms of the regularity of $\varphi$ 
(see also figure \ref{fig:teoremageneral}):

\begin{theorem}\label{thm:main}
Take  $y_0 >0$ small enough. Fix $p\ge 1$, $p\in \NN$, and consider the regularized vector field $Z_\eps$ in \eqref{regularizedvf} with $\varphi$ a $\CC^{p-1}$ function as in \eqref{caslineal} or \eqref{diferenciable}. Fix $0<\lambda<\frac{p}{2p-1}$.

There exist $\eps _0 >0$, $L^-<0$, and $ \alpha (\eps)= x_0^-+\alpha ^- \eps + \beta ^- \eps ^{2\lambda}+ O(\eps^{\lambda +1})$, where $\alpha ^-$, $\beta^-$ are the constants given in Proposition \ref{prop:flowtangency},
such that the map $P_\eps$ restricted to  the interval  $\II:=[L^-, \alpha (\eps)] $ verifies:
\begin{itemize}
\item
If $\varphi$ is a piecewise linear function ($p=1$):
$$
P_\eps (x)= x_0^+ + \alpha^+ \eps  +O(\eps ^2), \ \forall x\in \II
$$
\item
If $\varphi$ is of class $\CC^{p-1}$ ($p\ge 2$):
$$
P_\eps (x)
=x_0^+ +\alpha ^+ \eps +\beta^+ (\eta (0))^2\eps ^{\frac{2p}{2p-1}} 
+ \OO(\eps^{\frac{3p-1}{2p-1}}, \eps^{\frac{p(p+1)}{(2p-1)^2}}), \ \forall x\in \II,
$$
where $\eta (u)$ is the unique  solution of equation: 
\begin{eqnarray}\label{eq:edoeta0pteorema}
\frac{d \eta}{d u}= \frac{2}{4\eta  -\frac{\varphi^{(p)}(1)}{p!}u^p}
\end{eqnarray}
satisfying $\eta  (u)-\frac{\varphi^{(p)}(1)}{4 p!}u^p \to 0$ as $u\to -\infty$. Here we denote as
$$
\varphi^{(p)}(1):= \lim _{v\to 1^-}\varphi^{(p)}(v).
$$
\end{itemize}
\end{theorem}

\begin{figure}
\begin{center}
\includegraphics[width=10cm]{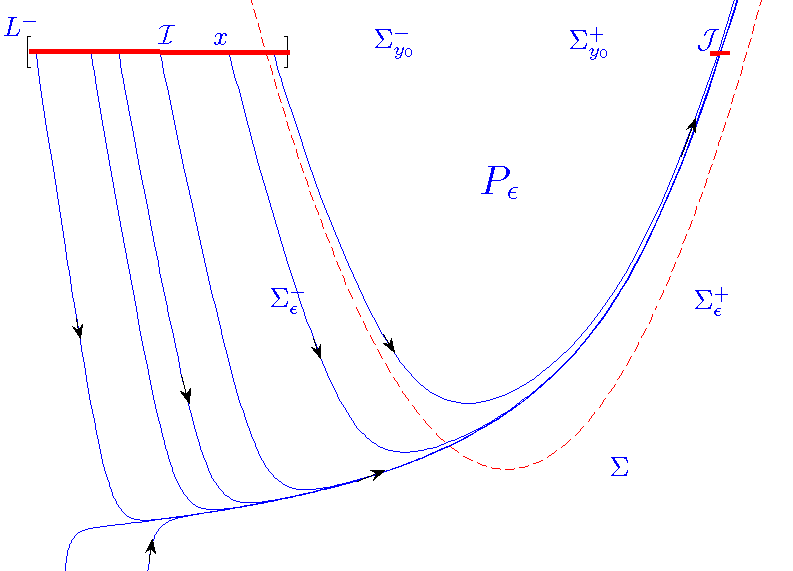}
\end{center}
\caption{Dynamics of the  Poincar\'{e} map $P_\eps$ for the regularized  system $Z_\eps$.}
\label{fig:teoremageneral}
\end{figure}

\subsection{Global results: existence of periodic orbits}
Now suppose that the upper vector field $\X$ has a global recurrence in such a way 
that there exists a {\emph{exterior}} Poincar\'{e} map:
\begin{equation}
\begin{array}{rcl}
P^e : \Sigma ^+_{y_0} &\to& \Sigma ^-_{y_0}\\
(x,y_0)&\mapsto &(P^e (x),y_0)
\end{array}
\end{equation}
which is smooth, 
and denote by:
\begin{equation}\label{eq:externalmap}
P^e(x_0^+) = x_0^-+\gamma, \quad
\frac{d P^e}{dx}(x^+_0) =c\le 0,
\end{equation}
where we remind that $x_0^\pm = W^{u,s}_+(0,0)\cap \Sigma _{y_0}^\pm$.
We compose this external map with the Poincar\'{e} map $P_\eps = \bar P\circ \PP_\eps \circ \bar P$ studied in Theorem \ref{thm:main}.

Next theorem gives conditions to ensure the existence of fixed points of the return Poincar\'{e} map $P^e\circ P_\eps$, which give rise to periodic orbits for the regularized system $Z_\eps$.

\begin{theorem}\label{thm:po}
Consider the map $P^e\circ P_\eps$ restricted to the interval $\II$ given in Theorem \ref{thm:main}. 
Let $c$ and $\gamma$ the constants given in \eqref{eq:externalmap}, and let us call $\Delta = \alpha _- - c \alpha _+$, where  $\alpha ^\pm$ are the constants given in Proposition \ref{prop:flowtangency}. Then, one has:
\begin{itemize}
\item
If $\gamma >0$, or if $\gamma=0$ and $ \Delta <0$, then, for $0<\eps <\eps_0$,
$$
P^e\circ P_\eps(\II)\cap \II=\emptyset
$$
and therefore  $P^e\circ P_\eps$ has no fixed points in the interval  
$\II$.
\item
If 
$\gamma <0 $,   or if $\gamma=0$ and $ \Delta >0$,
the map 
$P^e\circ P_\eps$ is a contraction  in $\II$ for $0<\eps <\eps_0$ and therefore it has a fixed point in this interval.

Let us call $\Gamma _\eps$ the corresponding periodic orbit of the regularized system $Z_\eps$.
\begin{itemize}
\item
If $\gamma<0$ the periodic orbit $\Gamma_\eps$ approaches, as $\eps \to 0$, to the sliding  cycle  $\Gamma_0$ of the Filippov system $Z$ given by
$
\Gamma _0=W^u_+(0,0) \cup \{ (x,0),\ x^*\le x \le 0\}$, where $(x^*,0)= W^u_+(0,0)\cap \Sigma$.
\item
If $\gamma =0$ and $\Delta >0$, the periodic orbit $\Gamma_\eps$ approaches, as $\eps \to 0$, to a grazing  periodic orbit $\Gamma_0$ of the Filippov system $Z$ given by $\Gamma _0=W^u_+(0,0)=W^s_+(0,0)$, which is a hyperbolic attracting periodic orbit of the vector field $\X$.
\end{itemize}
\item
The limit $\Gamma_\eps \to \Gamma _0$ is not uniform in the following sense:
\begin{itemize}
\item
In the region $(x,y) \in \VV^+$, $y>\eps$, one has that $\Gamma_\eps $ is $\eps$-close to $\Gamma_0$.
\item
If we call   $(\gamma_0^\eps ,\eps)=\Gamma_0 \cap \{ (x,\eps), \ x>0\}$, and 
$(\gamma_\eps^\eps ,\eps)=\Gamma_\eps \cap \{ (x,\eps), \ x>0\}$, one has that
$$
\gamma_\eps ^\eps = \OO(\eps ^{\frac{p}{2p-1}}), \quad \gamma_0 ^\eps =\OO(\eps ^{\frac{1}{2}}).
$$
\end{itemize} 
\end{itemize}
\end{theorem}

\begin{proof}

We  look for fixed points of the Poincar\'{e} map 
$P^e\circ P_\eps$. 
By Theorem \ref{thm:main}, all the points in the interval
$\II$ 
are send by $P_\eps$ to an interval $\JJ$ of size, at most, $\OO(\eps^{\frac{3p-1}{2p-1}}, \eps^{\frac{p(p+1)}{(2p-1)^2}})$ centered at the point
$x=x^+_0 +\alpha ^+ \eps +\beta^+ (\eta _p(0))^2\eps ^{2p/(2p-1)}$.

The map $P^e$ sends this point to:
\begin{eqnarray*}
&&P^e(x^+_0 +\alpha ^+ \eps +\beta^+ (\eta _p(0))^2\eps ^{2p/(2p-1)})\\
&=&x_0^- + \gamma + c (\alpha ^+ \eps +\beta^+ (\eta _p(0))^2\eps ^{2p/(2p-1)}) + 
\OO(  \eps ^2)\\
&=&
x_0^- + \gamma + c \alpha ^+ \eps + \OO(\eps ^{2p/(2p-1)}).
\end{eqnarray*}
Summarizing,  $P^e\circ P_\eps$ sends the whole interval 
${\cal I}=[ L^-,x^-_0+\alpha ^- \eps + \beta ^- \eps ^{2\la}+\OO( \eps ^{1+\la})] $ to an interval  $J$ of size, 
at most, $\OO(\eps ^{2p/(2p-1)})$ centered at the point
$x^-_0+\gamma +c \alpha ^+ \eps $. Therefore $P^e\circ P_\eps$ is a Lipchitz map with Lipchitz constant of order, at most, $\OO(\eps ^{2p/(2p-1)})$.

A sufficient condition to ensure that $J\subset {\cal I}$ and therefore that $P^e\circ P_\eps$ is a contraction, is that   
$x^-_0+\gamma +c \alpha ^+ \eps  <x^-_0+\alpha ^- \eps $. Let us call $\Delta = \alpha _- - c \alpha _+$.  
This condition is  verified if 
\begin{equation}\label{condper}
\gamma <\Delta \eps.
\end{equation}

Assume $\gamma>0$. 

In this case, taking $\eps>0$ small enough, if $\Delta \ge 0$, we can ensure that $0<\Delta \eps <\gamma$ and if $\Delta <0$, one has  $\Delta \eps<0 <\gamma$ for any positive $\eps$. Therefore in any case one has 
$$
\Delta \eps <\gamma
$$ 
which implies that   
$P^e\circ P_\eps (\II)\cap \II =\emptyset$. The same happens for $\ga=0$ if $\Delta <0$.

Assume $\gamma<0$.

If $\Delta \ge 0$  condition  \eqref{condper} is verified for any positive $\eps$. 
If $\Delta <0$ then taking $0<\eps <\frac{\gamma}{\Delta}$  condition \eqref{condper} is also verified.
Then, If $\gamma <0$, taking $\eps $ small enough one can ensure that  $P^e\circ P_\eps ({\mathcal I}) \subset J \subset {\mathcal I}$
and then  the map $P^e\circ P_\eps$  is a contraction. 
Consequently, there is a unique fixed point in the interval $J\subset \II$ which gives rise to a periodic orbit $\Gamma_\eps$.
Observe that the non-smooth system $Z$ has, in this case, a sliding cycle $\Gamma_0 = W^u_+(0,0) \cup \{ (x,0), \ x^*\le x\le 0\}$, 
where 
$(x^*,0)=W^u_+(0,0)\cap \Sigma$. 
Clearly, $\Gamma_\eps $ is  $\eps$-close to $\Gamma_0$ in the region $\{(x,y),\ y>\eps\}$.

Analogously, if $\gamma=0$, then one can ensure that condition \eqref{condper} is verified if  $\Delta >0$.
Observe that, in this case, $\Gamma_0 = W^u_+(0,0)$ is a grazing periodic orbit of $\X$ and is $\eps$-close to $\Gamma_\eps$.

To finish the proof let us observe that, on one hand, $\Gamma_0\cap \{(x,\eps),\ x>0\} = (\gamma_0^\eps, \eps)$ with 
$\gamma_0^\eps =\OO(\sqrt{\eps})$.

On the other hand $\Gamma_\eps\cap \{(x,\eps),\ x>0\} = (\gamma^\eps_\eps, \eps)$
and, using \eqref{eq:ppg}:
$$
\gamma^\eps_\eps= \bar P^{-1}(x_0^+ +\alpha ^+ \eps +\beta^+ (\eta(0))^2\eps ^{\frac{2p}{2p-1}} 
+ \OO(\eps^{\frac{3p-1}{2p-1}}, \eps^{\frac{p(p+1)}{(2p-1)^2}})) = \eta(0)\eps ^{\frac{p}{2p-1}} (1+o(1)) .
$$
\end{proof}

\begin{remark}\label{rem:delta}
To give a geometrical interpretation of the condition $\Delta >0$ let us observe the following.
We are assuming that $W^s_+(0,0)\cap \Sigma _{y_0}^-=(x_0^-,y_0)$, but also
 $W^u_+(0,0)\cap \Sigma _{y_0}^-=(x_0^-+\gamma,y_0)$. Therefore, if we consider  the Poincar\'{e} return map
associated to the regular vector field $\X$:
 $$
 \pi ^+:\Sigma_{y_0}^- \to \Sigma_{y_0}^-
 $$
and one has that $\pi ^+(x_0^-)=x^-_0+\gamma$. 
 
 Clearly, the case $\gamma=0$ corresponds to the case that the vector field $\X$ has a grazing periodic orbit. This orbit  is hyperbolic attracting when 
 $|(\pi^+)'(x_0^-)|<1$ and repelling when $|(\pi^+)'(x_0^-)|>1$. 
 
Let us point our that, by \eqref{xepsilon}, we know that  the point $(x_\eps,\eps)$ where the vector field $\X$ is tangent to $\Sigma _\eps$ verifies
$$
\bar x_\eps=P^{-1}(x_\eps) = x_0^-+\alpha^- \eps + \OO(\eps ^2)
$$
but the orbit of this point for the vector field $Z_\eps$ coincides with the orbit  given by  the vector field $\X$, therefore,
one  has that 
$$
\pi ^+(\bar x_\eps)= P^e(P_\eps(\bar x_\eps)).
$$
Now, we compute:
$$
P^e(P_\eps(\bar x_\eps))=P^e (\bar P(x_\eps))=P^e(x_0^+ +\alpha ^+ \eps + \OO(\eps ^2))= x_0^- +\gamma + c \alpha ^+ \eps + \OO(\eps ^ 2).
$$ 
If we Taylor expand the map $\pi^+$ around $x_0^-$:
$$
\pi ^+(\bar x_\eps)= \pi ^+ (x_0^-) + (\pi^+)'(x_0^-) (\bar x_\eps-x_0^-) 
+\OO(\bar x_\eps-x_0^-)^2= x_0^- +\gamma+ (\pi^+)'(x_0^-) \alpha ^- \eps + \OO(\eps ^2)
$$
and then we obtain:
$$
c \alpha ^+ = (\pi^+)'(x^-_0) \alpha^-
$$
therefore, $\Delta = \alpha ^-(1-(\pi^+)'(x_0^-))$, and 
the condition $\Delta >0$ is equivalent to $0<(\pi^+)'(x_0^-)<1$.
In the case $\gamma =0$ this condition is equivalent to ask that the periodic orbit of $\X$ is a hyperbolic attracting periodic orbit.

\end{remark}

In view of Remark \ref{rem:delta}, Theorem \ref{thm:main} and proposition \ref{prop:flowtangency} do not enable us to analyze the persistence of periodic orbits of the regularized system in the case that $\X$ has a centre. This is done in next Theorem \ref{thm:centre}. Previously, in next Proposition \ref{prop:comparacio}, we give some relations 
between the map $P_\eps$ and $P^+$,   the Poincar\'{e} map associated to the vector field $\X$ as a regular vector field in $\VV^+\cup \VV^-$:
\begin{equation}\label{eq:poincareX}
P^+: D^+\times\{y_0\}\subset \Sigma _{y_0}^-\to \Sigma _{y_0}^+.
\end{equation}
Clearly, there exists   a suitable constant $k <x_0^-$, which depends of the global properties of $\X$, such that $[k, \bar x_\eps ]\subset D^+$.
\begin{proposition}\label{prop:comparacio}
Let be $(x_0^-, y_0)=W^{s}_+(0,0)\cap \Sigma ^-_{y_0} $ and $\bar x_\eps$ given in \eqref{eq:tangent} and \eqref{eq:tangentbar}. 
Then, for any $x\in [x^-_0, \bar x_\eps]$ one has that 
$$
P_\eps (x)< P^+(x).
$$
\end{proposition}
\begin{proof}
As the vector fields $Z_\eps$ and $\X$ are the same in the region $y\ge \eps$ we will take the initial condition at $(x,\eps)$ for $x\in [x^-_\eps,  x_\eps]$
where $(x^-_\eps,\eps)=W^s_+(0,0)\cap \Sigma_\eps ^-$.

Consider the flow 
$\phi_{\X}(t;x,\eps)$. As the vector field $\X$ points down in $\Sigma _\eps ^-$ and the orbits can not cross the pseudoseparatrix of the fold point, the orbits remain in the region $\{(x,y), \ 0\le y\le \eps\}$ until they cross $\Sigma _\eps ^+$. Also, taking $\eps$ small enough, one can assume that \eqref{positivefilipov} is satisfied in this region.

Denote by $ (x(t),y(t))= \phi_{\X}(t;x,\eps)$ and by 
$$
\X_N= \X_N(x(t),y(t))= (\dot y(t),-\dot x(t))=\left(\X_2(x(t),y(t)), -\X_1(x(t),y(t))\right)
$$ 
the normal exterior vector to the orbit.
Then, we perform the scalar product:
\begin{eqnarray*}
<\X_N,Z_\eps>(x(t),y(t))&=&\X_2\left(\frac{\X_1+\Y_1}{2}+\varphi(\frac{y(t)}{\eps})\frac{\X_1-\Y_1}{2}\right)(x(t),y(t))\\
&-&
\X_1\left(\frac{\X_2+\Y_2}{2}+\varphi(\frac{y(t)}{\eps})\frac{\X_2-\Y_2}{2}\right)(x(t),y(t))\\
&=& \left(\frac{1+\varphi(\frac{y(t)}{\eps})}{2}\right)\left(\X_2\Y_1-\X_1\Y_2\right)(x(t),y(t))<0
\end{eqnarray*}
Then, as both vector fields are smooth and, except at $(x_\eps,\eps)$, they are not tangent to $\Sigma_\eps=\Sigma_\eps^-\cup\Sigma_\eps^+$, the orbit of   $\X$ strictly bounds $Z_\eps$ from bellow and therefore, if we denote by $t_1$ and $t_2$ the times when $\pi_y (\phi_{\X}(t_1;x,\eps))=\pi_y (\phi_{Z_\eps}(t_2;x,\eps))=\eps$, one has that
$\pi_x (\phi_{\X}(t_1;x,\eps))>\pi_x (\phi_{Z_\eps}(t_2;x,\eps))$.
\end{proof}

\begin{figure}
\includegraphics[width=13cm]{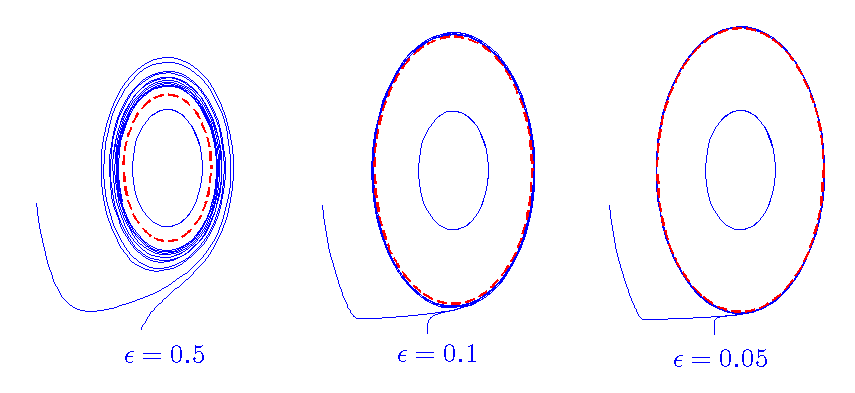}
\caption{behavior of the regularized system
in the case $\X$ has a center, for different values of the regularizing parameter $\eps$.}
\label{fig:center}
\end{figure}

\begin{theorem}\label{thm:centre}
Suppose that $\X$ has a center in $\VV^+$ surrounded by periodic orbits which intersect the switching surface $\Sigma$.

Then, for $\eps$ small enough the unique tangent orbit to $\Sigma _\eps= \Sigma _\eps^+\cup\Sigma _\eps^-$ of $\X$ is a periodic orbit of $Z_\eps$ that is semistable: it is attracting for all the orbits exterior to it but its interior is foliated by periodic orbits.
\end{theorem}

\begin{proof}
Consider the Poincar\'{e} map $P_\eps = \bar P\circ \PP_\eps\circ P$, and the return map $P^e\circ P_\eps$. It is clear that $P^e\circ P_\eps(\bar x_\eps)=\bar x_\eps$, where $\bar x_\eps$ is defined in \eqref{eq:tangentbar}, because the orbit through $\bar x_\eps$ is tangent to $\Sigma_\eps$, and therefore, being a periodic orbit of $\X$, is also a periodic orbit of $Z_\eps$. 

It is also important to note that $P^e\circ P^+=\pi ^+$, where $P^+$ is given in \eqref{eq:poincareX}, and we know that, as $\X$ has a center in $\VV^+$,  $\pi ^+(x)=x$ for all the points in its domain.

Now, as $\bar x_\eps\in \II$, given in Theorem \ref{thm:main}, if we take $x<\bar x_\eps$ one has, by proposition \ref{prop:comparacio}, that $P^+(x)>P_\eps (x)$, and therefore, as $P^e$ is decreasing (orbits in the plane can not intersect)
$$
x= P^e\circ P^+(x) < P^e\circ P_\eps (x)
$$
which gives that  $(P^e\circ P_\eps)^n (x)$ forms a strictly increasing sequence whose limit is the fixed point $\bar x_\eps$.
\end{proof}

\subsection{The grazing-sliding  bifurcation of periodic orbits}

Let us now consider some classical bifurcations of periodic orbits in non-smooth systems and see how they behave after the regularization.

Consider a family $Z_\mu$ of non-smooth planar systems such that they undergo a grazing sliding bifurcation of a hyperbolic attracting or repelling periodic orbit of the vector field $\X_\mu$ at $\mu=0$. 
Next theorem  shows how these bifurcations behave in  the corresponding regularized family $Z_{\mu,\eps}$.

\begin{theorem}\label{thm:sella-node}
Let  $Z_\mu$, $\mu \in \RR$ be a family of non-smooth planar systems that undergoes a grazing sliding bifurcation of a hyperbolic periodic orbit $\Gamma_\mu$ of the vector field $\X_\mu$ at $\mu=0$. We assume that, for $\mu>0$ the periodic orbit  $\Gamma_\mu$  is entirely contained in $\VV^+$, it becomes tangent to $\Sigma$ for $\mu=0$ and intersects both regions $\VV^\pm$ for $\mu<0$.
 
Consider the regularized family $Z_{\mu,\eps}$.
\begin{itemize}
\item
If $\Gamma_\mu$ is attracting, the regularized system has a  periodic orbit $\Gamma_{\mu,\eps}$ for any $\eps$, $\mu$ small enough. No bifurcation exists in the regularized system.
\item
If $\Gamma_\mu$ is repelling, the regularized system has a  periodic orbit $\Gamma_{\mu,\eps}$ for any  $\mu>0$ and $0<\eps<\eps _0(\mu)$ which coexists with the periodic orbit $\Gamma_\mu$ contained in $\VV^+\cap \{ (x,y), \ y>\eps\}$. 
For $\mu \le  0$ small enough, the system has no periodic orbits near $\Gamma_0$ if $\eps$ is small enough. 
Therefore the family $Z_{\mu\eps}$ undergoes a saddle node bifurcation of periodic orbits at $\mu=0$.
\end{itemize}
\end{theorem}

\begin{figure}
\begin{center}
\includegraphics[width=13cm]{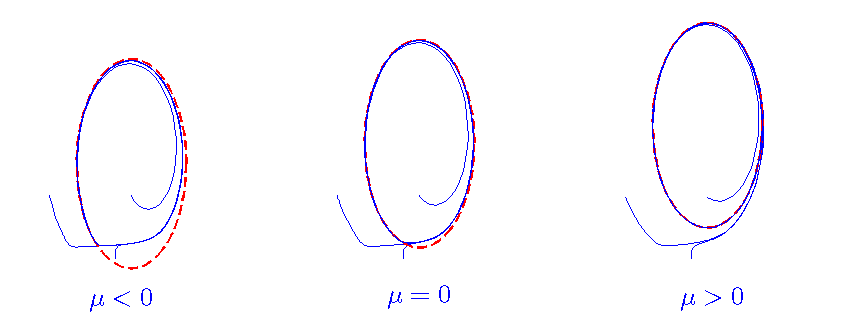}
\end{center}
\caption{No bifurcation of periodic orbits in  the regularized system corresponding with the grazing-sliding 
bifurcation in the Filipov system: case of a atracting periodic orbit.}
\label{grazingatractor}
\end{figure}

\begin{figure}
\begin{center}
\includegraphics[width=13cm]{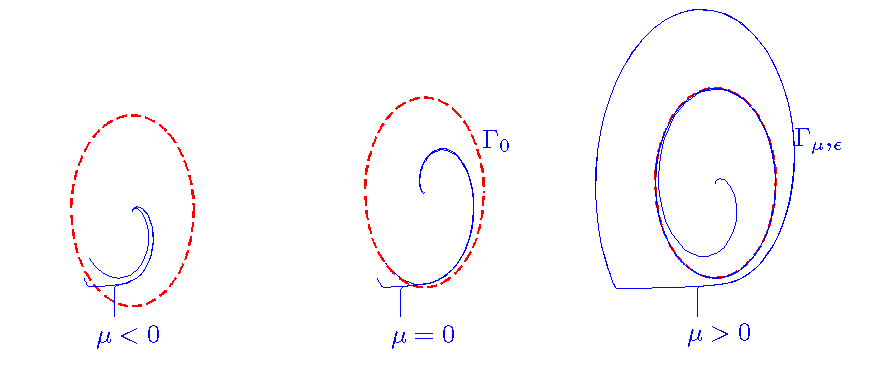}
\end{center}
\caption{Saddle-node bifurcation in the regularized system
corresponding with the grazing-sliding bifurcation in the Filipov system: case of a repelling periodic orbit.}
\label{grazingrepellor}
\end{figure}

\begin{proof}
One can assume that the fold point, which exists for $\mu$ small enough, is located at $(0,0)$.
As usual, we denote by $(x_0^\pm,y_0)= W^{u,s}(0,0)\cap \Sigma _{y_0}^\pm $, the intersection of its stable and unstable pseudo-separatrices with $\Sigma_{y_0}^\pm$ and we also assume that $x_0^\pm$ are independent of $\mu$.

Assume that the periodic orbit $\Gamma_\mu$ of the vector field $\X_\mu$ is attracting. 
In this case, for $\mu >0$, $\Gamma_\mu$ which is contained in $\VV ^+$, it becomes  tangent to $\Sigma$ for  $\mu =0$, and  then crosses $\Sigma $ for $\mu<0$ but, being $\Gamma_\mu$ attracting, a sliding cycle $\tilde \Gamma_\mu$ for the non-smooth system $Z$ appears. Observe that $\tilde \Gamma _0=\Gamma_0$.

Our external map $P^e$ satisfies, for $\mu=0$,  $P^e(x_0^+)=x_0^-$, and we can assume, without loss of generality that for $\mu$ small enough
the map $P^e$ is defined and verifies $P^e(x_0^+)= x_0^-+\mu$. 

By Theorem \ref{thm:po}, using $\gamma=\mu$,  we know that, for $\mu<0$, system $Z_{\mu,\eps}$ has a periodic orbit for $\eps$ small enough.
The result is also true for $\mu=0$ because, as $\Gamma_\mu$ is attracting, we have  by Remark 
\ref{rem:delta} that $\Delta >0$.
For $\mu >0$ we observe that, in the proof of Theorem \ref{thm:po}, the condition required to the existence of  a periodic orbit of $Z_\eps$ is 
\eqref{condper}, therefore, as  $\Delta >0$, if we write $\mu =\tilde \mu \eps$, condition \eqref{condper} is verified until 
$$
\tilde \mu <\Delta
$$
and therefore the periodic orbit which existed for $\mu <0$ persists for these values of $0\le \mu\le \Delta \eps$ if $\eps $ is small enough.

The case $\tilde \mu=\Delta$ corresponds, in first order, to the value of the parameter $\mu$ where the periodic orbit $\Gamma_\mu$ of the vector field $\X_\mu$ is entirely contained in the region 
$\{y> \eps \}$ and therefore, $\Gamma_\mu$ is also a  periodic orbit of $Z_{\mu,\eps}$, 
because $Z_{\mu,\eps}=\X$ in this region.

In fact, if we consider the return Poincar\'{e} map $\pi^+$  in $\Sigma_{y_0}^-$ associated to the vector field $\X_\mu$, and we denote by $(x_\mu,y_0)=(x_0^-+\mu^*,y_0)$ the intersection of  the periodic orbit $\Gamma_\mu$ of $\X_\mu $ with $\Sigma_{y_0}^-$, one has:
$$
x_0^-+\mu^*=\pi^+( x_0^-+\mu^*)=x_0^-+\mu + \pi'(x_0^-) \mu^* + O(\mu^*)^2,
$$ 
which gives   $\mu^* =\frac{\mu}{1-(\pi^+)'(x_0^-)}+ \OO(\mu^2)= \frac{\tilde \mu \eps }{1-(\pi^+)'(x_0^-)} + \OO(\eps^2)$.

Then, for $\mu =\tilde \mu \eps$, the periodic orbit $\Gamma _\mu$ of $\X_\mu$ intersects $\Sigma _{y_0}^-$ in a point
$$
x_0^-+\frac{\tilde \mu \eps }{1-(\pi^+)'(x_0^-)} + \OO(\eps ^2)
$$ 
and, by \eqref{xepsilon},  this point belongs to the interval $[K^-,P^{-1}(x_\eps)]=[K^-,\bar x_\eps]$ if
$$
\frac{\tilde \mu \eps }{1-(\pi^+)'(x_0^-)} \le \alpha ^- \eps + \OO(\eps ^2)
$$
which gives $\tilde \mu \le  \alpha ^- (1-(\pi^+)'(x_0^-))= \Delta $. 
Therefore, for $\tilde \mu \le \Delta$  the periodic orbit $\Gamma_{\tilde \mu \eps}$ belongs to the interval affected by the regularization, but when 
$\tilde \mu >  \Delta$ the periodic orbit does not intersect the region affected by the regularization. 
Therefore the periodic orbit $\Gamma_\mu$ of the vector field $\X_\mu$ is the continuation of the periodic orbit $\Gamma_{\mu, \eps}$ of $Z_{\mu,\eps}$, for $\mu \ge \Delta \eps$.

Assume now that the periodic orbit $\Gamma_\mu$ of the vector field $\X_\mu$ is repelling. Then, by Remark \ref{rem:delta},  one has $\Delta <0$.

Again, we assume that, for $\mu<0$, $\Gamma_\mu$ crosses $\Sigma$,  becomes  tangent to $\Sigma$ for $\mu=0$ and  then is contained in $\VV ^+$ for $\mu>0$. 
Therefore, in this case, for $\mu>0$, being $\Gamma_\mu$ repelling,  we have the co-existence of this periodic orbit of $\X_\mu$ and a sliding cycle $\tilde \Gamma_\mu$ of the non-smooth system $Z_\mu$. Both collide at $\mu=0$ and then disappear.

Our external map $P^e$ satisfies, for $\mu=0$,  $P^e(x_0^+)=x_0^-$, and we can assume, without loss of generality that for $\mu$ small enough
the map $P^e$ is defined and verifies $P^e(x_0^+)= x_0^- -\mu$. 

By Theorem \ref{thm:po}, using $\gamma=-\mu$,  we know that, for $\mu>0$,  the regularized vector field $Z_{\mu,\eps}$ has a periodic orbit $\Gamma_{\mu,\eps}$ for $\eps$ small enough.

Let us observe that the periodic orbit $\Gamma_{\mu,\eps}$ intersects $\Sigma_{y_0}^-$ in a point $(x_p,y_0)$, with $x_p- x_0^-=\OO(\eps) $. But the periodic orbit $\Gamma_\mu$ of the vector field $\X_\mu$  intersects $\Sigma_{y_0}^-$ in a point $(x_\mu,y_0)$, with $x_\mu-x_0^-=\OO(\mu)$, therefore both periodic orbits coexist.

When $\mu=0$, as the tangent periodic orbit  $\Gamma_0$ is repelling, one has that $\Delta <0$, and therefore, by Theorem  \ref{thm:po}, there is no periodic orbit in the regularized system $Z_{0,\eps}$ for $\eps$ small enough.

When   $\mu =\tilde \mu \eps$, 
one has again that
$x_\mu=  x_0^- -\mu^*$, with $\mu^* =\frac{\tilde \mu \eps }{1-\pi'(x_0^-)} + \OO(\eps^2)$,
and, by \eqref{xepsilon},  this point belongs to the interval $[K^-,P^{-1}(x_\eps)]$ if
$- \tilde \mu \le  \alpha ^- (1-\pi'(x_0^-))= \Delta $. 
Therefore, for $\tilde \mu =-\Delta$  the periodic orbit enters the interval affected by the regularization and meets $\Gamma_{\mu, \eps}$. Then, at $\mu=0$, both orbits disappear. This is a saddle node bifurcation.
\end{proof}

\subsection{Application to dry friction systems in a single degree of freedom}\label{dryfriction}

Let us consider a mass $m$ attached to a spring with a constant of recovery $K$. The mass is on a moving belt with constant velocity $v_d$.

If $x$ denotes the displacement of $m$ with respect to the equilibrium position of the spring $K$, on $m$ act two forces: a force of resistance of the spring $-Kx$ (assuming the spring linear), and a friction force between the mass and the belt. 

If we start from the equilibrium position $x=0$, the mass will begin to move in stick with the belt (stick phase) at velocity $v_d$ till the recovery force of the spring $-Kx$ compensate the static friction force and produce on $m$ a damped harmonic motion (slip phase) until that, by energy dissipation, the mass will be once more in sticking with the belt, and so on.

So the equations are divided according to whether or not the relative speed between the mass and the belt, $ v_r = \dot x -v_d$, is zero in two phases: 

\begin{itemize}
\item
Stick phase ($v_r = 0$), the equations are:
$$
m \ddot x = - K x +\FF_s(x),
$$
where $\FF_s(x)=\min(|Kx|,F_s){\rm sgn}(Kx)$, is the friction static force and $F_s$ is its maximum value. 

Note that if $| Kx | <F_s$, then $\ddot x= 0$ and $\dot x= v_d$, ie, $m$ moves in sticking with the belt until the force of the spring recovery reaches $F_s$. From this moment  on, $m$ begins to oscillate on the belt. But now it enters into a state where $v_r \ne  0$ and there the frictional force depends on $v
_r$. The system is now in slip phase.

\item
Slip phase ($v_r \ne 0$),  the equations of motion are 
$$
m\ddot x = - K x + \FF_d (v_r),
$$
where $\FF_d(v_r)$, represents the dynamic friction which has opposite sign to $v_r$.
\end{itemize}

Following R.I. Leine \cite{Leine00, Leine00b}
one considers three basic models of friction related to three different types of $\FF_d(v_r)$.
\begin{itemize}
\item
\emph{Stribeck model}.
This model incorporates the experimental evidence that the force of static friction is larger than the dynamic one, although there is a continuous transition from one state to other.
\item
\emph{Coulomb model}.
This model assumes that the dynamic friction is constant and equal to the static friction.
\item
\emph{Stiction model}.
This model assumes that there is not a regular transition between static and dynamic friction, but when the spring arrives to the value of static friction, the frictional force falls instantaneously and discontinuously to a value strictly less.
Note that in this model, unlike the other two, the dynamic friction has no lateral limits, but tends to whole intervals $[F_d, F_s]$ and $[-F_s,- F_d]$, respectively.

\end{itemize}

In \cite{Leine00b}, a possible function with the characterizes the  Stribeck model, putting $v_d=m=K=1$ is formulated:
$$
\FF_d (v_r) = - (\frac{F_s-F_d}{1+\delta |v_r|}+F_d){\rm sign} (v_r), \ 0< \delta \ll 1
$$
where $v_r=\dot x -1$.

\marginpar{que es delta}

Now the stick and slip systems are:
$$
\left.
\begin{array}{rcl}
\dot x &=& y \\
\dot y &=& -  x + \min ( | x |, F_s) {\rm sgn} ( x), 
\end{array}
\right\}
y =1 \mbox (stick)
$$
and, 
$$
\left.
\begin{array}{rcl}
\dot x &=& y \\
\dot y &=& -  x -   (\frac{F_s-F_d}{1+\delta |y-1|}+F_d) {\rm sign} (y -1),
\end{array}
\right\}
y \ne 1 \mbox (slip).
$$
The slip system  can be written as a Filipov system $Z=(\X, \Y)$ with switching surface:
$$
\Sigma =\{ (x,1), x\in \RR \}
$$
$$
\X :
\left.
\begin{array}{rcl}
\dot x &=& y \\
\dot y &=& -  x -  (\frac{F_s-F_d}{1+\delta (y-1)}+F_d), 
\end{array}
\right\}
y >1 
$$
and 
$$
\Y:
\left.
\begin{array}{rcl}
\dot x &=& y \\
\dot y &=& -  x +   (\frac{F_s-F_d}{1+\delta (1-y)}+F_d) ,
\end{array}
\right\}
y < 1 
$$
The region $| x | <F_s $ in the switching surface $y=1$ is again a sliding region and  the sliding  Filippov vector field is:
$$
\dot x= 1,
$$
which coincides with the stick field.

The vector field $\X$ has a invisible fold at $(-F_s,1)$ and points toward $\Sigma$ for $x>-F_s$.

It turns out that 
$\Y$ has a repeller  focus at the  point $(\frac{F_s-F_d}{1+\delta}+F_d,0 )$ for $\delta$ small enough with eigenvalues:

$$
\Lambda_\pm =\frac{\nu}{2}  \pm \frac{i}{2}\sqrt{4-\nu^2}
$$
where $\nu=\frac{\delta (F_s-F_d)}{(1+\delta)^2}>0$, therefore this is an unstable  focus.
If we denote by $\al=F_s-F_d$ and $\beta= F_d$, we have that the function  
$$
V(x,y)=(x-\beta-\frac{\alpha}{1+\delta})^2 + y^2
$$
is strictly growing over the solutions of  $X_-$, because:
$$
\frac{d V}{dt}(x,y)=2 \alpha\delta y^2 \frac{1}{1+\delta} \frac{1}{1-\delta(y-1)}>0
$$
if  $y<\frac{1+\delta}{\delta}$. 

Note that $\Y$ has a visible tangency point at $(F_s,1)$ and its unstable pseudoseparatrix $W^u_-(F_s,1)$ intersects the switching manifold 
at a point $(x^*,1)$ between the two fold points if $\delta$ is small enough. Therefore the Stribeck model has a sliding periodic orbit:
$$
\Gamma_0= W^u_-(F_s,1) \cup \{ (x,1),\ x^*\le x \le F_s\}.
$$

We can then apply Theorem \ref{thm:po} to this system and ensure that the corresponding regularized system $Z_\eps$ has a periodic orbit
$\Gamma _\eps \to \Gamma _0$ as $\eps \to 0$ (see figure \ref{coulomb}).

If for simplicity we take $m = K = v_d = 1$, the equations of motion for the Coulomb model are:
\begin{eqnarray*} 
\ddot x &=& -  x + \min ( | x |, F_s) {\rm sign} ( x), \ v_r =0 \mbox (stick)\\
\ddot x &=& -  x -  F_s {\rm sign}v_r, \ v_r \ne 0 \mbox (slip)
\end{eqnarray*}
which give  two systems:
$$
\left.
\begin{array}{rcl}
\dot x &=& y \\
\dot y &=& -  x + \min ( | x |, F_s) {\rm sgn} ( x), 
\end{array}
\right\}
y =1 \mbox (stick)
$$
and 
$$
\left.
\begin{array}{rcl}
\dot x &=& y \\
\dot y &=& -  x -  F_s {\rm sign} (y -1),
\end{array}
\right\}
y \ne 1 (\mbox (slip).
$$
Where this last can be written as a Filipov system $Z=(\X,\Y)$:
$$
\left.
\begin{array}{rcl}
\dot x &=& y \\
\dot y &=& -  x -F_s,
\end{array}
\right\}
y \ge 1
$$
and
$$
\left.
\begin{array}{rcl}
\dot x &=& y \\
\dot y &=& -  x +  F_s,
\end{array}
\right\}
y \le 1 .
$$

We see that  the region $| x | <F_s$ in the switching surface $y=1$ is an sliding region between the two fields $\X$ and $\Y$. 
The points $(-F_s,1)$ and $(F_s,1)$ are, respectively,  invisible and visible tangency points.
We also see that in  $| x | <F_s$ the sliding  Filippov vector field is:
$$
\dot x= 1,
$$
which coincides with the stick field.
In this model, the point $(F_s,1)$ is a center surrounded by periodic orbits of the vector field $\Y$. 
Therefore, one can apply Theorem \ref{thm:centre} and we obtain, in the regularized system, a periodic orbit 
tangent to the section $y=1-\eps$ which persists in the regularized system and becomes a semi-stable periodic orbit (see figure \ref{coulomb}).

This coincidence between the stick equations and the Filipov sliding vector field does not occur in the Stiction model. 
This model has the same slip equations, and therefore gives the same non-smooth vector filed outside the switching manifold 
$y=1$, but different stick ones (see \cite{Leine00, Leine00b}). The resulting system does not follow the Filippov convention, so it is 
outside the scope of this paper. A study of different conventions and its regularizations will be the main goal of a forthcoming paper.
\begin{figure}
\includegraphics[width=13cm]{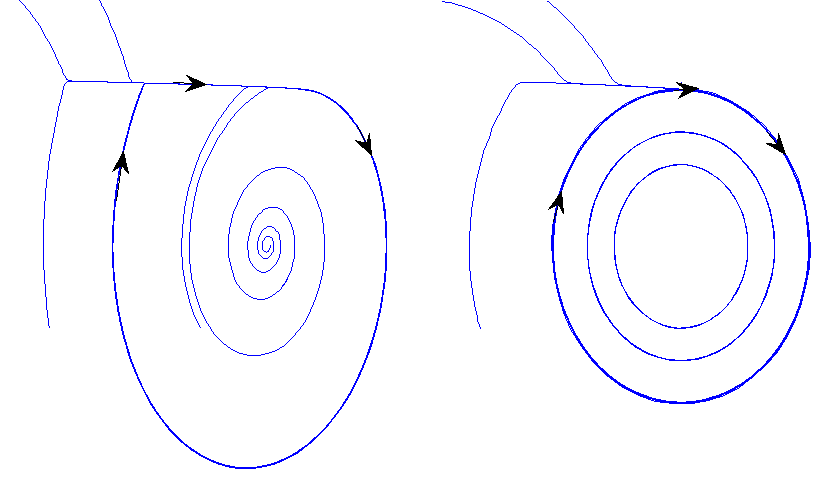}
\caption{Atracting periodic orbit (left) and semistable periodic orbit (right) corresponding to the regularization of the dry 
friction oscillator following Stribeck and Coulomb models.}
\label{coulomb}
\end{figure}

\subsection{Bifurcation of a sliding homoclinic to a saddle}

In this section we will study how the regularized vector field $Z_\mu$ behaves when the non-smooth vector field $Z$ has a 
sliding homoclinic orbit.

Let's consider the non-smooth vector field $Z$ with the same conditions  \eqref{generalform}, \eqref{cond:visiblefold} and 
\eqref{cond:visiblefold1} but now assume that the fold point $(0,0)$ has a separatrix connection with a saddle $(\x,\y)\in \VV ^+$.

Generically, this can happen in one parameter families $Z_\mu$ undergoing  a sliding homoclinic bifurcation to a saddle \cite{KutznesovRG03}.
That is, $Z_\mu$ has a saddle  $(\x,\y)$ in $\VV^+$ and, without loss of generality, we suppose independent of $\mu$.
Then, we suppose that, for $\mu <0$ both stable and unstable curves of the saddle  $W^{s,u}(\x,\y)$ intersect transversally the switching manifold $\Sigma$. 
For $\mu=0$ the unstable manifold  $W^{u}(\x,\y)$ remains transversal to $\Sigma$, but the stable $W^{s}(\x,\y)$ touches $\Sigma$ 
tangentially in a visible  fold point, that we assume at $(0,0)$, producing a pseudo-separatrix connexion between the stable manifold of the 
saddle  and the unstable pseudo-separatrix of the fold, in $\VV^+$:
$$
W^s(\x,\y)=W^u_+(0,0).
$$
For $\mu>0$ the unstable manifold of the saddle $W^{u}(\x,\y)$ remains transversal to $\Sigma$, but the stable  $W^{s}(\x,\y)$ 
moves away from $\Sigma$ inside $\VV^+$, and the unstable pseudo-separatrix of the fold does not intersect $\Sigma$ anymore. We assume, without lost of generality, as in the grazing case, that we use a coordinate system such that the fold point remains at $(0,0)$ for $\mu$ small enough.

The analysis of the regularization of this bifurcation follows closely theorems \ref{thm:po} and 
\ref{thm:sella-node}, provided we control $P^e$.        
In order to do it, suppose without  loss of generality that the eigenvalues of the saddle point $(\x,\y)$, 
$\la_2<0<\la_1$, are independent of $\mu$. 
It is well known that there exists a local change of variables, $(x,y) \to (u,v)$, in a neighborhood of  the saddle point such that, 
in the new coordinates, that we denote by
$(u,v)$, the system $\X$ reads:
\begin{equation}\label{eq:sistemasella}
\left.
\begin{array}{rcl}
\dot u &=& \la _1 u +uf(u,v)\\
\dot v &=& \la_2 v+ v g(u,v)
\end{array}
\right \}
|u| \le \delta, \quad |v|\le\delta, \quad f,g = \OO(u,v)
\end{equation}
with $\la_2<0<\la_1$.

Clearly, one has that given any $K>0$ one can choose $\delta>0$ such that $|f(u,v)| <K$ and $|g(u,v)| <K$ if $|u|<\delta$  and $|v|<\delta$.

In the  coordinates $(u,v)$ the saddle is at $(0,0)$ and the stable and unstable manifolds are given, respectively, by $u=0$, and $v=0$.
Moreover, in  these coordinates, we have the following lemma, whose proof is an straightforward application of Gronwall Lemma to system 
\eqref{eq:sistemasella}.

\begin{lemma}\label{lem:mapasella}
Let $K>0$ such that $\la_2+K<0$. 
Then, there exists $\delta= \delta (K)>0$ small enough such that:
given any solution of system \eqref{eq:sistemasella} with initial conditions $(u_0,-\delta)$, 
with $u_0\in [-\delta,0)$, there exists $T\ge 0$ such that 
$u(T)=-\delta$, moreover, 
$$
|v(T)|\le \delta ^{1+\frac{\la_2+K}{\la_1+K}} |u_0|^{-\frac{\la_2+K}{\la_1+K}}.
$$
\end{lemma}

Now, we can study the regularization $Z_{\mu,\eps}$ of the sliding homoclinic to a saddle bifurcation in $Z_\mu$.

To relate this situation with the grazing bifurcation in Theorem \ref{thm:sella-node}, 
we consider the points $x^\pm_0=W^{u,s}_+(0,0) \cap \Sigma ^{\pm}_{y_0}$, and we obtain that the external map behaves as:

\begin{lemma}\label{lem:gammabeta}
Denote by $x_0^u=W^u(\x,\y)\cap \Sigma_{y_0}^-$ and by   $x_0^\pm = W^{s,u}(0,0) \cap \Sigma _{y_0}^\pm$ .
Assume that $x_0^u$, $x_0^\pm$ are independent of $\mu$ and that $x_0^u<x_0^-$. Denote by 
$x_0^s=W^s(\x,\y)\cap \Sigma_{y_0}^+$ and assume $x_0^s= x_0^+ -\mu$,
with  $\mu$ small enough. 

Consider the exterior  map 
$$
P^e:D^e\times \{y_0\} \subset \Sigma_{y_0}^+\to \Sigma_{y_0}^-
$$
\begin{itemize}
\item
If $\mu <0$, one has that $x_0^+\in D^e$ and $P^e(x_0^+)=x_0^- +\gamma$ with $\gamma = \gamma (\mu)<0$.
\item
If $\mu >0$,  the exterior map is not defined in $x_0^+$.
\end{itemize}
\end{lemma}

\begin{proof}
The exterior map follows the orbits which pass close to the saddle $(\x,\y)$, therefore, using 
Lemma \ref{lem:mapasella}, we have that the map 
$$
P^e:D^e \times\{y_0\}\subset \Sigma_{y_0}^+\to \Sigma_{y_0}^-
$$
verifies that if   $ x\in D^e$, then  $x<x_0^s=x_0^+-\mu$ and:
\begin{equation}\label{eq:poincaresella}
|P^e(x)-x_0^u| \le \D |x-x_0^++\mu|^{-\frac{\la_2 +K}{\la_1+K}}
\end{equation}
where $\D >0$ is a suitable constant independent of $\mu$.
In particular, if $\mu <0$ small enough, we have that $x_0^+ < x_0^s=x_0^+-\mu$, and therefore:
$$
P^e(x_0^+)-x_0^u=\OO( |\mu|^{-\frac{\la_2 +K}{\la_1+K}}).
$$
Now we have:
$$
P^e(x_0^+)-x_0^- =P^e (x_0^+)-x_0^u + x_0^u-x_0^-= x_0^u-x_0^-+\OO( |\mu|^{-\frac{\la_2 +K}{\la_1+K}}).
$$
Now, as $-\frac{\la_2 +K}{\la_1+K}>0$, if we take  $\mu$  small enough, using that $x_0^u-x_0^- <0$, one has that
$$
P^e(x_0^+)-x_0^- <0,
$$
therefore $P^e(x_0^+)=x_0^-+\gamma$, with $\gamma= \ga (\mu)<0$.

In the case $\mu >0$, one has that $x_0^+>x_0^s = x_0^+-\mu$. 
As the unstable pseudo-separatrix of the fold can not intersect the stable manifold on the saddle, it can not  intersect again the section 
$\Sigma^-_{y_0}$.
\end{proof}

Now, we can give the result about periodic orbits in the regularized system.

\begin{theorem}\label{thm:separatrix}
Let $Z_\mu = (\X _\mu , \Y _\mu)$ be a family of non-smooth vector fields that undergoes a sliding homoclinic bifurcation  
generated by a generic tangency between the stable manifold of a saddle point $(\x,\y)\in \VV^+$ of $\X$ and the switching manifold $\Sigma$, 
which occurs for $\mu =0$, while the unstable manifold of the saddle is transversal to $\Sigma$.
Assume that for $\mu <0$ both stable and unstable curves of the saddle $W^{u,s}(\x,\y)$ intersect transversally the switching manifold 
$\Sigma$ and for $\mu>0$ the unstable manifold of the saddle $W^{u}(\x,\y)$ remains transversal to $\Sigma$, but the stable $W^{s}(\x,\y)$ 
moves away from $\Sigma$ inside $\VV^+$, creating a visible fold point (of $\X$) whose  unstable pseudo-separatrix in $\VV^+$ 
does not intersect $\Sigma$ anymore.
Assume also that $\Y_\mu$ is transversal to $\Sigma$ and points towards $\Sigma$ for any $\mu$ small enough.
Consider the regularized family $Z_{\mu, \eps}$, $\eps>0$, then:
\begin{itemize}
\item
If $\mu < 0$, the non-smooth system $Z_\mu$ has an sliding periodic orbit $\Gamma_\mu$, and 
the regularized system $Z_{\mu,\eps}$ has an attracting periodic orbit $\Gamma_{\mu,\eps}$ for $\eps$ small enough uniformly in $\mu$ which approaches, when $\eps \to 0$ the sliding periodic orbit $\Gamma_\mu$.
\item
If $\mu = 0$, the system $Z_\mu$ has an sliding homoclinic  orbit $\Gamma_0$, and 
the regularized system has an attracting periodic orbit $\Gamma_{\mu,\eps}$ for $\eps$ small enough uniformly in $\mu$ which approaches, when $\eps \to 0$, the sliding homoclinic orbit $\Gamma_0$.
\item
If $\mu >0$, for $\eps$  small enough, the vector field $Z_{\mu,\eps}$ has no periodic orbits in a region close to the stable separatrix of the saddle point.
\item
If $\mu= \tilde \mu \eps$, with $0\le \tilde \mu <-\alpha ^+$, where $\alpha ^+$ is given in proposition \ref{prop:flowtangency},
the family $Z_{\tilde \mu \eps,\eps}$ has an attracting periodic orbit for $\eps$ small enough which becomes an homoclinic orbit to $(\x,\y)$ for $\tilde \mu = -\alpha ^+ + \OO(\eps)$.
\end{itemize}
\end{theorem}
\begin{figure}
\includegraphics[width=13cm]{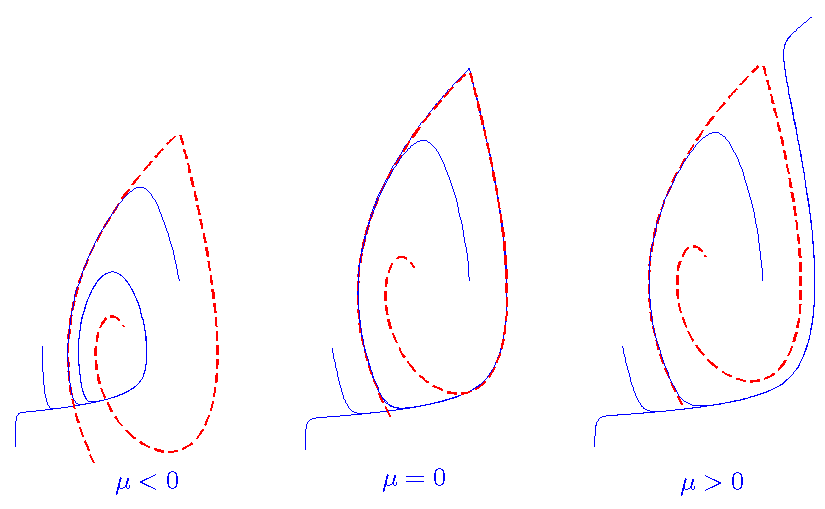}
\caption{Homoclinic bifurcation in the regularized system corresponding with the sliding homoclinic bifurcation in the Filipov system.}
\label{homoclinic}
\end{figure}

\begin{proof}
One can assume that the fold point, which exists for $\mu$ small enough, is located at $(0,0)$. 
As usual, we denote by  $(x_0^{\pm}, y_0)=W^{u,s} (0,0) \cap \Sigma _{y_0}^{\pm}$, the intersection of its stable and unstable 
pseudo-separatrices with  $\Sigma _{y_0}^\pm$ and by 
$W^{s,u} (\x,\y) \cap \Sigma _{y_0}^\pm=(x_0^{s,u}, y_0)$. We also assume that $x_0^\pm$ and $x_0^u$ are independent of 
$\mu$ with $x_0^u<x_0^-$, and that $x_0^s=x_0^+ -\mu$.

For $\mu <0$, we can apply Lemma \ref{lem:gammabeta} and we obtain that  $P^e(x_0^+)=x_0^- +\gamma$ with $\gamma <0$. 
This implies that the non-smooth vector field $Z_\mu$ has a sliding cycle $\Gamma_\mu  = W^{u}_+(0,0) \cup \{ (x,0), \ x^*\le x \le 0\}$, where 
$(x^*,0)= W^{u}_+(0,0)\cap \Sigma$.

As $\ga<0$, one can apply Theorem \ref{thm:po} and we obtain the existence of a periodic orbit $\Gamma_{\mu, \eps}$ for $Z_{\mu,\eps}$.

In the case $\mu =0$ one can not directly apply Theorem \ref{thm:po} because $x_0^+=x_0^s$ and therefore the 
exterior Poincar\'{e} map $P^e$ is not regular at $x=x_0^+$.
Then, we use the results about $P_\eps$ in Theorem \ref{thm:main}, and we have that, on one hand 
$x_0^u \in \II=[ L^-,x^-_0+\alpha ^- \eps + \mu ^- \eps ^{2\la}+O( \eps ^{1+\la})] $, 
and on the other hand 
the return map $P^e\circ P_\eps$ is defined in  the interval 
$\II$ and, if $x\in \cal I$:
$$
P^e \circ P_\eps (x) = P^e(x^+_0 +\alpha ^+ \eps + \OO(\eps ^{\frac{2p}{2p-1}}))
$$
and then, as $\alpha ^+ <0$, one has that $x^+_0 +\alpha ^+ \eps + \OO(\eps ^{\frac{2p}{2p-1}})< x_0^+=x_0^s$ and, 
applying inequality \eqref{eq:poincaresella} for $\mu =0$:
$$
|P^e(x^+_0+\alpha ^+\eps + \OO(\eps ^{\frac{2p}{2p-1}}))-x_0^u|
< \D (\alpha ^+\eps + \OO(\eps ^{\frac{2p}{2p-1}}) )^{-\frac{\la_2+K}{\la_1+K}}
\le \bar \D \eps^{-\frac{\la_2+K}{\la_1+K}}.
$$
Therefore, 
we can ensure that, as  $-\frac{\la_2+K}{\la_1+K}>0$, for any $x\in \II$:
 $$
P^e\circ P_\eps (x)= P^e(x^+_0+\alpha ^+\eps + \OO(\eps ^{\frac{2p}{2p-1}}))=x_0^u +
 \OO(  \eps^{-\frac{\la_2+K}{\la_1+K}}) \in \II
 $$ 
and therefore $P^e\circ P_\eps$ send the interval $\cal I$ to an interval $J\subset \II$ centered at the 
point $x_0^u<x_0^-$ and of size $\OO(  \eps^{-\frac{\la_2+K}{\la_1+K}})$,
and is a contraction. This gives the existence of a periodic orbit $\Gamma_{0,\eps}$
of the regularized vector field $Z_{0,\eps}$.

Let us observe that, for $\mu=0$ the non-smooth system $Z_0$ has not a periodic cycle but a homoclinic one 
$\Gamma_0= W^u_+(0,0) \cup W^u(x_s,y_s)  \cup \{ (x,0), \ x^*\le x \le 0\}$, where 
$(x^*,0)= W^{u}(x_s,y_s)\cap \Sigma$. Clearly $\Gamma_{0,\eps} \to \Gamma_{0}$ as $\eps \to 0$.
Nevertheless, it is straightforward to see that, for $\mu <0$ we also have that $P^e\circ P_\eps$ 
send the interval $\cal I$ to an interval $J\subset \II$ and is a contraction. This gives the uniformity in $\eps$ for $\mu \le 0$.

If $\mu >0$, we use again that  the interval $\II$ is sent by $P_\eps$ to an interval $\JJ$ centered at 
$x^+_0+\alpha ^+\eps $ of size $ \OO(\eps ^{\frac{2p}{2p-1}})$.
On the other hand the intersection of $W^s(\x,\y)\cap \Sigma _{y_0}^+ =(x^s_0,y_0)$, with $x_0^s =x_0^+-\mu$, therefore if 
$$
x_0^s<x^+_0+\alpha ^+\eps + \OO(\eps ^{\frac{2p}{2p-1}}) 
$$
the exterior map is not defined in this interval.

Observe that this happens if $\eps$ is small enough and 
$
-\alpha ^+\eps <\mu
$.
As $\alpha ^+ <0$, this condition is verified if $\eps$ is small enough for $\mu >0$. 
One then conclude that if $\mu >0$ there is no return of the whole interval $\II$ to itself and therefore 
the system has no periodic orbits in this neighborhood of the saddle if $\eps $ is small enough.

If $\mu =\tilde \mu \eps$, with $\tilde \mu >0$,  we use again that
$x_0^u \in \II=[ L^-,x^-_0+\alpha ^- \eps + \mu ^- \eps ^{2\la}+O( \eps ^{1+\la})]$, 
and $P_\eps (\II)$ is an interval $\JJ$ centered at $x^+_0+\alpha ^+\eps $ of size $ \OO(\eps ^{\frac{2p}{2p-1}})$.
Then, the first condition one needs to ensure that $P^e$ is defined in this interval is
$$
x^+_0+\alpha ^+\eps + \OO(\eps ^{\frac{2p}{2p-1}}) <x_0^s=x_0^+-\tilde \mu \eps 
$$ 
which is fulfilled if 
$\tilde \mu <-\alpha ^+$.
Under this condition, we have again that
$$
|P^e(x^+_0+\alpha ^+\eps + \OO(\eps ^{\frac{2p}{2p-1}}))-x_0^u|
< \D ((\alpha ^++\tilde \mu)\eps + \OO(\eps ^{\frac{2p}{2p-1}}) )^{-\frac{\la_2+K}{\la_1+K}}
\le \bar \D \eps^{-\frac{\la_2+K}{\la_1+K}}
$$
and therefore 
$$
 P^e(x^+_0+\alpha ^+\eps + \OO(\eps ^{\frac{2p}{2p-1}}))=x_0^u + \OO(  \eps^{-\frac{\la_2+K}{\la_1+K}}) \in \II
 $$ 
and consequently  
$P^e\circ P_\eps$ 
sends the interval $\II$ to an interval $J\subset\II $ centered at $x_0^u<x_0^-$ of size $ \OO(  \eps^{-\frac{\la_2+K}{\la_1+K}})$
and is a contraction. 
This gives the existence of a periodic orbit $\Gamma_{\tilde \mu \eps,\eps}$
of the regularized vector field $Z_{\tilde \mu \eps,\eps}$ for $\tilde \mu <-\alpha ^+$ and $\eps$ small enough.

We want to emphasize  that, as 
$x_0^u \in \II$ one has that $P_\eps (x_0^u)=x^+_0+\alpha ^+\eps + \OO(\eps ^{\frac{2p}{2p-1}})$. 
Therefore, for $\tilde \mu <-\al ^+$ one has that $P_\eps (x_0^u)< x_0^s = x_0^+ -\tilde \mu \eps$ and, if 
$\tilde \mu >-\al ^+$ one has that $P_\eps (x_0^u)> x_0^s = x_0^+ -\tilde \mu \eps$. Therefore the value $\tilde \mu =-\al ^+$  
corresponds, in first order,  to the value where the regularized vector field has a homoclinic orbit associated to the saddle $(\x,\y)$. 
We have then that the periodic orbits which existed for $\mu <-\al ^+ \eps$ disappear in a homoclinic orbit giving rise to the so 
called "homoclinic " bifurcation of $Z_{\mu,\eps}$ (see figure \ref{homoclinic}).
\end{proof}

\begin{remark}
Another situation where this phenomenon occurs is in the case that $\X_\mu$ is a Hamiltonian system. 
In this case, generically, the stable and unstable manifolds of $(\x,\y)$ coincide along a homoclinic orbit which surrounds a 
collection of subharmonic orbits. 
In this case,  for $\mu <0$ both stable and unstable curves of the saddle intersect transversally the switching manifold $\Sigma$ and 
therefore the homoclinic connexion disappears. 
Then for $\mu=0$ the homoclinic orbit  is tangent to $\Sigma$, producing a pseudo-separatrix connexion between the saddle and the visible fold. 
For  $\mu >0$ the homoclinic orbit is contained in $\VV^+$ and the unstable pseudoseparatrix of the visible fold, does not intersect $\Sigma$ anymore.
\end{remark}

 \begin{theorem}\label{thm:hseparatrix}
Let $Z_\mu$ be a family of non-smooth vector fields such that $\X_\mu$ is a Hamiltonian vector field and has an homoclinic orbit to a 
saddle point  $(\x,\y)\in \VV^+$ of $\X$, that undergoes a sliding homoclinic bifurcation 
generated by a generic tangency between the homoclinic orbit of the  saddle and the switching manifold $\Sigma$ which occurs for $\mu =0$.
Assume that for $\mu <0$ both stable and unstable curves of the saddle intersect transversally the switching manifold $\Sigma$ 
and for $\mu>0$ the homoclinic orbit is contained in $\VV^+$. 
Assume also that $\Y_\mu$ is transversal to $\Sigma$ and points towards $\Sigma$ for $\mu$ small enough.
Consider the regularized family $Z_{\mu, \eps}$, then:
\begin{itemize}
\item
If $\mu < 0$, 
the  system $Z_{\mu}$ has a grazing  periodic orbit $\Gamma_{\mu}$ and the regularized system $Z_{\mu, \eps}$  has a semistable 
periodic orbit $\Gamma_{\mu,\eps}$ for $\eps$ small enough uniformly in $\mu$, which approaches, when $\eps \to 0$,  
the grazing  periodic orbit $\Gamma_\mu$.
\item
If $\mu =0$, the system $Z_\mu$ has an sliding homoclinic orbit  $\Gamma_0$, and $Z_{0, \eps}$ has a 
semiestable  periodic orbit $\Gamma_{0,\eps}$ which approaches, when $\eps \to 0$,  the sliding homoclinic  orbit $\Gamma_0$.
\item
If $\mu>0$ the only periodic orbits of $Z_{\mu,\eps}$  near the stable separatrix of the saddle are the subharmonic orbits of $Z_\mu$
\item
The periodic orbit $\Gamma_{\mu,\eps} $ exists until  $\mu= \tilde \mu \eps$, with $\tilde \mu <-\alpha ^+$, 
where $\alpha ^+$ is given in proposition \ref{prop:flowtangency}. When $\tilde \mu$ approaches $-\al ^+$ this orbit 
becomes the homoclinic orbit of $\X$.
\end{itemize}
\end{theorem}

\begin{figure}
\begin{center}
\includegraphics[width=12cm]{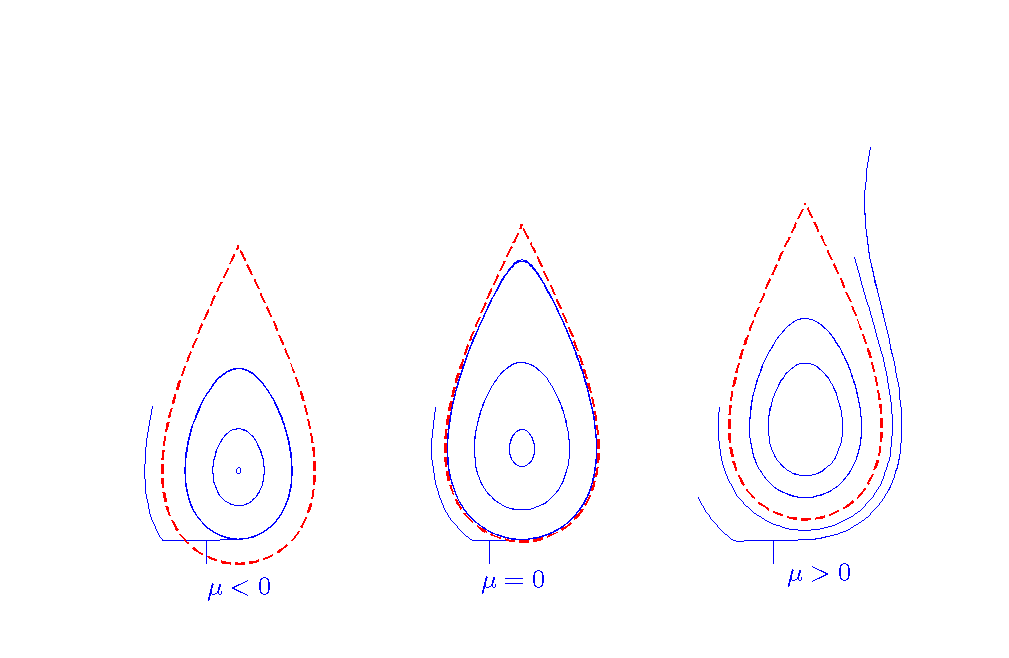}
\end{center}
\caption{Homoclinic bifurcation in the regularized system corresponding with the sliding homoclinic bifurcation in the Filipov system: Hamiltonian case.}
\label{homoclinich}
\end{figure}

\begin{proof}
As $\X$ is Hamiltonian, the homoclinic orbit $W^u(\x,\y)=W^s (\x,\y)$ surrounds a family of subharmonic periodic orbits.
As in Theorem \ref{thm:separatrix}, one can assume that the fold point, which exists for $\mu$ small enough, is located at $(0,0)$. 
Again, we denote by  $=(x_0^{\pm}, y_0)=W^{u,s} (0,0) \cap \Sigma _{y_0}^{\pm}$, the intersection of its stable and unstable pseudo-
separatrices with  $\Sigma _{y_0}^\pm$ and by 
$W^{s,u}(\x,\y) \cap \Sigma _{y_0}^\pm=(x_0^{s,u}, y_0)$. 
We also assume that $x_0^\pm$ are independent of $\mu$, 
and $x_0^s=x_0^+ -\mu$ and $x_0^u = x_0^-+\mu$.

Is $\mu <0$ small but fixed, and $\eps$ is small enough, one has that 
$x^+_0 <x_0^s $.
Therefore, in this case, the stable and unstable peudo-separatrizes of the fold coincide in a grazing periodic orbit 
$\Gamma_\mu$ whose interior is full of periodic orbits surrounding a centre.
We are then in the hypotheses of Theorem \ref{thm:centre} and we obtain, in the regularized system, that the  
periodic orbit $\Gamma_{\mu,\eps}$ of $\X_\mu$ which is tangent to the section $y=\eps$  persists in the regularized system 
becoming  a semistable periodic orbit.

When $\mu =0$ one has that $ x_0^+ =x^u_0$ and $ x_0^- =x^s_0$, therefore, we have two heteroclinic connexions between the 
fold and the saddle forming an homoclinic orbit of the saddle $(\x,\y)$ tangent to $\Sigma$.
By Theorem \ref{thm:main}, we have that the map $P_\eps$ sends the interval 
$ \II=[ L^-,x^-_0+\alpha ^- \eps + \beta ^- \eps ^{2\la}+O( \eps ^{1+\la})]$, 
to an interval $\JJ$ centered at $x^+_0+\alpha ^+\eps $ of size $ \OO(\eps ^{\frac{2p}{2p-1}})$.
As $\al ^+ <0$, one has that $x_0^+ + \alpha ^+ \eps <x_0^s $ and one can apply  inequality \eqref{eq:poincaresella}  obtaining
$$
|P^e(x^+_0+\alpha ^+\eps +\OO(\eps ^{\frac{2p}{2p-1}})) -x_0^u| \le \D 
|\alpha ^+\eps +\OO(\eps ^{\frac{2p}{2p-1}}) |^{-\frac{\la_2+K}{\la_1+K}}
$$
and then $P^e(\JJ)$ is an interval $J$ containing  the point $x_0^u=x_0^-$ and of size $|\eps |^{-\frac{\la_2+K}{\la_1+K}}$.
An important observation is that, being the interval $\JJ$ in the left of the point $x_0^s$ we know that $P^e(\JJ)=J \subset 
[x_0^-, x_0^- +\OO(|\eps |^{-\frac{\la_2+K}{\la_1+K}})]$.

Once we have this interval $J$ contained in the interior of the homoclinic loop, we can use the reasoning of 
Theorem \ref{thm:centre} to obtain that the successive iterates of the return map for any point of this interval 
form a increasing sequence which converges to the point $\bar x_\eps$, which is the intersection of the periodic 
orbit $\Gamma_{\mu,\eps}$ of $\X$ tangent to $\{y=\eps\}$ with $\Sigma_{y_0}$.

If $\mu >0$, then one has that $ x_0^+ >x^s_0=x_0^+ -\mu$.
By Theorem \ref{thm:main}, we have that the map $P_\eps$ sends the interval 
$ \II=[ L^-,x^-_0+\alpha ^- \eps + \mu ^- \eps ^{2\la}+O( \eps ^{1+\la})]$, 
to an interval $\JJ$ centered at $x^+_0+\alpha ^+\eps $ of size $ \OO(\eps ^{\frac{2p}{2p-1}})$.
But now one has that, if $\eps$ is small enough,  $x_0^+ + \alpha ^+ \eps >x_0^+-\mu=x_0^s $ and therefore the exterior map $P^e$ is 
not defined in this interval. 
As a consequence, all the orbits beginning at $\II$ do not intersect $\Sigma_{y_0}^-$ anymore and there is no possibility of existence of 
periodic orbits near the fold.

When $\mu =\tilde \mu \eps$, $0\le \tilde\mu$,  one has that the exterior map $P^e$ is still defined in the interval $\JJ$ if  
 $x_0^+ + \alpha ^+ \eps <x_0^+-\tilde \mu\eps $ and this occurs, again, while $\tilde \mu <-\alpha ^+$.
Therefore, for these range of parameters, we still have a semistable periodic orbit in the system.
Let us observe that the value $\tilde \mu = -\alpha ^+$ gives, in first order, the value of $\mu$ such that the homoclinic orbit of $\X$ 
is tangent to $\Sigma _\eps$ and, therefore, this tangent semistable periodic orbit disappears (see figure \ref{homoclinich}).
 \end{proof}

\section{Proof of Theorem \ref{thm:main}}\label{sec:proof}

In this section we will study the Poincar\'{e} map $P_\eps = \bar P\circ \PP_\eps\circ P:\Sigma_{y_0}^- \to \Sigma _{y_0}^+$.
By Proposition \ref{prop:flowtangency} we know the behavior of the maps $P$ and $\bar P$, which only depend of the vector field 
$\X$ and are the same for $Z$ and for its regularization $Z_\eps$.
That's not the case for the map $\PP_\eps$ that, for the non-smooth system, is simply:
$$
\PP_0(x) = \gamma\sqrt{\eps} + O(\eps), \quad \gamma >0,\  \mbox{if} \ x<-\ga \sqrt{\eps}
$$
where, $W^{u,s}_+(0,0)\cap \Sigma _\eps ^\pm=(\pm \gamma \sqrt{\eps} + O(\eps),\eps)
$

To study the map $\PP_\eps$ we need to control  the behavior of solutions of $Z_\eps$ near the fold point $(0,0)$. 
This is done in next sections using geometric singular perturbation theory and matching asymptotic expansions.

As we want to perform a local analysis near $(0,0)$, which is a fold-regular point for $Z$, 
following \cite{GuardiaST11}, we assume  that, locally, near $\Sigma $, the systems can be written as:
\begin{equation}\label{def:X}
\X(x,y)=\left(\begin{array}{l}
        1\\
        2x
       \end{array}\right)
\end{equation}
and
\begin{equation}\label{def:Y}
\Y(x,y)=\left(\begin{array}{l}
        0\\
        1
       \end{array}\right).
\end{equation}
Later, in section \ref{sec:generalfold}, we will show how to extend all the results in 
this section to a general vector field $Z$ near a fold-regular point.

Observe that, for the vector fields \eqref{def:X} and \eqref{def:Y}, we have explicit expressions for the maps $P$, $\bar P$:
\begin{equation}\label{eq:pp}
P(x)= -\sqrt{x^2+\eps-y_0}, \quad  \bar P(x )= \sqrt{x^2-\eps+y_0}.
\end{equation}
Therefore, in this case, one has $x^\pm _0= \pm \sqrt{y_0}$, and the constants given in Proposition \ref{prop:flowtangency} are   
$\alpha ^\pm =\mp \frac{1}{2\sqrt{y_0}}$, and $\beta ^\pm = \pm \frac{1}{2\sqrt{y_0}}$.

The regularized system \eqref{regularizedvf} $Z_\eps$ leads to the differential equations:
\begin{equation}\label{eq:regularized}
\begin{array}{rcl}\dot x &=& \frac{1}{2}(1+\varphi (\frac{y}{\eps}))\\
\dot y &=& \frac{1+2x}{2} +\frac{1}{2}\varphi (\frac{y}{\eps})(2x-1).
\end{array}
\end{equation}

\subsection{The slow invariant manifold}

Let us observe that system \eqref{eq:regularized} can be written, with the change of variable $y=\eps v$ as:
\begin{equation}\label{eq:slow}
\begin{array}{rcl}\dot x &=& \frac{1+\varphi (v)}{2}\\
\eps \dot v &=& \frac{1+2x}{2} +\frac{1}{2}\varphi (v)(2x-1)
\end{array}
\end{equation}
and, following singular perturbation methods, we will call this system slow system.
If we now perform the change of time $t= \eps \tau$ we get the so called fast system, corresponding to a vector field, depending regularly on $\eps$,
that we call $Z_\eps$:
\begin{equation}\label{eq:fast}
\begin{array}{rcl}
x' &=& \eps \frac{1+\varphi (v)}{2}\\
v' &=& \frac{1+2x}{2} +\frac{1}{2}\varphi (v)(2x-1) .
\end{array}
\end{equation}

If we put $\eps =0$ in system \eqref{eq:slow} we get:
$$
\begin{array}{rcl}
\dot x &=& \frac{1+\varphi (v)}{2}\\
0&=& 1+2x +\varphi (v)(2x-1) 
\end{array}
$$
which is a differential equation in a manifold. This manifold is usually called the \emph{slow manifold}, that, for our system, is a curve:
\begin{equation}\label{SM}
\Lambda_0=\{ (x,v), \  \varphi(v) = \frac{1+2x}{1-2x}, \ x\le 0\}.
\end{equation}
Observe that, for the functions $\varphi$ given in  \eqref{diferenciable}, 
$\Lambda _0$ only exists for negative values of $x$, because for these values one has that $-1\le \frac{1+2x}{1-2x}\le 1$.

$\Lambda _0$  is a manifold of critical points of the fast system  \eqref{eq:fast} for $\eps =0$. 
Moreover,  for $(x,v) \in \Lambda _0$:
\begin{equation}\label{matriunhm}
DZ_0 (x,v)= 
\left( 
\begin{array}{cc} 0&0\\ 1+\varphi(v)& \frac{\varphi'(v)}{2}(2x-1)
\end{array}
\right)
\end{equation}
As  $\varphi'(v)(2x-1)\le 0$ for all the points in $\Lambda_0$,  the manifold $\Lambda_0$ is a normally hyperbolic attracting manifold  
for the vector field $Z_0$.
Except in the linear case, for the functions $\varphi$ we consider, it is clear that $\varphi'(1)=0$, and therefore, as 
$(0,1) \in \Lambda _0$, we will have that $\Lambda_0$ looses its hyperbolic character when  $x\to 0$. 
In any compact subset of the region $x<0$, we can apply Fenichel theorem \cite{Fenichel79, Jones94}, which ensures the existence of a normally 
hyperbolic attracting invariant manifold $\Lambda _\eps$
for $\eps$ small enough of system \eqref{eq:fast} (and \eqref{eq:slow}):

\begin{theorem}\label{thm:fenichel} 
Consider any numbers $L, N>0$. 
Then, there exists $\eps _0>0$ and  constants $K,c >0$, such that  for $|\eps|<\eps _0$ system \eqref{eq:slow} 
has a normally hyperbolic invariant manifold $\Lambda _\eps$ such that in the region $-L\le x\le -N$  is $\eps$-close to $\Lambda _0$, 
that is, there exists a smooth function $m(x;\eps) $  such that
\begin{itemize}
\item
$
\Lambda _\eps = \{(x,v) \ -L \le x\le -N,  v=m(x;\eps)\}$ 
is a normally hyperbolic attracting locally invariant manifold of system \eqref{eq:fast}.
\item
If  $-L \le x\le -N$ we have that $|m(x;\eps)-m_0(x)|\ \le K \eps$, where $m_0(x) = \varphi ^{-1}(\frac{1+2x}{1-2x})$.
\item
There exists a neighborhood $U$ of $\Lambda_\eps$ such that for any point $z_0 \in U$ one has that there is a point $z^* \in \Lambda _\eps $ 
such that 
\marginpar{call allargar  $U$ a punts de la forma $(x_0,1)$}
$$
|\phi(t, z_0)-\phi(t, z^*)| \le K e^{-c\frac{t}{\eps}}, \quad t \ge 0
$$
where $\phi$ is the flow of system \eqref{eq:slow}.
\item
The set $\{ (x_0,1), \  -L\le x_0 \le - N\}$ is contained in $U$.
\end{itemize}
\end{theorem}
\begin{proof}
The proof of this theorem can be found in \cite{Fenichel79}. So, we only need to proof the last item.
By Fenichel theorem, we know that there exists a neighborhood $U$ of the manifold $\Lambda _\eps$ where it is exponentially attracting
for the slow system \eqref{eq:slow}.
Consider now a subset $U'$ such that
$\Lambda _\eps\subset U'\subset U$. 
Fix $-3N\le x_0\le -2N$, and consider the solution $z(t;\eps)$  of system \eqref{eq:fast} with initial condition $z_0=(x_0,1)$.
It is clear that, for any $T>0$, there exists $\eps _0=\eps_0(T)$ such that for $|\eps| \le \eps _0$, one has 
$$
z(\tau;\eps)= z(\tau;0)  +O(\eps), \ 0\le \tau \le T
$$ 
where $z(\tau;0)$ is the solution of 
\begin{equation}\label{eq:slow0}
\begin{array}{rcl}
x' &=& 0\\
v' &=& \frac{1+2x}{2} +\frac{1}{2}\varphi (v)(2x-1)
\end{array}
\end{equation}
and therefore 
is of the form $z(\tau;0)= (x_0, v(\tau))$, where $v(\tau)$ is the solution of the second equation with initial condition $v(0)=1$. 
On the other hand the second component of this vector field is zero in the slow manifold $\Lambda _0$ and is negative for points $(x,v)$ 
such that $v >m_0(x)$. 
Moreover, for $(x,v)$ such that $-3N\le x\le -2N$, and $m(x;\eps) \le v \le 1$, such that $(x,v) \not \in U'$, there exists $M>0$ such that
$$
\frac{1+2x}{2} +\frac{1}{2}\varphi (v)(2x-1) \le -M
$$
and therefore, we know that there exists a time $T=T(x_0)$ such that $z(T;0) \in U\setminus U'$, and consequently, there exists $\eps _0=\eps
_0(x_0)$ such that
for $|\eps|\le \eps _0$ one has that 
$z(T;\eps)  \in U\setminus U'$. Now, as $x_0$ is in a compact set, there exists $\eps _0$  such that the result is true for any point 
$(x_0,1)$, with $-L\le x_0\le -2N$. Now, we rename $N$ as $2N$ and we obtain the result.
\end{proof}

This Theorem gives us the existence of the slow invariant manifold and its property of being attracting for points  of 
the form  $(x,1)$ for $-L\le x\le -N$, for fixed $N>0$. Later, in Theorem \ref{prop:atractiogran}, we will see that in 
fact the manifold is attracting also for points which are closer to the point $(0,1)$.

\begin{remark}\label{rem:canvivar}
By Theorem \ref{thm:fenichel} we know that, for any $N>0$,  in $-L\le x\le -N<0$ the Fenichel invariant manifold  can be described by
$$
v=m(x;\eps), \ -L\le x\le -N, \ 0\le \eps\le 0
$$ 
where $m(x;\eps)$ is a differentiable function, even for $\eps =0$. 
Moreover, the invariant character and the fact that $m(x;0)= m_0(x)$ implies that $m(x;\eps)$ has a unique expansion on $-L\le x<0$:
$$
m(x;\eps)= m_0(x)+\eps m_1(x)+ \OO(\eps ^2).
$$
Of course, this expansion is only valid on $-L\le x<-N$, that is, when $N\to 0$, the range of $\eps$-validity of the expansion tends to zero.

Nevertheless, if we fix $L>0$ small enough (but independent of $\eps$), one can guarantee that 
$
m'_0(x) >M>0
$, in fact we have that $m'_0(x)\to \infty$ as $x\to 0$. 
Therefore, we can express the slow manifold $\Lambda _0$ as
$x=n_0(v)$ for $m_0(-L) \le v\le 1$, 
and due to the unicity of the asymptotic expansion and the uniform validity in $-L\le x\le -N$, the invariant manifold $v=m(x;\eps)$  
can also be expressed, inverting $m$ as $x=n(v;\eps)$, with 
$$
n(v;\eps )= n_0(v)+\eps n_1(v) +\OO(\eps^2)
$$
where  the functions $n_i$ are uniquely determined for $m_0(-L) \le v\le 1$ by the invariance condition. 
Naturally, the asymptotic validity can only take place for $m_0(-L) \le v\le m_0(-N)$.

Then, if $m_1(x)>0$ for $-L\le x<0$, we will have:
$$
m(x;\eps) <m_0(x), \ -L \le x\le -N<0,$$ 
and 
 if $n_1(v)>0$, we will have:
$$
n_0(v) <n(v;\eps), \ m_0(-L) \le v\le m_0(-N)<0.
$$ 
\end{remark}

Once we know that the orbit of all the  points in $U$ gets exponentially close to $\Lambda_\eps$ and that  
$\Lambda_\eps$  is $\eps$-close to  $\Lambda_0$ until $(x,v)$ enter the region $x\ge -N$, now we want to 
follow the orbits when they get closer to the point $(0,1)$. 
In this region Fenichel theorem is no valid so we will  use some asymptotic expansions to get the main terms in the asymptotic series 
of the invariant manifold $\Lambda _\eps$.
Consequently, as all the orbits are exponentially small close to $\Lambda _\eps$,  these terms will be valid for the asymptotic 
expansion of any solution of the system \eqref{eq:fast}.

As we will see in next sections, the way the manifold $\Lambda _\eps$, and therefore all the orbits in $U$, 
behave near $(0,1)$ strongly depends of the regularity of function $\varphi$.

\subsection{The slow manifold close to $(0,1)$: linear case}\label{lineal}
We first consider the linear case where $\varphi$ is defined in \eqref{caslineal}. In that case, system \eqref{eq:slow} reads:
\begin{equation}\label{eq:fastlineal}
\left. \begin{array}{rcl}\dot x &=& \frac{1+v}{2}\\
\eps \dot v &=& \frac{1+2x}{2} +\frac{v}{2}(2x-1)
\end{array} \right\},\quad  \mbox{for} -1\le v \le 1, 
\end{equation}
and is given by the vector fields $\X$ given in \eqref{def:X} for $v\ge 1$ and by $\Y$ given in \eqref{def:Y} for $v\le -1$.
If one considers system \eqref{eq:fastlineal} for any $(x,v)\in \RR^2$, it has a slow manifold
$\Lambda _0= \{(x,v) \  x < \frac{1}{2}, \  v=\frac{1+2x}{1-2x}, \}$ and it is 
a normally hyperbolic attracting invariant manifold for $x\le N$, if we fix $N<\frac{1}{2}$. 
Therefore, we can apply Fenichel theorem for $0<N<\frac{1}{2}$
and we get a normally hyperbolic invariant manifold $\Lambda_\eps$ for $\eps $ small enough which is given by
$v=m(x;\eps)$ with the function $m$ verifying:
\begin{equation}\label{linearm}
1+2x + m(x;\eps)(2x-1)= \eps (1+m(x;\eps)) m'(x;\eps),
\end{equation}
and the function $m$ is given, up to order $O(\eps^2)$, by:
\begin{equation}
\label{eq:expansiolineal}
m(x;\eps)=m_0(x)+\eps m_1(x)+O(\eps^2)
\end{equation}
with
$$
m_0(x)=\frac{1+2x}{1-2x}, \quad m_1(x)=\frac{1+m_0(x)}{2x-1}m_0'(x)=-\frac{8}{(1-2x)^4}.
$$
Of course, the manifold $\Lambda _\eps$ is the invariant manifold of our regularized system \eqref{eq:fastlineal} until it reaches $v=1$.

To proof that the invariant manifold $\Lambda _\eps$ is attracting for points closer to the fold, we need some extra information of it. 
This is done in next proposition.
\begin{proposition}\label{varietatconfinadalineal}
There exists $K>0$ and $\varepsilon_0>0$, such that, if $0<\varepsilon<\varepsilon _0$ 
the invariant manifold $v=m(x;\eps)$ verifies, for $-L\le x\le \frac{1}{4}$:
\be\label{eq:varietatconfinadalineal}
m_0(x)-\eps K\le m(x;\eps)\le m_0(x)
\ee
\end{proposition}

\proof
To proof this proposition we will see that any orbit of system \eqref{eq:fastlineal} that enters the set:
\be\label{blocconfinadalineal}
\tilde B=\{ (x,v), \ -L\le x\le \frac{1}{4}, \ m_0(x)-\eps K\le v \le m_0(x)
\}
\ee
leaves in through  the border given by $x=\frac{1}{4}$.
Therefore, one needs to check that the flow points  inwards in the other three  borders.
In the  upper border  
$\tilde B^+=\{ (x,v), \ -L\le x\le \frac{1}{4}, \ v = m_0(x)\}$, 
the vector field 
in \eqref{eq:fastlineal} is of the form $((1+m_0(x))/2,0)$.
As  $1+m_0(x)>0$ 
the flow  points inward $\tilde B$ along this border.
So, we need to check the border 
$$
\tilde B^-=\{ (x,v),-L\le x\le \frac{1}{4}, \  v=m_0(x)-\eps K\},
$$ 
whose normal exterior vector is
$n=(m_0'(x),-1)$, and it is  enough to see that
$$
<n,X> _{|\tilde B ^-}<0
$$
for $2X=(\eps (1+m_0(x)-K\eps), 1+2x+(m_0(x)-K\eps)(2x-1))$,
which becomes:
\begin{eqnarray*}
&&m'_0(x)\eps [1+m_0(x)-K\eps]-[1+2x+(m_0(x)-K\eps)(2x-1)] \\
&=& 
\eps \left( -\frac{16}{(1-2x)^5}+ K(2x-1)+\frac{4K\eps}{(1-2x)^4} \right)\le \eps (-\frac{K}{2} + 8K\eps)<0
\end{eqnarray*} 
and this last term is negative for any finite $K>0$, if we take $\eps$ small enough, therefore
$$
<n,X>_{|\tilde B^{-}} \le -\eps (\frac{K}{2}-1) \ll 0.
$$
In the  left border given by $x=-L$ 
the vector field in \eqref{eq:fastlineal} has $\dot x>0$, 
therefore the flow  points inward $\tilde B$ along this border.

Now, we just need to see that, for $x=-L$, the manifold $\Lambda _\eps$ enters $\tilde B$. But this is an easy consequence of 
the expansion \eqref{eq:expansiolineal} and the fact that $m_1(-L,\eps) <0$ and  bounded for $L\to \infty$.
\endproof

From this Proposition \ref{varietatconfinadalineal} we have that the manifold $\Lambda _\eps$ will leave the regularized zone at 
a point $(x_1,\eps)$, with  $1= m_0(x_1)+\eps m_1(x_1) + O(\eps ^2)$, which, using \eqref{eq:expansiolineal}, gives
$x_1= 2\eps + O(\eps ^2)$. 

The same will happen to all the points whose orbits get exponentially closer to $\Lambda_\eps$.
Next proposition shows that this happens to all the solutions with initial conditions at points 
$(x_0,1)$, if $-L\le x_0\le -\eps^\la$, and $\la <1$.
Let's introduce the equations for the orbits of system \eqref{eq:fastlineal}:
\begin{equation}\label{eq:orbitslineal}
\eps \frac{d v}{dx}=\frac{1+2x+v(2x-1) }{1+v}
\end{equation}
and we have the following:
\begin{proposition}\label{prop:atractiogranlineal}
Fix $0<\la<1$ and take any  point $(x_0,1)$, with  $-L\le x_0\le -\eps^\la$.
Then, the orbit of system \eqref{eq:orbitslineal} with initial condition $v(x_0)= 1$ 
stays exponentially close to the invariant manifold $v=m(x;\eps)$ in the region $x\ge 0$.
\end{proposition}

\proof
To proof this proposition,  we perform the change of variables 
$ w=v-m(x;\eps)$ in equation \eqref{eq:orbitslineal} obtaining:
\be \label{eq:orbitswlineal}
\eps \frac{d \, w}{d\, x} = -g(x,\eps)w
\ee
where $g(x;\eps)$ is the positive function:
\[
g(x;\eps)=     \frac{-2x+1+\eps m'(x;\eps)}{1+m(x;\eps)+w(x;\eps)}.
\]

Note that we already know the existence of the solution $w(x;\eps)$ for $x\le 0$, in fact we know it verifies the bound
\[
0\le w(x;\eps)\le 1-m(x;\eps).
\]
For this reason we use the notation $g(x;\eps)$ even if this function depends on $w(x;\eps)$.

Clearly, the solution of \eqref{eq:orbitswlineal} with initial condition $w(x_0) = 1-m(x_0;\eps)$ can be written as:
$$
w(x)= e^{-\frac{1}{\eps}\int _{x_0}^{x} g(s;\eps)ds}w(x_0).
$$
Using that for $x\le 0$ we have that
$
g(x;\eps )\ge \frac{1}{2}
$
we can bound $w(x)$:
\begin{eqnarray*}
|w(x;\eps)|&\le & | w(x_0)|e^{- 
\frac{1}{2\eps}(x-x_0)},
\end{eqnarray*}
therefore if $x_0 \le -\eps ^{\la}$ with $\la <1$, any solution gets exponentially closer to the invariant manifold $v=m(x;\eps)$ for $x\ge 0$.

\endproof

Let us observe that, from proposition \ref{prop:atractiogranlineal} and the fact that the Fenichel manifold reaches $v=1$ for $x= 2\eps +\OO(\eps ^2)$, 
we know that any solution of the system arrives to $v=1$ exponentially close to it, therefore it also cuts $v=1$ at $x= 2\eps +\OO(\eps ^2)$.

\subsubsection{Asymptotics for the Poincar\'{e} map $P_\eps$}

After Theorem \ref{thm:fenichel} and propositions \ref{varietatconfinadalineal} and 
\ref{prop:atractiogranlineal}, we can conclude that the Poincar\'{e} map $\PP_\eps$ is defined in the set 
$[-L, -\eps ^\la]\times \{\eps\}$. Moreover
\begin{equation}\label{ppepsilonlineal}
\forall x \in [-L,-\eps ^\la], \quad \PP_\eps (x) = 2\eps+ \OO(\eps^2).
\end{equation}
Fix $1/2 <\lambda <1$.
Taking into account that,  by \eqref{eq:ppg}
$$
P^{-1}(-\eps ^{\la}) = x^-_0+\alpha ^- \eps + \beta ^- \eps ^{2\la}+\OO( \eps ^{1+\la})
$$ 
we have that 
$$
P([L^-,x^-_0+\alpha ^- \eps + \beta ^- \eps ^{2\la}+O( \eps ^{1+\la})]) \subset [-L, -\eps ^\la]
$$ for a suitable constant $L^-$.

On the other hand we know that the map $\bar P$ is given by formulas \eqref{eq:ppg}.

Therefore we conclude that the map $P_\eps = \bar P\circ \PP_{\eps} \circ  P$
\begin{equation}\label{pepsilonlineal}
\begin{array}{rcl}
P_\eps: [ L^-,x^-_0+\alpha ^- \eps + \beta ^- \eps ^{2\la}+O( \eps ^{1+\la})] 
\times \{y=y_0\} &\to& \Sigma ^+_{y_0} \\
(x,1) & \mapsto &(P_\eps(x), 1)  
\end{array}
\end{equation}
is given by
$$
P_\eps(x)= \bar P(2\eps + O(\eps^2)) =x^+ _0+\alpha ^+ \eps + \OO(\eps ^{2}).
$$ 
Therefore, all the points in the set  $\II=[-L^-, x^-_0+\alpha ^- \eps + \beta ^- \eps ^{2\la}+O( \eps ^{1+\la})] \times \{y_0\}$, $0<\la<1$,
are send by $P_\eps$ to a set 
$\JJ\times \{y_0\}$ and the interval $\JJ$  has, at most, size $\eps^{2}$ and it is centered at the point 
$x^+ _0+\alpha ^+ \eps $. 
Consequently, the Lipchitz constant of $P_\eps$ is, at most, $O(\eps ^{2})$.

\subsection{The slow manifold close to $(0,1)$: smooth $\CC^1$ case}\label{smct01}
\subsubsection{Extending the outer domain}

When the regularizing function $\varphi$ is $\CC^{p-1}$, with $p\ge 2$,  the slow manifold $\Lambda _0$ given in \eqref{SM}, bends near $(0,1)$:
$$
m_0(0)=1, \quad m_0'(x) \to \infty \mbox{ as } x\to 0^-.
$$
As a consequence, the Fenichel manifold can not be expressed as a graph over the $x$ variable when $x$ is near $0$.
In scope of Remark \ref{rem:canvivar}, it is natural then to look for the Fenichel manifold and also for all the orbits of system 
\eqref{eq:fast} as graphs over the $v$ variable.
Then, we consider the equation for the orbits of system \eqref{eq:fast} as:
\begin{equation}\label{eq:orbitsc1}
\frac{dx }{dv}= \eps \frac{1+\varphi(v)}{1+2x+\varphi(v)(2x-1)}.
\end{equation}
To study the behavior of the orbits close to $(x,v)=(0,1)$ we will look for the formal expansion of the Fenichel manifold as 
$$
x = n(v;\eps)= n_0(v)+\eps n_1(v)+ \cdots + \OO(\eps ^n)
$$
where now the slow manifold $\Lambda _0$ is written as 
$$
\Lambda _0=\{ (x,v), \ x=n_0(v) , \ v\le 1\}.
$$
and $n_0(v) = \frac{1}{2}\frac{\varphi(v)-1}{\varphi(v)+1}$.
As  the function $n(v;\eps)$ is a solution of the  equation \eqref{eq:orbitsc1}, it verifies:
\begin{equation}\label{eq:manifold}
(1+2n+\varphi(v)(2n-1))n'=\eps(1+\varphi(v))
\end{equation}
where $'=\frac{d}{d\, v}$.
Solving this invariance equation for $n$ formally one obtains:
\begin{eqnarray}
n_0(v) &=& \frac{1}{2}\frac{\varphi(v)-1}{\varphi(v)+1},\label{expansio1}\\
n_1(v)&=&\frac{1}{2}\frac{1}{n_0'(v)},\label{expansio2}\\
n_2(v)&=&-2 n_1'(v)n_1^2(v)=\frac{1}{2}\frac{n_0''(v)}{(n_0'(v))^2},\nonumber \\
n_3(v)&=&\frac{n_2'(v) n_1^3(v)-n_2^2(v)}{n_1(v)}, \dots  \nonumber\\
\end{eqnarray}
It will be enough for our purposes to keep the two first terms in this expansion.
Looking at the behavior of these functions near $v=1$, one can see that, by \eqref{eq:fipropdeu}, this behavior depends on the value $p$.

From now on in this section we will deal with the $\CC ^1$ case, which corresponds to  $p=2$:
\begin{eqnarray}
n_0(v) &=& \frac{\varphi''(1)}{8}(v-1)^2 + \OO(v-1)^3 \label{no}\\
n_1(v)&=&\frac{2}{\varphi ''(1)} \frac{1}{(v-1)} +\OO(1)\label{n1}\\
n_2(v)&=& \OO(\frac{1}{(v-1)^4})\nonumber \\
n_3(v)&=&\OO(\frac{1}{(v-1)^7})\nonumber 
\end{eqnarray}
etc.,
were we  use the notation
\begin{equation}
\varphi''(1) := \lim _{v\to 1^-}\varphi''(v).
\end{equation}
Therefore, if it exists a solution  $n(v;\eps)$ with asymptotic expansion close to $v=1$, it will behave as:
$$
n(v;\eps) =  \frac{\varphi''(1)}{8}(v-1)^2 (1+\OO(v-1)) +  \frac{2}{\varphi''(1)}\frac{\eps }{v-1} (1+\OO(v-1))+ O(\frac{\eps ^2}{(v-1)^4})+\dots
$$
and this asymptotic expansion fails for
$$
(v-1)^3 =O(\eps),
$$
which indicates that the invariant manifold should remain close to $x=n_0(v)$ until 
$v=1-O(\eps^\frac{1}{3})$.
Next proposition gives rigorously this behavior:

\begin{proposition}\label{blocouter}
Take any $0<\lambda_1<\frac{1}{3}$.
Then, there exists  $M>0$ big enough,  $\delta= \delta (M)>0$ small enough,  and 
$\eps_0= \eps_0(M,\delta)>0$ such that, for $0<\eps\ \le \eps _0$, any solution of system \eqref{eq:fast} which enters the set
$$
\mathbf{B} =   \left\{ (x,v), \ -\delta <v-1<-\eps^{\lambda_1} , \ n_0(v) \le x \le n_0(v)+ \frac{M \eps }{|v-1|}\right\}
$$
leaves it through the boundary $v=1-\eps ^{\la_1}$
\end{proposition}

\proof
To proof this proposition we will see that the vector field points inwards in $3$ of the boundaries of $\bf{B}$.

To see that the flow enters through
$
{\bf{B}}^+ = \left\{ (x,v), \ -\delta <v-1<-\eps^{\lambda_1} , \  x=n_0(v)+ \frac{M \eps }{1-v} \right\}
$
we consider the exterior normal vector to it:
$$
n^*=\left(1,-n'_0(v) -\frac{ M \eps }{(1-v)^2}\right)
$$
and we will proof that $<Z_\eps, n^*>_{|{\bf{B}}^+} <0$.
Computing this scalar product, using the definition of $n_0$ in \eqref{expansio1},  and the fact that 
$1+\varphi (v)\ge 0$, we get the equivalent inequality: 
$$
1+  \frac{2 M}{v-1 }\left[n_0'+\frac{M\eps}{(v-1)^2}\right]<0
$$
which gives:
\begin{equation}\label{desigualtatbox}
1+  \frac{ 2M n'_0}{v-1}+\frac{2M^2\eps}{(v-1)^3}<0 .
\end{equation}

Now we need to check that, taking $M$ big enough and $\delta$ small enough, there exists $\eps_0=\eps_0(M,\delta)$, such that for 
$0<\eps\le \eps _0$, 
this inequality holds if 
$-\delta \le v-1\le -\eps ^{\lambda_1}$ 
for $0<\lambda_1<1/3$.

In $\bf B^+$, using the local behavior of $n'_0(v) =\frac{\varphi''(1)}{4}(v-1)+\OO(v-1)^2$, one has that there exists a constant 
$C$ independent of $\delta$ and $M$ such that:
\begin{eqnarray*}
\left| \frac{ n_0'(v)}{v-1} -\frac{\varphi''(1)}{4}\right| \le C \delta .
\end{eqnarray*}

In ${\bf B}$, one has that $\frac{\eps}{\delta ^3}\le \frac{\eps}{|v-1|^3}\le \eps ^{1-3\lambda_1}<1$. Therefore
one can write \eqref{desigualtatbox} as
\begin{equation}\label{desigualtatbox2}
M\frac{\varphi''(1)}{2} +1 + g (v;\eps) <0 
\end{equation}
where the function $g(v)$ satisfies:
$$
|g(v;\eps) | \le  2M \left [C \delta +M\eps ^{(1-3\lambda_1)}\right] .
$$
As $\varphi''(1)$ is negative, one can choose $M$ big enough, for instance $M\frac{\varphi''(1)}{2}+1<-2C$,  
and then take $\eps _0$ and $\delta$ small enough such that $|g(v;\eps)| <C$ if $\eps <\eps_0$, to have that \eqref{desigualtatbox2} holds.

At  the points $(n_0(v),v)$  of the boundary
${\bf B}^- = \{(x,v) \, ,\quad x= n_0(v)\}$
one has that  the vector field \eqref{eq:fast} is given by
$Z_\eps (n_0(v),v) = (\frac{1+\varphi(v)}{2},0)$ and therefore, as $1+\varphi(v) >0$ for $v\ge -1$ the flow points inward also in this boundary.

When $v=1-\delta$ and $n_0(v) \le x \le n_0(v)+ \frac{M \eps }{|v-1|}$ we also have that $\dot v>0$ and therefore the flow also points 
inward in this boundary.
To conclude the proof we just observe that once the orbits  enter the set ${\bf B}$ as $\dot v>0$ in ${\bf B}$, they can only leave it 
through the upper boundary $v=1-\eps {\la _1}$.
\endproof

By Fenichel theorem \ref{thm:fenichel} and remark \ref{rem:canvivar}, we know that that the invariant manifold $\Lambda_\eps$ is a smooth 
manifold that is $\eps$- close to $\Lambda _0$, which is given by $v=m_0(x)$ ,  until it arrives to $v=1-\delta$. Moreover, $m_0(x)$ is an  
invertible function whose inverse is $n_0(v)$. 
Therefore, in this region the Fenichel manifold can be written as:
$$
x=n(v;\eps)=n_0(v)+ \eps n_1(v) + \OO(\eps ^2), \ \mbox{for} \ v=1-\delta
$$ 
and, as $n_1(v) >0$ for $-1\le v\le 1$ (see \eqref{expansio2}), redefining the constants $M$ big enough and  $\delta$ small enough in proposition 
\ref{blocouter}, the manifold enters in the domain ${\bf B}$ for $v^*= 1-\delta$. 
Then,  it stays there at least until $v^{**}=1-\eps ^{\lambda_1}$ verifying:
$$
n_0(v) < n(v;\eps) < n_0(v)+ \frac{M \eps}{1-v}, \  \mbox{if} \ 1-\delta \le v \le 1-\eps ^{\lambda_1}.
$$
Moreover,  using Theorem \ref{thm:fenichel}, as the manifold $\Lambda _\eps$ attracts exponentially any other solution, all the solutions of 
system \eqref{eq:slow} with initial conditions in $U$ verify the same inequality.

Furthermore, as, for any $\la >0$, one has that  $n_0(1-\eps ^\lambda)= \frac{\varphi''(1)}{8}\eps ^{2\lambda}+ \OO(\eps ^{3\lambda})$,
one concludes  that  
$$
n(1-\eps^\lambda;\eps) =\frac{\varphi''(1)}{8}\eps ^{2\lambda} +  \OO(\eps^{1-\lambda},\eps^{3\lambda})
$$
for any $0<\lambda \le \lambda_1<\frac{1}{3}$ and $\lambda_1$ is the value given in proposition \ref{blocouter}.

And, again, as all the solutions enter in the block exponentially closer to $\Lambda_\eps$,
any solution $x(v)$ with initial condition $x(1)= x_0$ with $-L\le x_0 \le -N$ verifies the same asymptotics:
$$
x(1-\eps^\lambda) =\frac{\varphi''(1)}{8}\eps ^{2\lambda} +  \OO(\eps^{1-\lambda},\eps^{3\lambda}).
$$

\subsubsection{The inner domain}

To reach $v=1$ we need to change our strategy.  Looking at the asymptotic behavior of the functions 
$n_0(v)$, $n_1(v)$, given in \eqref{expansio1} \eqref{expansio2}, one can see that the expansion of 
$n(v;\eps)$ looses its asymptoticity for $v=1-O(\eps ^{1/3})$.
Moreover, $n(v;\eps)$ has order $\eps^{2/3}$ for these values of $v$.
To study this range of values of $v$ we perform the change:
\begin{equation}\label{canviinner}
\begin{array}{rcl}
x &= &\eps ^{2/3} \eta  \\
v &=&1+\eps ^{1/3}u
\end{array}
\end{equation}
to system \eqref{eq:slow} obtaining:
\begin{eqnarray}
\eta'&=& \eps^{ 1/3}\frac{1+\varphi(1+\eps ^{1/3}u)}{2}\label{eq:systeminner}\\
u' &=& \frac{\eps^{- 1/3}}{2}\left(1+2\eps^{2/3}\eta+\varphi(1+\eps ^{1/3}u)(2\eps^{2/3}\eta-1)\right)\nonumber
\end{eqnarray}

The equation for the orbits \eqref{eq:orbitsc1} in these new variables becomes:
\begin{equation}\label{eq:manifoldinner0}
\frac{d \eta}{du}= \frac{\eps^{2/3}(1+\varphi(1+\eps ^{1/3}u))}
{\left(1+2\eps^{2/3}\eta(u)+\varphi(1+\eps ^{1/3}u)(2\eps^{2/3}\eta(u)-1)\right)}
\end{equation}
Calling $\mu = \eps ^{1/3}$, one can write this equation as:
\begin{equation}\label{eq:manifoldinnerorbits}
\frac{d \eta}{du}= \frac{\mu ^2(1+\varphi(1+\mu u))}{\left(1+2\mu^2\eta(u)+\varphi(1+\mu u)(2\mu ^2\eta(u)-1)\right)}
\end{equation}
and we need to study the extension of a solution of this equation $\eta(u;\eps)$, with initial condition  
$\eta (u^*;\eps)$, with 
\begin{equation}\label{eq:ustar}
u^*= \frac{v^*-1}{\eps ^{1/3}}=-\eps ^{\lambda_2-1/3}, \ \mbox{with} \  0<\la _2\le  \la _1,
\end{equation}
where $\la_1$ is given in Proposition \ref{blocouter}, verifying
\begin{equation}\label{eq:initialconditioninnern}
|\eps^{2/3}\eta (u^*;\eps)-n_0(v^*)| \le M\eps ^{1-\lambda_2}
\end{equation}
where 
\begin{equation}\label{eq:vstar}
v^*= 1+\eps ^{1/3}u^*=1-\eps^{\lambda_2}, 
\end{equation}
 to the domain:
\begin{equation}\label{eq:ustardomain}
u^* \le u \le 0, \ u^*= -\eps ^{\la_2 - 1/3}
\end{equation}
which corresponds to $v^*\le v\le 1$.

Formally expanding the solution $\eta(u;\eps)$ of equation \eqref{eq:manifoldinnerorbits} in powers of $\mu=\eps ^{1/3}$
\begin{equation}\label{eq:expansioneta}
\eta (u;\eps)=\eta _0(u)+\mu \eta _1(u)+\OO(\mu^2)
\end{equation}
one can see that  $\eta _0$ is the solution of the so called \emph{inner equation}:
\begin{eqnarray}\label{eq:edoeta0}
\eta _0'= \frac{d \eta_0}{d u}= \frac{4}{8\eta _0 -\varphi''(1)u^2}
\end{eqnarray}
which, with the changes $\bar \eta = \alpha \eta, \quad \bar u = \mu u$, where 
$$
\alpha = -(\frac{\varphi''(1)}{2})^{1/3}, \quad \beta =\left( \frac{(\varphi''(1))^2}{32}\right)^{\frac{1}{3}}
$$
becomes
\begin{equation}
\frac{d \, \bar \eta}{d \,\bar u}= \frac{1}{\bar \eta + \bar u ^{2}}.
\end{equation}

It is known \cite{MischenkoR80} that this  equation has a unique solution $\bar \eta _0(\bar u)$ 
which approaches the parabola $\bar \eta = - \bar u ^2$ as $\bar u \to -\infty$. In fact one has that
\begin{eqnarray}
\bar \eta _0(\bar u) &=& -\bar u^2 - \frac{1}{2\bar u} - \frac{1}{8 \bar u ^4} + 
\OO(\frac{1}{\bar u^7}), \ \bar u \to -\infty\\
\bar \eta _0(\bar u) &=& \Omega _0- \frac{1}{\bar u}+ 
\OO(\frac{1}{\bar u^3}), \ \bar u \to \infty.
\end{eqnarray}
Going back to our variables one has that equation \eqref{eq:edoeta0} has a solution  $\eta_0(u)$ satisfying:
\begin{equation}\label{eq:asimeta0}
\begin{array}{rcl}
\eta _0 ( u) &=& \frac{\varphi''(1)}{8} u^2 + \frac{2}{\varphi ''(1) \, u} +\frac{16}{(\varphi ''(1))3 \,  u ^4} 
+ \OO(\frac{1}{ u^7}),  u \to -\infty\\
\eta_0 ( u) &=&- \frac{2^{1/3}\Omega _0}{(\varphi''(1))^{1/3}}+\frac{4}{\varphi ''(1) \,  u}
+ \OO(\frac{1}{ u^3}),  u \to \infty.
\end{array}
\end{equation}
On the other hand, if one considers the next term in the expansion \eqref{eq:expansioneta} of $\eta (u,\eps)$, one has that
$\eta _1(u)$ is the solution of the equation:
$$
\eta '_1(u)= -\frac{8}{(4\eta _0(u)-\frac{\varphi''(1)}{2} u^2)^2} \eta _1 + 
\frac{2 \frac{\varphi'''(1)}{6} u^3}{(4\eta _0-\frac{\varphi''(1)}{2} u^2)^2} 
$$
which is a linear equation.
It is straightforward to see that there is a solution $\eta _1$  of this equation that, near $-\infty$, 
behaves as:
\begin{equation}\label{asineta1}
\eta _1 (u) \simeq \frac{\varphi'''(1)}{24}u^3+ \OO(u^2),
\end{equation}
and this suggests to consider the isolating block defined by a condition of the type
\begin{equation}\label{eq:noubloc}
|\eta (u) -\eta _0(u)| \le K\mu |u|^3.
\end{equation}

As a consequence of the expansion of $\eta _0$ near $-\infty$  in \eqref{eq:asimeta0} and the asymptotic expansion of $n_0(v)$ near $v=1$ 
\eqref{no}, one has that there exist constants $K_1$, $K_2$, such that
$$
|\eps ^{2/3}\eta _0(u^*)-n_0(v^*)| \le K_1 \eps ^{3\lambda_2} + K_2 \eps ^{1-\lambda_2},
$$
where $v^*$, $u^*$ are  given in \eqref{eq:vstar} and \eqref{eq:ustar} and therefore, by \eqref{eq:initialconditioninnern} and \eqref{eq:noubloc}
one has:
\begin{equation}\label{initialconditionmatching}
|\eps ^{2/3}\eta (u^*;\eps)-\eps ^{2/3}\eta _0(u^*)| \le M \eps ^{1-\lambda_2}+K_1 \eps ^{3\lambda_2} + K_2 \eps ^{1-\lambda_2},
\end{equation}
and we can conclude that the solution given by proposition \ref{blocouter} verifies \eqref{initialconditionmatching} at $u=u^*$ if we take 
$\lambda_2 <\min{(1/4,\lambda_1)}$.

Next proposition proves that any solution verifying \eqref{eq:noubloc} at $u=u^*$, stays close to $\eta_0(u)$ until $u=0$ which corresponds to $v=1$.

\begin{proposition}\label{prop:blocinner}
Take any $0<\la_2<\frac{1}{4}$. Then, there exists $u_0>0$, $K>0$,  and $\eps _0 = \eps _0(u_0,K)$, such that for $|\eps | \le \eps _0$, any solution of system \eqref{eq:systeminner}which enters the set
$$
{\bf B_2}=\{ (u,\eta), \ u^*\le u \le 0, \quad |\eta (u) -\eta _0(u)| \le K\mu M(u)\}
$$
where $u^*=-\eps ^{\la _2-\frac{1}{3}}$,  $\mu = \eps ^{1/3}$,  and the function
$M(u)$ is  defined by:
$$
M(u)= \left\{ 
\begin{array}{ll}
-u^3 & -\infty \le u \le -u_0 <0 \\
u_0^3 &-u_0 \le u \le 0,
\end{array}\right.
$$
leaves it through the boundary $u=0$.
\end{proposition}

\proof
To proof this proposition we need to see that the vector field \eqref{eq:systeminner} points inwards in the three boundaries of ${\bf B_2}$: 
$$
{\bf B_2}^\pm=\{ (u,\eta), \ u^*\le u \le 0, \quad \eta (u) =\eta _0(u)\pm  K\mu M(u)\}, 
$$
and $u=u^*$.

The exterior normal vector to ${\bf B_2^+}$ is given by 
$n^+= (1, -\eta '_0(u)-K\mu M'(u))$, therefore the condition for the solutions don't leave the box on its right boundary ${\bf B_2^+}$ is:
\begin{equation}\label{blocinner}
E:= <v, n^+> <0
\end{equation}
where $v=(\mu ^2 (1+\varphi(1+\mu u)), 1+2\mu ^2 \eta +\varphi(1+\mu u)(2\mu ^2 \eta -1))$.

First observation is that
\begin{eqnarray*}
E&=& \mu ^2 (1+\varphi(1+\mu u)) -E_2\\
E_2&=& E_1 (\eta '_0(u)+K\mu M'(u))\\
E_1&=& 1+2\mu ^2 \eta +\varphi(1+\mu u)(2\mu ^2 \eta -1) = 1-\varphi(1+\mu u)+2\mu ^2(1+ \varphi(1+\mu u))\eta
\end{eqnarray*}
We can develop $E_1$ by using the Taylor series of the function $\varphi$:
$$
\varphi(1+\mu u)=1+\frac{\varphi''(1)}{2}(\mu u)^2+ O((\mu u)^3)
$$
and the fact that, in ${\bf B_2^+}$, 
one has that $\eta (u) =\eta _0(u)+ K\mu M(u)$, and also the equivalence \eqref{eq:edoeta0}, obtaining
\begin{eqnarray*}
E_1 
&=& \frac{2\mu ^2}{\eta '_0(u)}+ 4 K \mu ^3 M(u) + g(u;\mu)
\end{eqnarray*}
where $g(u;\mu)$ is exactly given by
\begin{equation}\label{eq:g}
g(u;\mu)=-\varphi(1+\mu u)+1+\frac{\varphi''(1)}{2}(\mu u)^2 + 2\mu ^2 (\varphi(1+\mu u)-1)(\eta _0(u)+K\mu M(u))
\end{equation}
From the asymptotics of $E_1$ one easily obtains:
$$
E_2=  2\mu ^2 +4K\mu ^3 M(u)\eta'_0(u)+ \tilde g(u;\mu)
$$
where
\begin{equation}\label{eq:tildeg}
\tilde g(u;\mu)= \left(   \frac{2\mu ^2}{\eta'_0}+ 4K\mu^3 M(u)+g(u;\mu)\right)K\mu M'(u)+ g(u;\mu)\eta'_0(u)
\end{equation}
and finally:
$$
E=-4K\mu ^3 \eta'_0(u) M(u)+ \bar g(u;\mu)
$$
and 
\begin{equation}\label{eq:barg}
\bar g(u;\mu)= - \tilde g(u;\mu)+ \mu ^2\left(\varphi(1+\mu u)-1\right)
\end{equation}

Now we need to bound the remainder $\bar g(u;\mu)$. 
To this end, using the asymptotics for $\eta _0$ given in \eqref{eq:asimeta0}, 
we know that there exists $a>1$, $C>0$ such that:
\begin{eqnarray*}
|\eta _0(u)| \le C u^2, \quad |\eta '_0(u)| \le C u, \quad \mbox{if} \quad u\le -a \\
|\eta _0(u)| \le C, \quad |\eta '_0(u)| \le C , \quad \mbox{if} \quad -a\le u \le 0 .
\end{eqnarray*} 
In the sequel we will take $u_0 >a$ and we denote by the letter $C$ to any constant independent of $u_0$, $K$.
Also, we will use that, in the considered domain, $|\mu u| <1$ and that we can assume that $K>1$.

Using these bounds for $\eta_0$ and \eqref{eq:fipropdeu} with $p=2$,  we can bound $g(u;\mu)$ as
\begin{eqnarray*}
u^*\le u\le -u_0& ,& \quad |g(u;\mu)|\le C\left( |\mu u |^3 (1+K|\mu u|^2)\right)\\
-u_0 \le u \le 0 &,& \quad |g(u;\mu)| \le C(\mu u_0 )^3 (1+K |\mu u_0|^2)
\end{eqnarray*}
From this bound we obtain:
\begin{eqnarray*}
u^*\le u\le -u_0& ,& \quad |\tilde g(u;\mu)|\le C ( \mu ^3 |u|^4 + \mu ^3  K |u| +\mu ^4 K^2|u|^5 + \mu ^5 K |u|^6 +\mu ^6 K^2 |u|^7 )\\
-u_0 \le u \le 0 &,& \quad |\tilde g (u;\mu)|\le C(\mu u_0 )^3 (1+K|\mu u_0|^2)
\end{eqnarray*}
and for $\bar g(u;\mu)$:
\begin{eqnarray*}
u^*\le u\le -u_0& ,& \quad |\bar g(u;\mu)|\le  C ( \mu ^3 |u|^4 + \mu ^3  K |u| +\mu ^4 K^2|u|^5 + \mu ^5 K |u|^6 +\mu ^6 K^2 |u|^7 ) \\
-u_0 \le u \le 0 &,& \quad |\bar g(u;\mu) |\le C(\mu u_0 )^3 (1+K|\mu u
_0|^2)+ \mu ^4  u_0 ^2
\end{eqnarray*}

Finally, one can write:
$$
E= 4K\mu ^3 M(u) \eta'_0(u) \left( -1+ G(u;\mu)\right)
$$
where $G$ is the function 
$$
G(u;\mu)=\frac{\bar g(u;\mu)}{4K \mu ^3  M(u)\eta'_0(u)}=\frac{\bar g(u;\mu)}{4K \mu ^3  M(u)}   (8\eta_0(u)-\varphi''(1)u^2).
$$ 
Using that 
\begin{eqnarray*}
|4K \mu ^3  M(u)\eta'_0(u)| \ge C K \mu ^3 u^4 \quad \mbox{if} \quad u\le -u_0\\
|4K \mu ^3  M(u)\eta'_0(u)| \ge C K \mu ^3 u_0^3 \quad \mbox{if} \quad -u_0 \le u \le 0 ,
\end{eqnarray*}
the function $G$  verifies the following bounds:
\begin{eqnarray*}
u^*\le u\le- u_0& ,& \quad |G(u;\mu)|\le C  (\frac{1}{K}+\frac{1}{|u|^3}+ \mu K |u|+\mu ^2 |u|^2 +\mu ^3 K |u|^3) 
\\ 
-u_0 \le u \le 0 &,& \quad |G(u;\mu)| \le C(\frac{1}{K} + \mu ^2 u_0 ^2 +\frac{\mu}{K u_0})
\end{eqnarray*}
and therefore, 
using that  $\eta' _0(u)$ is a positive function for any $u\le 0$ 
one can choose $K$  and $u_0$ big enough in such a way that 
$$
|\frac{1}{K}+\frac{1}{u_0^3} | \le \frac{1}{4C}
$$
and then, using that $|\mu u| \le  |\mu u^*| = \eps ^{\la _2}$, 
$E$ is negative if $\eps$, and therefore $\mu =\eps ^{\frac{1}{3}}$, is small enough.

The proof for ${\bf B_2^-}$ is analogous.

When $u=u^*$ one has that the flow of \eqref{eq:systeminner} verifies $u'>0$, therefore it also points inwards ${\bf B_2}$.

\endproof

As in ${\bf B_2}$ one has that $\dot u>0$,  we have that the solutions which enter  ${\bf B_2}$ leave it at $u=0$. 
By \eqref{initialconditionmatching}, the invariant manifold $n(v;\eps)=\eps ^{2/3}\eta(\frac{v-1}{\eps^{1/3}})$, and 
therefore any solution $x(v;\eps)$, enters in it at $v=v^*$  and we have then it  crosses the line $v=1$ at a point verifying:
$$
x(1;\eps)= \mu ^2\eta_0(0)+ O(\mu ^3) = \eps ^{2/3} \eta (0)+ O(\eps)
$$

\subsubsection{Exponential attraction of the whole neighborhood of the fold}\label{sec:exponential}

Once we have a complete control of the Fenichel invariant manifold
until it reaches the boundary $v=1$ of our regularized system \eqref{eq:slow},  now it is necessary to prove that this manifold attracts all 
the points in the section 
$\{ (x,v), \ v=1, \ -L\le x\le -\varepsilon^\la\}$ for $0<\lambda <\frac{2}{3}$. This is an extension of the last item of Theorem \ref{thm:fenichel}.
To this end, we need   a better control of the manifold $\Lambda _\eps =\{ (x,v), \ v=m(x;\varepsilon)\}$ in this region.

Using that the function $m(x;\eps)$ is the inverse of the function $n(v;\eps)$ (see Remark \ref{rem:canvivar}) which verifies the 
invariance equation \eqref{eq:manifold}, we obtain an invariance equation for $m$:
$$
1+2x + \varphi (m(x;\eps))(2x-1)= \eps (1+\varphi (m(x;\eps))) m'(x;\eps).
$$
Writing:
$$
m(x;\eps)= m_0(x)+\eps m_1(x)+ \eps ^2 m_2(x)+\dots
$$
one  gets:
\begin{eqnarray}
m_0(x) &=&\varphi ^{-1}(\frac{1+2x}{1-2x})\label{m0}\\
m_1(x)&=&\frac{(1+\varphi(m_0(x))m_0'(x)}{\varphi'(m_0(x))(2x-1)}
=-\frac{1}{2}(m_0'(x))^2 \nonumber
\end{eqnarray}
were we have used the relation
\be \label{vpm}
\varphi '(m_0(x))=\frac{4}{m' _0(x) (1-2x)^2}
\ee
and one can get more terms in the expansion, but just with these terms one can guess the main part of the asymptotic expansion of $m(x;\eps)$.

Observe that:
$$
m_0(x)=\varphi ^{-1}(\frac{1+2x}{1-2x})=\varphi ^{-1}(1+2x+4x^2 +\dots)
$$
on the other hand $\varphi (v)= 1+\frac{\varphi''(1)}{2}(v-1)^2+\OO(v-1)^3$, $v\le 1$,  and therefore we obtain that:
\be\label{m0-1}
m_0(x) =1 -\frac{2}{\sqrt{-\varphi''(1)}}\sqrt{|x|}+ \OO(x),\quad m_1(x)
=\OO(\frac{1}{x}), \ x\le 0.
\ee
Looking at these terms one can guess that the asymptotic expansion for $m(x;\eps)$ will fail at 
$x=\OO(\eps ^{2/3})$, which corresponds to $v=m_0(x) \simeq 1+ \OO(\eps^{1/3})$ as we saw in proposition \ref{blocouter}. 
This is  given in next proposition.

\begin{proposition}\label{prop:atractiogran1}
Let $L>0$ the constant given in Theorem \ref{thm:fenichel} and $0<\la <2/3$.
Then, there exists $K>0$ and $\varepsilon_0>0$, such that, if $0\le \varepsilon<\varepsilon _0$ 
the invariant manifold $\Lambda_\eps=\{(x,v), \ v=m(x;\eps)\}$ verifies, for $-L\le x\le -\eps ^{\la}$:
\be\label{eq:varietatconfinada}
m_0(x)+\frac{\eps K}{x}\le m(x;\eps)\le m_0(x) .
\ee
\end{proposition}

\proof
To proof this proposition we will see that the set
\be\label{blocconfinada}
\tilde B=\{ (x,v), \ -L\le x\le -\eps ^\la, \quad m_0(x)+\frac{\eps K}{x}\le v \le m_0(x)
\}
\ee
is positively invariant for system \eqref{eq:fast}.
Therefore, one needs to check that the flow points  inwards in three of the borders of $\tilde B$.

In the  upper border  
$\tilde B ^+ =\{ (x,v), \ -L\le x\le -\eps ^\la, \quad v= m_0(x)\}$, 
the vector field 
$Z_\eps$ 
in \eqref{eq:fast} is of the form $(\eps(1+\varphi(m_0(x)))/2,0)$, 
and $(1+\varphi(m_0(x)))/2>0$ therefore the flow 
points inward $\tilde B$ along this border.

So, we need to check the border 
$$
\tilde B^-=\{ (x,v),\quad v=m_0(x)+\frac{K\eps}{x}\},
$$ 
whose normal exterior vector is
$n=(m_0'(x)-\frac{K\eps}{x^2},-1)$, and it is  enough to see that
$$
<n,X>_{|\tilde B^-} <0
$$
for $X=(\eps (1+\varphi(m_0(x)+\frac{K\eps}{x})), 1+2x+\varphi(m_0(x)+\frac{K\eps}{x})(2x-1))$,
which becomes:
$$
(m'_0(x)-\frac{K\eps}{x^2})\eps [1+\varphi(m_0(x)+\frac{K\eps}{x})]-[1+2x+\varphi(m_0(x)+\frac{K\eps}{x})(2x-1)]<0 .
$$
Taylor expanding the function $\varphi$ one has that:
\begin{equation}\label{Taylorphi}
\varphi(m_0(x)+\frac{K\eps}{x})=\varphi(m_0(x))+\varphi '(m_0(x))\frac{K\eps}{x} +h(x;\eps)
\end{equation}
our condition reads:
\be\label{eq:M}
\eps \left[ m'_0(x)(1+\varphi(m_0(x)))-\varphi '(m_0(x))\frac{K}{x}(2x-1)\right] +M(x;\eps) <0
\ee
where
\begin{eqnarray*}
M(x;\eps) &=& \eps ^2 m'_0(x)\varphi '(m_0(x))\frac{K}{x}- \eps ^2 \frac{K}{x^2}[1+\varphi(m_0(x))]\\
&-&\eps^3 \frac{K^2}{x^3}\varphi '(m_0(x))
+\left( \eps m' _0(x)- \frac{K\eps ^2}{x^2}-(2x-1)\right)h(x;\eps)
\end{eqnarray*}
Using \eqref{vpm} and that, by \eqref{m0-1}, there exist $C_1$, $C_2$:
\be \label{eq:boundsmprime}
\begin{array}{rcl}
\frac{C_1}{\sqrt|x|} &\le& m'_0(x)\le \frac{C_2}{\sqrt|x|}, \ \text{for}\quad -L\le x<0\\
C_1 \sqrt{|x|} &\le & 1-m_0(x) \le C_2 \sqrt{|x| },
\end{array}
\ee
we obtain that, the $\OO(\eps)$  terms of \eqref{eq:M} can be bounded,
choosing $K$ big enough depending on $C_1$, $C_2$, and therefore on $L$:
\begin{eqnarray*}
\eps \left[ m'_0(x)(1+\varphi(m_0(x)))-\varphi '(m_0(x))\frac{K}{x}(2x-1)\right] \\
=\eps \frac{2(m' _0(x))^2 x-4K}{m'_0(x)(1-2x)x} 
\le \frac{2C_2^2-4K}{C_1}\frac{\eps}{\sqrt{|x|}}\le -2\frac{\eps}{\sqrt{|x|}}.
\end{eqnarray*}

To end the proof we need to bound the higher order terms of \eqref{eq:M}  contained in $M(x;\eps)$.
Using again \eqref{vpm} and bounds \eqref{eq:boundsmprime}, we obtain:
\begin{eqnarray*}
 |\eps ^2 m'_0(x)\varphi '(m_0(x))\frac{K}{x}| &\le& \eps ^2 \frac{ 4 K}{(1-2x)^2 x}\le 4K \eps ^{1-\frac{1}{2}\la}\frac{\eps}{\sqrt{|x|}}\\
 |\eps ^2 \frac{K}{x^2}[1+\varphi(m_0(x))]|&\le& \frac{\eps ^2 2 K}{x^2(1-2x)}\le 2K\eps ^{1-\frac{3}{2}\la} \frac{\eps}{\sqrt{|x|}}\\
 |\eps^3 \frac{K^2}{x^3}\varphi '(m_0(x))| &\le& \frac{4\eps ^3 K^2}{C_1(1-2x)^2|x|^{5/2}} \le \frac{4 K^2 \eps ^{2-2\la}}{C_1} \frac{\eps}{\sqrt{|x|}}.
 \end{eqnarray*}
 Finally, using that, for any $0<\delta <1$, there exists $C_3>0$ such that
 $$
 |\varphi''(v)|\ \le C_3 \ \text{for} \quad 0<v \le 1- \delta
 $$ 
 and using that, for $\eps $ small enough $|m_0(x)+ \frac{K\eps}{x}-1|\le \delta $ if $-L\le x\le -\eps ^{\la}$ and also \eqref{Taylorphi}, one has:
$$
\left|\left( \eps m' _0(x)- \frac{K\eps ^2}{x^2}-(2x-1)\right)h(x,\eps)\right| \le \left( \eps ^{1-\frac{1}{2}\la} C_2
+K\eps ^{2-2\la}+(2L+1)\right) C_3 K^2 \eps ^{1-\frac{3}{2}\la} \frac{\eps}{\sqrt{|x|}}.
$$
Finally, putting all these bounds together, one has that, if $\eps$ is small enough, we get
$$
|M(x,\eps)| \le \frac{1}{2}\frac{\eps}{\sqrt{|x|}}
$$
and therefore
$$
<n,X>_{\tilde B^-} \le (-2+\frac{1}{2})\frac{\eps}{\sqrt{|x|}} \ll 0.
$$
At the boundary $x=-L$ one has that $\dot x>0$ therefore the flow points inward also in this border.
 
Now, we know that any orbit entering $\tilde B$ stays in it until it reaches $x=-\eps ^\lambda$.
But, by Theorem \ref{thm:fenichel} and Remark \ref{rem:canvivar} we know that the invariant manifold $\Lambda_\eps$ at $x=-N$ is given by
$$
x=m(x;\eps) = m_0(x)+\eps m_1(x)+\OO(\eps^2)
$$
and $m_1(x)<0$. 
Therefore, adjusting the constants to have $K>Nm_1(-N)$, the manifold enters $\tilde B$ and verifies \eqref{eq:varietatconfinada} for $-L\le x\le -\eps ^\la$.

\endproof

Next step is to see that the manifold $\Lambda_\eps$ attracts all the solutions with initial conditions at points $(x_0,1)$, if $-L\le x_0\le -\eps^\la$.
Let's introduce the equation for the orbits of system \eqref{eq:fast}:
\begin{equation}\label{eq:orbits}
\frac{\eps dv}{dx}=\frac{1+2x+\varphi(v)(2x-1)}{1+\varphi(v)}
\end{equation}
Then, one has:
\begin{proposition}\label{prop:atractiogran}
Fix $0<\la<\frac{2}{3}$ and take any  point $(x_0,1)$, with  $-L\le x_0\le -\eps^\la$.
Then, the orbit of system \eqref{eq:orbits} with initial condition $v(x_0)=1$ stays exponentially close to the 
invariant manifold $v=m(x;\eps)$ in the region $x_0\le x< -\eps^{2/3}$.
\end{proposition}

\proof
To proof this proposition,  we perform the change of variables 
$ w=v-m(x;\eps)$ in equation \eqref{eq:orbits} obtaining:
\be \label{eq:orbitsw}
\eps \frac{d \, w}{d\, x} = -g(x;\eps)\varphi'(m(x;\eps))w-g(x;\eps) F(x,w;\eps)
\ee
where 
\[
F(x,w;\eps) =\varphi(m(x;\eps)+w)-\varphi(m(x;\eps))-\varphi '(m(x;\eps)) w.
\]
and where $g(x;\eps)$ is the positive function:
\[
g(x;\eps)=     \frac{-2x+1+\eps m'(x;\eps)}{1+\varphi(m(x;\eps)+w(x;\eps))}.
\]

Note that we already know the existence of the solution $w(x;\eps)=1-m(x;\eps)$ for $x_0 \le x$,  satisfying:
\[
0\le w(x;\eps)\le 1-m(x;\eps).
\]
For this reason we use the notation $g(x;\eps)$ even if this function depends on $w(x;\eps)$.

In the sequel, we will use the following expression for the function $F(x,w;\eps)$:
\begin{equation}\label{eq:F}
|F(x,w;\eps)|=  A(x;\eps)w, \quad A(x;\eps)= \int _{0}^{1}\varphi '(m(x;\eps)+ s w(x;\eps)) ds - \varphi' (m(x;\eps)) ds.
\end{equation}
It is important to stress that as $w(x;\eps)\ge 0$ and $\varphi '$ is decreasing in the considered domain, the function $A(x;\eps)$ is negative.

Clearly, the solution of \eqref{eq:orbitsw} with initial condition $w(x_0) = 1-m(x_0;\eps)$ can be written as:
$$
w(x)= e^{-\frac{1}{\eps}\int _{x_0}^{x} g(s;\eps)\varphi'(m(s;\eps))ds}\left[w(x_0)- \frac{1}{\eps}\int_{x_0}^{x}
e^{\frac{1}{\eps}\int _{x_0}^{\nu} g(s;\eps)\varphi'(m(s;\eps))ds}g(\nu;\eps) F(\nu, w(\nu;\eps))d\nu \right] .
$$

Defining $\tilde w (x;\eps) = e^{\frac{1}{\eps}\int _{x_0}^{x} g(s;\eps)\varphi'(m(s;\eps))ds} w(x;\eps)$, and using \eqref{eq:F} we obtain:
$$
|\tilde w(x;\eps)|\le  | w(x_0)|+ 
\frac{1}{\eps}\int_{x_0}^{x} |g(\nu;\eps) A(\nu;\eps) \tilde w(\nu;\eps) |d\nu
=  | w(x_0)|-
\frac{1}{\eps}\int_{x_0}^{x} g(\nu;\eps) A(\nu;\eps) |\tilde w(\nu;\eps) |d\nu .
$$

Applying Gronwall's lemma we get:
$$
|\tilde w(x;\eps)|\le  | w(x_0)|e^{- 
\frac{1}{\eps}\int_{x_0}^{x} g(\nu;\eps) A(\nu;\eps) d\nu}
$$
and therefore
$$
|w(x;\eps)|\le  | w(x_0)|e^{- 
\frac{1}{\eps}\int_{x_0}^{x} g(\nu;\eps) (A(\nu;\eps) +\varphi'(m(\nu;\eps))d\nu}=
| w(x_0)|e^{- 
\frac{1}{\eps}\int_{x_0}^{x} g(\nu;\eps) (\int _{0}^{1}\varphi '(m(\nu;\eps)+ s w(\nu;\eps)) ds) d\nu} .
$$
To bound this last expression we use the following facts:
\begin{itemize}
\item
For $x\le 0$ we have that
$
g(x;\eps )\ge \frac{1}{2}
$
\item
Given $0<\delta <1$, there exist constants $c_1$, $c_2$, such that
 for $0<v \le 1- \delta$ one has:
 $$
 c_1(1-v) \le \varphi'(v) \le c_2 (1-v).
 $$
 \end{itemize}
 
 Using \eqref{eq:varietatconfinada} and \eqref{m0-1}, one has that  $|m+sw-1| \le \delta$ and therefore we have, for $x \le -\eps ^\la$:
\begin{eqnarray*}
|w(x;\eps)|&\le & | w(x_0)|e^{- 
\frac{c_1}{2\eps}\int_{x_0}^{x}  (\int _{0}^{1}(1-m(\nu;\eps)+ s w(\nu;\eps)) ds) d\nu}
=| w(x_0)|e^{- 
\frac{c_1}{2\eps}\int_{x_0}^{x}  ((1-m(\nu;\eps)+ \frac{w(\nu;\eps)}{2} ) ) d\nu}\\
&\le & 
| w(x_0)|e^{- 
\frac{c_1}{2\eps}\int_{x_0}^{x}  (1-m(\nu;\eps) ) d\nu} 
\le  
| w(x_0)|e^{- 
\frac{c_1}{2\eps}\int_{x_0}^{x}  (1-m_0(\nu) ) d\nu}
\le | w(x_0)|e^{- 
\frac{\bar c_1}{2\eps}(|x_0|^{3/2}-|x|^{3/2})}
\end{eqnarray*}
where we have used \eqref{eq:boundsmprime}. Now, these bounds guarantee that the solution $w(x;\eps)$ exists for $x_0< x\le -\eps ^\la$ and verifies the same bounds.
\endproof

\subsubsection{Asymptotics for the Poincar\'{e} map $P_\eps$}

Fix $0<\la <2/3$.  
After Theorem \ref{thm:fenichel} and propositions \ref{blocouter}, \ref{prop:blocinner} and 
\ref{prop:atractiogran}, we can conclude that the Poincar\'{e} map $\PP_\eps$ is defined in the set 
$[-L, -\eps ^\la]\times \{\eps\}$. Moreover
\begin{equation}\label{ppepsilon}
\forall x \in [-L,-\eps ^\la], \quad \PP_\eps (x) = \eps ^{2/3} \eta _0 (0)+ O(\eps).
\end{equation}
Taking into account that,  by \eqref{eq:ppg}
$$
P^{-1}(-\eps ^{\la}) = x^-_0+\alpha ^- \eps + \beta ^- \eps ^{2\la}+O( \eps ^{1+\la})
$$ 
we have that, calling $\II= [L^-,x^-_0+\alpha ^- \eps + \beta ^- \eps ^{2\la}+O( \eps ^{1+\la})]$,
$$
P(\II) \subset [-L, -\eps ^\la]
$$ 
where $P(L^-)=-L$.

On the other hand we know that the map $\bar P$ is given by formulas \eqref{eq:ppg}.

Therefore we conclude that the map $P_\eps = \bar P\circ \PP_{\eps} \circ  P$
\begin{equation}\label{pepsilon}
\begin{array}{rcl}
P_\eps:\II \times \{y=y_0\} &\to& \Sigma ^+_{y_0} \\
(x,y_0) & \mapsto &(P_\eps(x), y_0)  
\end{array}
\end{equation}
is given by
$$
P_\eps(x)= \bar P(\eps ^{2/3}\eta _0(0)+ O(\eps)) =x^+_0 +\alpha ^+ \eps +\beta^+ (\eta_0(0))^2\eps ^{4/3} + O(\eps^{5/3}).
$$ 
Therefore, all the points in the interval   $\II$ 
are send by $P_\eps$ to an interval 
$\JJ$ which has, at most, size $\OO(\eps^{5/3})$ and it is centered at the point 
$x^+_0 +\alpha ^+ \eps +\beta^+ (\eta_0(0))^2\eps ^{4/3}$. 
Consequently, the Lipchitz constant of $P_\eps$ is, at most $O(\eps ^{5/3})$.

\subsection{The slow manifold close to $(0,1)$: smooth $\CC^{p-1}$ case}\label{sec:spc}
When the regularizing function $\varphi$ is $\CC^{p-1}$ with $p\ge 3$, the slow manifold has the same qualitative behavior explained in the previous section. In this section we will stress the main quantitative differences between the $\CC^1$ case and the general $\CC^{p-1}$ case.

The expansion of the solution 
$$
x = n(v;\eps)= n_0(v)+\eps n_1(v)+ \cdots + \OO(\eps ^n)
$$
is exactly the same  as in \eqref{expansio1} and \eqref{expansio2} but now, the local behavior of the terms in this expansion is different.
Without loss of generality we assume in this section that  $p$ is even and that $\varphi^{(p)}(1)< 0$. The case  $p$ odd is identically treated with $\varphi ^{(p)}(1)>0$.

We will have that, near $v=1$, using that
$$
\varphi (v)= 1 + \frac{\varphi ^{(p)}(1)}{p!} (v-1)^p+ \OO((v-1)^{p+1}), \ v\le 1
$$
one has
\begin{eqnarray}
n_0(v) &=& \frac{\varphi^{(p)}(1)} {4 p!}(v-1)^p + \OO((v-1)^{p+1}) \label{nop}\\
n_1(v)&=&\OO\left(\frac{1}{(1-v)^{p-1}}\right) \nonumber \\
n_2(v)&=& \OO\left(\frac{1}{(1-v)^{3p-2}}\right) \nonumber \\
n_3(v)&=&\OO\left(\frac{1}{(1-v)^{5p-3}}\right) \nonumber \\
\end{eqnarray}
in general we have:
$$
n_l(v)=\OO\left(\frac{1}{(v-1)^{(2l-1)p-l}}\right)
$$
therefore, the asymptotic expansion for $n(v;\eps)$ close to $v=1$ behaves as
$$
n(v;\eps) =  \frac{\varphi^{(p)}(1)}{4 p!}(v-1)^p  +  
\OO\left(\frac{\eps }{(v-1)^{p-1}}\right) + \OO\left(\frac{\eps ^2}{(v-1)^{3p-2}}\right)+\dots +
\OO\left(\frac{\eps ^l}{(v-1)^{(2l-1)p-l}}\right)$$
and this  expansion looses its asymptotic character for
$$
(v-1)^{2p-1} =\OO(\eps)
$$
which indicates that the invariant manifold  is close to $x
=n_0(v)$ until 
$v=1-\OO(\eps^\frac{1}{2p-1})$.
Next proposition, whose proof is completely analogous to proposition \ref{blocouter}, gives rigorously this behavior

\begin{proposition}\label{blocouterp}
Take any $0<\lambda_1<\frac{1}{2p-1}$.
Then, there exists $M>0$ big enough,  $\delta=\delta (M)>0$ small enough  and  
$\eps_0=\eps_0(M,\delta)>0$,   such that, for $0\le \eps\ \le \eps _0$, 
 any solution of system \eqref{eq:slow} which enters the set
$$
\mathbf{B^p} =  \left\{ (x,v), \ -\delta <v-1<-\eps^{\lambda_1} , \ n_0(v) \le x\le n_0(v)+ \frac{M \eps }{|v-1|^{p-1}}\right\}
$$
leaves it through the boundary $v=1-\eps ^{\la_1}$.
\end{proposition}

Then the invariant manifold $\Lambda_\eps$,   which is given by 
$$
x=n(v;\eps)=n_0(v)+ \eps n_1(v) + \OO(\eps ^2),
$$ 
with $n_1(1-\delta)>0$, enters in the domain ${\bf B^p}$ and   it stays there at least until $v^*=1-\eps ^{\lambda_1}$ satisfying:
$$
n_0(1-\eps ^{\lambda_1}) < n(1-\eps^{\lambda_1};\eps)
< n_0(1-\eps ^{\lambda_1})+ M\eps ^{1-(p-1)\lambda_1}. 
$$
As the manifold attracts exponentially any  solution beginning in $U$ (see Theorem \ref{thm:fenichel}), all the solutions 
of the system verify the same inequality.

Moreover, as $n_0(1-\eps ^{\lambda})= \frac{\varphi^{(p)}(1)}{4 p!}\eps ^{p\lambda}+ O(\eps ^{(p+1)\lambda})$
one has that, for any solution $x(v)$ beginning in $U$:
$$
x(1-\eps^{\lambda}) =\frac{\varphi^{(p)}(1)}{4 p!}\eps ^{p\lambda}+ O(\eps ^{(p+1)\lambda},\eps^{1-(p-1)\lambda})
$$
for any $0<\la \le \la _1 <1/(2p-1)$.

For  $v=1-\OO(\eps ^{1/(2p-1)})$,  $n(v;\eps)=\OO(\eps^{p/(2p-1)})$. Therefore, in this case,  we perform the change:
\begin{eqnarray*}
v=1+\eps ^{1/(2p-1)}u\\
x= \eps ^{p/(2p-1)} \eta  .
\end{eqnarray*}
The equation for the orbits \eqref{eq:orbitsc1} in these new variables is:
\begin{equation}\label{eq:manifoldinnerp}
\frac{d \eta}{du}=\frac{\eps^{p/(2p-1)}(1+\varphi(1+\eps ^{1/(2p-1)}u))}{\left(1+2\eps^{p/(2p-1)}\eta(u)+\varphi(1+\eps ^{1/(2p-1)}u)(2\eps^{p/(2p-1)}\eta(u)-1)\right)}.
\end{equation}
Calling $\mu = \eps ^{1/(2p-1)}$, one can write this equation as:
\begin{equation}\label{eq:manifoldinner}
\frac{d \eta}{du}=\mu ^p\frac{(1+\varphi(1+\mu u))}{\left(1+2\mu^p\eta(u)+\varphi(1+\mu u)(2\mu ^p\eta(u)-1)\right)},
\end{equation}
and we need to study the extension of a solution of this equation $\eta(u;\eps)$, with initial condition  $\eta (u^*;\eps)$, 
with $u^*= \frac{v^*-1}{\eps ^{1/(2p-1)}}=-\eps ^{\lambda_2-1/(2p-1)}$, for 
$0<\la _2\le \la _1 <1/(2p-1)$,
verifying
\begin{equation}\label{eq:initialconditioninner}
|\eps^{p/(2p-1)}\eta (u^*;\eps)-n_0(v^*)| \le M\eps ^{1-(p-1)\lambda_2}
\end{equation}
where $v^*= 1+\eps ^{1/(2p-1)}u^*=1-\eps^{\lambda_2}$, to the domain:
\begin{equation}\label{dominiiiner}
u^*\le u \le 0.
\end{equation}

Expanding the solution $\eta(u;\eps)$ of equation \eqref{eq:manifoldinner} in powers of $\mu=\eps ^{1/(2p-1)}$, 
one can see that  $\eta _0$ is the solution of the equation:
\begin{eqnarray}\label{eq:edoeta0p}
\eta _0'= \frac{d \eta_0}{d u}= \frac{2}{4\eta _0 -\frac{\varphi^{(p)}(1)}{p!}u^p}.
\end{eqnarray}

We need to study equation \eqref{eq:edoeta0p} to obtain an analogous result as the one given in \cite{MischenkoR80} for 
equation \eqref{eq:edoeta0}. 
With  the changes of variables:
$\bar \eta = \alpha \eta, \quad \bar u = \beta u$, where 
$$
\alpha =2^{\frac{p-2}{2p-1}}\left(-\frac{\varphi ^{(p)}(1)}{p!}\right)^{\frac{1}{2p-1}} , 
\quad \beta = 2^{-\frac{3}{2p-1}}\left(-\frac{\varphi ^{(p)}(1)}{p!}\right)^{\frac{2}{2p-1}}
$$
it becomes
\begin{equation}\label{eq:etaobarp}
\frac{d \, \bar \eta}{d \,\bar u}= \frac{1}{\bar \eta + \bar u ^{p}}.
\end{equation}

\begin{proposition}\label{prop:asymptoticsp}
Equation \eqref{eq:etaobarp} has a unique solution $\bar \eta _0(\bar u)$ verifying:
\begin{equation}
\bar \eta _0(\bar u) = -\bar u^p - \frac{1}{p}\frac{1}{ \bar u^{p-1}} + O( \frac{1}{\bar u^{3p-2}}), \quad  \bar u \to -\infty
\end{equation}
Moreover, there exists a constant $K>\frac{1}{p}$ such that:
\begin{equation}
-\bar u^p <\bar \eta_0 (\bar u) < -\bar u ^p -\frac{K}{\bar u^{p-1}} , \quad \bar u\le 0 .
\end{equation}
\end{proposition}

\begin{figure}
\begin{center}
\includegraphics[width=10cm]{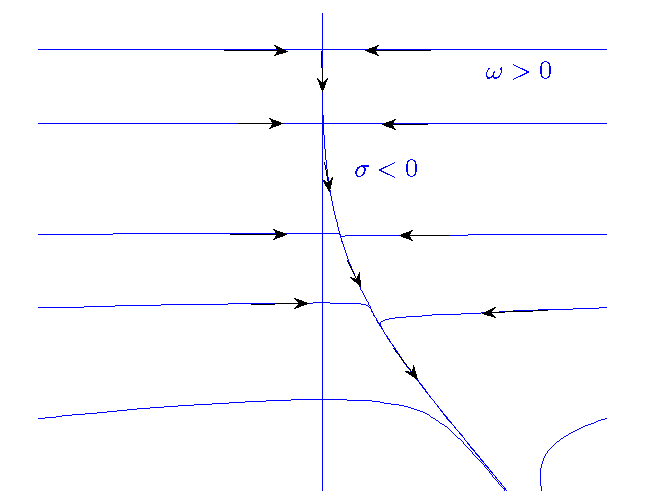}
\end{center}
\caption{The central invariant manifold of system \eqref{central}.}
\label{riccatiinfinit}
\end{figure}

\begin{proof}
To proof this proposition we consider the vector field whose orbits are solutions of \eqref{eq:etaobarp}:
\begin{eqnarray*}
 \dot{\bar \eta} &=&1\\
\dot{\bar u}&=&\bar \eta +\bar u^p
\end{eqnarray*}
for $u\le 0$ and $\eta \le 0$.

As the curve $\bar \eta + \bar u^p=0$ is a isocline of slope zero, we will see that the region
$$
\BB=\{ (\bar u, \bar \eta), \ - \bar u^p\le  \bar \eta\le -\bar u^p- \frac{K}{\bar u^{p-1}} , \ \bar u<0\}
$$
is an isolating block in the region $\bar u<0$ as $\bar u \to -\infty$. 
The boundary 
$$
\BB^-=\{ (\bar u, \bar \eta), \   \bar \eta= -\bar u^p, \ \bar u<0\}
$$
is positively invariant because the vector field is given by $(1,0)$ and it points inwards $\BB$.
To see that $\BB^+$ is also positively invariant  we take the  exterior normal vector
$(1, p\bar u ^{p-1}+K(1-p)\bar u^p)$
and we need to check that
$$
<(1,\bar \eta + \bar u^p), (1, p\bar u ^{p-1}+\frac{K(1-p)}{\bar u^p})>_{|\BB^+} <0 ,
$$
that is:
$$
1-Kp-\frac{K^2 (1-p)}{\bar u ^{2p-1}}<0 .
$$
As we are assuming that $p$ is even,  the term $\frac{K^2 (1-p)}{\bar u ^{2p-1}}$ is positive, therefore the expression above is negative if we take $K >\frac{1}{p}$.

To prove the existence of the solution $\bar \eta _0(\bar u)$ we perform the changes:
$$
w= \bar \eta +\bar u ^p, \ \mbox{and} \ \sigma = \frac{1}{\bar u}
$$ 
obtaining:
$$
\begin{array}{rcl}
w' &=& 1+ \frac{p}{\sigma ^{p-1}} w\\
\sigma ' &=& -\sigma ^2 w
\end{array}
$$
for $w\ge 0$ and $\sigma \le 0$.
After a change of time (multiplying the equations by $-\sigma ^{p-1}$) one obtains an equivalent system whose orbits are the same:
\begin{equation}\label{central}
\begin{array}{rcl}
\frac{dw}{d \tau} &=& -p w -\sigma ^{p-1} \\
\frac{d \sigma}{d \tau} &=& \sigma ^{p+1} w
\end{array}
\end{equation}
whose equilibrium point $(0,0)$ corresponds to the null-cline $\bar \eta + \bar u ^p=0$ at $\bar u = -\infty$.
This equilibrium point is partially hyperbolic and the linearization of the vector field at $(0,0)$ is given by
$$
\begin{array}{rcl}
\frac{dw}{d \tau} &=& -p w  \\
\frac{d \sigma}{d \tau} &=& 0 .
\end{array}
$$
whose matrix has eigenvectors  $(1,0)$ and $(0,1)$ associated to the eigenvalues $-p$ and $0$.
One can apply to this point the Central Manifold Theorem \cite{Carr81} and we know that there exists a local invariant manifold which can be described by $\Lambda _c=\{ (w,\sigma), \ w= g(\sigma)\}$ with $g(\sigma)$ a $\CC^\infty$ function, in a neighborhood of $\sigma =0$ with $g(0)=g'(0)=0$ and which verifies:
$$
0= p g(\sigma)+\sigma^{p-1}+g(\sigma)g'(\sigma) \sigma^{p+1}, \ \forall \sigma
$$
which gives:
$$
g(\sigma)= -\frac{1}{p}\sigma^{p-1}+ \OO(\sigma^{3p-2}).
$$
On the central manifold $\Lambda _c$
we have that
$$
\sigma'=g(\sigma)\sigma^{p+1}= -\frac{1}{p} \sigma^{2p} + \OO(\sigma^{4p-1}).
$$
We see that, for $\sigma<0$, the central manifold $\Lambda_c$ is overflowing ($\sigma'<0$) and therefore it is unique \cite{Sijbrand85}. 
We conclude that  there is a unique solution $(w_0(\tau), \sigma_0(\tau))$ in $\sigma<0$ such that 
$$
(w_0(\tau), \sigma_0(\tau)) \to (0,0) \ \mbox {as} \ \tau \to -\infty .
$$
The situation is summarized in figure \ref{riccatiinfinit}.
Going back to the original variables $(\bar \eta, \bar u)$, we get that 
the unique central manifold enters the region 
$\{ (\bar \eta, \bar u), \ \bar \eta +\bar u^p >0, \ \bar u <0\}$.

Moreover, it has the asymptotic expression:
$$
\bar \eta_0 = -\bar u ^p -\frac{1}{p} \bar u^{1-p}+ \OO(\bar u ^{2-3p})
$$
but this solution for $\bar u$ near $-\infty$ is inside the block $\BB$, and we have seen that this block is positively invariant for the flow  if 
$K>\frac{1}{p}$. 
Therefore, if $K$ is big enough, the central manifold remains   $\BB$ until   $\bar u \le 0$. 
\end{proof}

\begin{remark}
If $p\ge 3$ is odd, the equivalent equation to \eqref{eq:etaobarp} is:
$$
\frac{d \, \bar \eta}{d \,\bar u}= \frac{1}{\bar \eta - \bar u ^{p}},
$$
and the block is given by
$$
\BB=\{ (\bar u, \bar \eta), \  \bar u^p\le  \bar \eta\le \bar u^p- \frac{K}{\bar u^{p+1}} , \ \bar u<0\}
$$
and is isolating for $K>\frac{1}{p}$.
The obtained solution verifies:
$$
\bar \eta (\bar u) = \bar u^p + \frac{1}{p}\frac{1}{ \bar u^{p-1}} + O( \frac{1}{\bar u^{3p-2}}), \quad  \bar u \to -\infty .
$$
\end{remark}

\bigskip

From proposition \ref{prop:asymptoticsp}, and using that:
$
\frac{2\al}{\beta}=\frac{4}{\al}=-\frac{\varphi^{(p)}(1)}{p! \beta ^p}$, we obtain that
$$
\frac{4\beta^p}{\al}=2 \al\beta^{p-1}=\frac{\varphi^{(p)}(1)}{p! },
$$
going back to our variables one has that $\eta_0(u)$ verifies:
\begin{eqnarray*}
\eta _0 ( u) &=& 
\frac{\varphi ^{(p)}(1)}{4\,p!} u^p + \frac{2(p-1)!}{\varphi ^{(p)}(1) \, }u^{1-p} +  O(u^{2-3p}),  u \to -\infty\\
\frac{\varphi ^{(p)}(1)}{4\,p!} u^p 
&<& \eta_0(u) <\frac{\varphi ^{(p)}(1)}{4\, p!}u^p +\frac{2K(p-1)!}{\varphi^{(p)}(1)}u^{1-p}, \quad u\le 0
\end{eqnarray*}
with $k>\frac{1}{p}$.
As a consequence of this expansion and the asymptotic expansion of $n_0(v)$ near $v=1$ given in \eqref{nop}, one has that there exist constants $K
_1$, $K_2$, such that
$$
|\eps ^{p/(2p-1)}\eta _0(u^*)-n_0(v^*)| \le K_1 \eps ^{\lambda_2} + K_2 \eps ^{1-\lambda_2(p-1)},
$$
and therefore, by \eqref{eq:initialconditioninner} one has, as in \eqref{initialconditionmatching}:
\begin{equation}\label{initialconditionmatchingp}
|\eps ^{p/(2p-1)}\eta (u^*;\eps)-\eps ^{p/(2p-1)}\eta _0(u^*)| \le M \eps ^{1-\lambda_2}+K_1 \eps ^{\lambda_2} + K_2 \eps ^{1-\lambda_2},
\end{equation}
On the other hand, if one consider the next term in the expansion of $\eta (u,\eps)$, one has:
$$
\eta (u;\eps)=\eta _0(u)+\mu \eta _1(u)+O(\mu^2)
$$ 
where $\eta _1(u)$ is the solution of the equation:
$$
\eta '_1(u)= -\frac{8}{\left(4\eta _0(u)-\frac{\varphi^{(p)}(1)}{(p)!} u^{p}\right)^2} \eta _1 + \frac{2 \frac{\varphi^{(p+1)}(1)}{(p+1)!} u^{p+1}}{\left(4\eta _0(u)-2\frac{\varphi^{(p)}(1)}{(p)!} u^{p}\right)^2} 
$$
and one can see that the solution $\eta _1$ near $-\infty$ behaves as:
$$
\eta _1 (u) \simeq \frac{\varphi^{(p+1)}(1)}{4(p+1)!}u^{p+1} , \quad u \to -\infty
$$
and one can see the next proposition, whose proof is analogous to Proposition \ref{prop:blocinner}:

\begin{proposition}\label{prop:blocinnerp}
There exists $u_0>0$, and $K>0$, such that for $0<\eps  \le \eps _0$, the set
$$
{\bf B_2^p}=\{ (u,\eta), \ u^*\le u \le 0, \quad |\eta (u) -\eta _0(u)| \le \bar K\mu M(u)\}
$$
where $M(u)$ is the function defined by:
$$
M(u)= \left\{ 
\begin{array}{ll}
-u^{p+1} & -\infty \le u \le -u_0 <0 \\
u_0^{p+1} &-u_0 \le u \le 0
\end{array}\right.
$$
and $\mu = \eps ^{1/(2p-1)}$, is an isolating block for equation \eqref{eq:manifoldinnerp}.
\end{proposition}
Once we have that ${\bf B_2^p}$ is an isolating block and that, by \eqref{initialconditionmatchingp}, the solution $x(v,\eps)$ enters in it at $v=v^*$ we have that our solution crosses the line $v=1$ at a point verifying:
$$
x(1;\eps)= \mu ^p\eta_0(0)+ O(\mu ^{p+1}) = \eps ^{p/(2p-1)} \eta _0(0)+ O(\eps^{(p+1)/(2p-1)})
$$

\subsubsection{Exponential attraction of the whole neighborhood of the fold}

As we did in section \ref{sec:exponential} we now see that the invariant  manifold attracts all the points in the section 
$\{ (x,v), \ v=1, \ -L\le x\le -\varepsilon^\la\}$ for $0<\lambda <\frac{p}{2p-1}$. 
We point out the main differences in this case.
The expansion
$$
m(x;\eps)= m_0(x)+\eps m_1(x)+ \eps ^2 m_2(x)+\dots
$$
behaves now as
$$
m_0(x)=\varphi ^{-1}(\frac{1+2x}{1-2x})=\varphi ^{-1}(1+2x+4x^2 +\dots)
$$
on the other hand $\varphi (v)= 1+\frac{\varphi^{(p)}(1)}{p!}(v-1)^p+\OO(v-1)^{p+1}$ and therefore we obtain that:
$$
m_0(x) =1 +\OO( |x|^{\frac{1}{p}}), 
\quad m_1(x)=\OO(|x|^{\frac{1-p}{p}})
$$
Looking at these terms one can guess that the asymptotic expansion for $m(x;\eps)$ will fail at 
$x=\OO(\eps ^{\frac{p}{2p-1}})$.
\begin{proposition}\label{varietatconfinadap}
Consider $-L<-N<0$ and $0<\la <\frac{p}{2p-1}$.
Then, there exists $K>0$ and $\varepsilon_0>0$, such that, if $0\le \varepsilon<\varepsilon _0$ 
the invariant manifold $v=m(x;\eps)$ verifies, for $-L\le x\le -\eps ^{\la}$:
\be\label{eq:varietatconfinadap}
m_0(x)+\frac{\eps K}{x ^{\frac{2p-2}{p}}}\le m(x;\eps)\le m_0(x)
\ee
\end{proposition}

\proof
The proof follows the same lines that proposition \ref{prop:atractiogran1}, proving that the set
\be\label{blocconfinadap}
\tilde B=\{ (x,v), \ -L\le x\le -\eps ^\la, \quad m_0(x)+\frac{\eps K}{x ^{\frac{2p-2}{p}}}\le m(x;\eps)\le m_0(x)
\}
\ee
is positively invariant for system \eqref{eq:slow}.
Now, instead of \eqref{eq:boundsmprime}, we will use  \eqref{vpm}, which gives that  there exist $C_1$, $C_2$:
\be \label{eq:boundsmprimep}
\frac{C_1}{|x|^{\frac{1}{p}}} \le m'_0(x)\le \frac{C_2}{|x|^{\frac{1}{p}}}, \ \text{for}\quad -L\le x<0.
\ee
\endproof

Next step is to see that the manifold $\Lambda_\eps$ attracts all the solutions with initial conditions at points $(x_0,1)$, if $-L\le x_0\le -\eps^\la$.

\begin{proposition}\label{prop:atractiogranp}
Fix $0<\la<\frac{p}{2p-1}$ and take any  point $(x_0,1)$, with  $-L\le x_0\le -\eps^\la$.
Then, the orbit of system \eqref{eq:orbits} with initial condition $v(x_0)=1$ stays exponentially close to the 
invariant manifold $v=m(x;\eps)$ in the region $x_0\le x< -\eps^{\frac{p}{2p-1}}$.
\end{proposition}
\proof
The  proof of  this proposition is also similar to proposition \ref{prop:atractiogran}, 
performing  the change of variables 
$ w=v-m(x;\eps)$ in equation \eqref{eq:orbitsc1}  and using Gronwall lemma to bound $w$ we arrive to:
$$
|w(x;\eps)|\le | w(x_0)|e^{- 
\frac{1}{\eps}\int_{x_0}^{x} g(\nu;\eps) (\int _{0}^{1}\varphi '(m(\nu;\eps)+ s w(\nu;\eps)) ds) d\nu}
$$
To bound this last expression we use the following facts:
\begin{itemize}
\item
For $x\le 0$ we have that
$
g(x;\eps )\ge \frac{1}{2}
$
\item
Given $0<\delta <1$, there exist constants $c_1$, $c_2$, such that
 for $|v-1| \le \delta$ one has:
 $$
 c_1(1-v)^{p-1} \le \varphi'(v) \le c_2 (1-v)^{p-1}
 $$
 \end{itemize}
Obtaining:
\begin{eqnarray*}
|w(x;\eps)|&\le & | w(x_0)|e^{- 
\frac{c_1}{2\eps}\int_{x_0}^{x}  (\int _{0}^{1}(1-m(\nu;\eps)+ s w(\nu;\eps))^{p-1} ds) d\nu}
=| w(x_0)|e^{- 
\frac{c_1}{2\eps}\int_{x_0}^{x}  ((1-m(\nu;\eps)+ \frac{w(\nu;\eps)}{2} ) )^{p-1} d\nu}\\
&\le & 
| w(x_0)|e^{- 
\frac{c_1}{2\eps}\int_{x_0}^{x}  (1-m(\nu;\eps) )^{p-1} d\nu} 
\le  
| w(x_0)|e^{- 
\frac{c_1}{2\eps}\int_{x_0}^{x}  (1-m_0(\nu) )^{p-1} d\nu}
| w(x_0)|e^{- 
\frac{\bar c_1}{2\eps}(|x_0|^{\frac{2p-1}{p}}-|x|^{\frac{2p-1}{p}})}
\end{eqnarray*}
\endproof
and then, if $x_0< x \le -\eps ^{\frac{p}{2p-1}}$ the orbits gets exponentially close to the invariant manifold.

\subsubsection{Asymptotics for the Poincar\'{e} map $P_\eps$}

Fix $0<\la <{\frac{p}{2p-1}}$.  
After Theorem \ref{thm:fenichel} and propositions \ref{blocouterp}, \ref{prop:blocinnerp} and 
\ref{prop:atractiogranp}, 
we can conclude that the Poincar\'{e} map $\PP_\eps$ is defined in the set 
$[-L, -\eps ^\la]\times \{\eps\}$. Moreover
\begin{equation}\label{ppepsilonp}
\forall x \in [-L,-\eps ^\la], \quad \PP_\eps (x) =  \eps ^{ \frac{p}{2p-1} } \eta _0(0)+ \OO(\eps ^{\frac{p+1}{2p-1}}).
\end{equation}
Taking into account that,  by \eqref{eq:ppg}
$$
P^{-1}(-\eps ^{\la}) = x^-+\alpha ^- \eps + \beta ^- \eps ^{2\la}+O( \eps ^{1+\la})
$$ 
we have that 
$$
P(\II)) \subset [-L, -\eps ^\la]
$$ 
where $\II=[L^-,x^-+\alpha ^- \eps + \beta ^- \eps ^{2\la}+\OO( \eps ^{1+\la})]$ and  $L^- =P^{-1}(-L)$.

On the other hand we know that the map $\bar P$ is given by formulas \eqref{eq:ppg}.

Therefore we conclude that the map $P_\eps = \bar P\circ \PP_{\eps} \circ  P$
\begin{equation}\label{pepsilonp}
\begin{array}{rcl}
P_\eps: \II \times \{y=y_0\} 
&\to& \Sigma ^+_{y_0} \\
(x,y_0) & \mapsto &(P_\eps(x), y_0)  
\end{array}
\end{equation}
is given by
$$
P_\eps(x)= \bar P(\eps ^{\frac{p}{2p-1}}\eta _0(0)+ \OO(\eps ^{\frac{p+1}{2p-1}}) 
=x^+ +\alpha ^+ \eps +\beta^+ (\eta_0(0))^2\eps ^{\frac{2p}{2p-1}} 
+ \OO(\eps^{\frac{3p-1}{2p-1}}, \eps^{\frac{p(p+1)}{(2p-1)^2}}).
$$ 
Therefore, all the points in the set  $\II\times \{y_0\}$ 
are send by $P_\eps$ to a set 
$\JJ\times \{y_0\}$ and the interval $\JJ$  has, at most, size $\OO(\eps^{\frac{3p-1}{2p-1}},  
\eps^{\frac{p(p+1)}{(2p-1)^2}})$ and it is centered at the point 
$x^+ +\alpha ^+ \eps +\beta^+ (\eta_0(0))^2\eps ^{\frac{2p}{2p-1}}$. 

Consequently, the Lipchitz constant of $P_\eps$ is, at most $\OO(\eps^{\frac{3p-1}{2p-1}},  
\eps^{\frac{p(p+1)}{(2p-1)^2}})$.

\begin{remark}

The results of Theorem \ref{thm:main} and Proposition \ref{prop:flowtangency} lead to two facts. 
On one hand we obtain that the Poincar\'{e} map $P_\eps$ has a domain which includes a region at distance $\OO(\eps ^\lambda)$ to the stable 
pseudoseparatrix $W^s_+(0,0)$ of the  fold, this is an improvement of previous results where one only needs to control the solutions, 
and therefore the Poincar\'{e} map, up to finite distance to the fold.
On the other hand, we only obtain that the Lipchitz constant of this map is of order 
$\OO(\eps^{\frac{3p-1}{2p-1}}, \eps^{\frac{p(p+1)}{(2p-1)^2}})$. In fact, one can see that this Lipchitz constant is exponentially 
small with respect to $\eps$ (see \cite{KrupaS01, Bonet87}) but this is not necessary for our purposes. 
The method to obtain this more accurate result, consists in applying the results of propositions \ref{blocouter} and \ref{prop:blocinner} only 
to follow the evolution of the Fenichel manifold to show that it intersects  the section $v=1$ in a point $(x(\eps),1)$ given by 
$\PP_\eps (x)$ in \eqref{ppepsilon}. 
Once we know the evolution of this invariant manifold $\Lambda_\eps$, one can show, studying the variational equations around it, 
that all the orbits begining at $(1,x)$ with $x\le -\eps ^\la$, evolve exponentially close to it. 
Nevertheless, in our case, the only needed result is the fact that all these orbits arrive to the section $v=1$ at a point 
which is``on the left" of the unstable pseudoseparatrix  $W^u_+(0,0)$ of the fold point, and these accurate quantitative results are not necessary.
\end{remark}

\subsection{The general fold}\label{sec:generalfold}

In the previous sections we have rigorously computed the Poincar\'{e} map $P_\eps$ on the sections $\Sigma^{\pm}_{y_0}$ as a composition of three maps:
$$
P_\eps = P\circ \PP_\eps \circ \bar P
$$
The maps $P$ and $\bar P$ were studied for a generic vector field $\X$ having a tangency point at  $(0,0)$ in Proposition \ref{prop:flowtangency} 
giving formulas \eqref{eq:ppg}, but the singular map $\PP_\eps$ was computed using singular perturbation 
theory in a simplified vector $Z=(\X,\Y)$ in \eqref{def:X}, \eqref{def:Y}, coming from a normal form in \cite{GuardiaST11}. 
Nevertheless, as our method needs differentiability  properties, we can not claim that the results obtained  are automatically valid for any Filipov 
vector field with a regular-fold visible point. For this reason, in this section we will consider the case of  a general vector field and we will 
point out the main technicalities to obtain  the same result as in  \eqref{ppepsilon}.

So, let as assume that we have the non smooth system \eqref{def:Filippov}, and we assume that $\X$ has a visible fold at $(0,0)$ and $\Y$ is pointing 
towards $\Sigma$. Assume also that conditions \eqref{generalform},\eqref{cond:visiblefold}, \eqref{cond:visiblefold1} are verified.
The first simplification of the vector field $Z$ is provided by the classical flow-box theorem applied to the vector $\Y$. 
Applying the change of variables to  both vector fields defining  $Z$, we obtain:
\begin{proposition}
There exists a smooth change of variables $(x,y) = \hat \psi ( \hat x, \hat y)$, where $\hat \psi: U\subset \RR^2 \to \RR^2$ verifying 
$ \hat \psi ( \hat x,0) =( \hat x, 0)$,
such that, if we call 
$ \hat Z (\hat x, \hat y)=  \hat \psi ^* Z (\hat x, \hat y)= (D \hat \psi (\hat x, \hat y))^{-1} Z\circ  \hat \psi (\hat x, \hat y)$ 
to  the transformed vector field, one has $ \hat Z=( \hat {\X},  \hat {\Y})$, and
\begin{itemize}
\item
$
 \hat {\Y} = (0,1)^t
$
\item
$
 \hat {\X} = (1+ O_1( \hat x, \hat y), 2 \hat x+ \hat b \hat y + O_2( \hat x, \hat y))^t
$, and  $O_2(\hat x,0)=0$.
\end{itemize}
\end{proposition}
\proof
The first part of the proof consists in applying the flow-box theorem to the vector field $\Y$. This theorem provides a smooth change of variables 
$(x,y)= \psi (\tilde x,\tilde y)$, where 
$\psi=(\psi _1,\psi_2)$, that transforms the vector field $\Y$ into  $\tilde {\Y} =(0,1)^t$. 
One can also ask the function $\psi $ to leave invariant a transversal manifold of the flow, that we choose to be $\Sigma$. 
Therefore this map verifies $\psi(\tilde x,0)=(\tilde x,0)$ and, consequently,  
$\frac{\partial \psi _1}{\partial \tilde x}(\tilde x, 0)=1 $, 
and $\frac{\partial \psi _2}{\partial \tilde x}(\tilde x, 0)=0 $.
Moreover, as 
$$
D\psi (0,0) \left(\begin{array}{c}0\\1\end{array}\right)=
\left(\begin{array}{c}\Y_1(0,0)\\ \Y_2(0,0)\end{array}\right)
$$
one has that $\frac{\partial \psi _2}{\partial \tilde y}(0, 0)=\Y_2(0,0)>0 $.
Now using that $D\psi (\tilde x,0) \tilde {\X} (\tilde x,0)= \X (\tilde x, 0)$, one obtains that 
$$
\tilde {\X}  (0,0)=\left(\begin{array}{c}c\\0\end{array}\right),
$$
with $c=\X_1(0,0)\neq 0$.
Moreover, 
$$
\tilde {\X_2} (\tilde x,0)=(\frac{\partial \psi _2}{\partial \tilde y} )^{-1}  (\tilde x,0)\X_2 (\tilde x,0)
$$
therefore the tangency  at $(0,0)$ is preserved and visible.
Once we have applied the flow box theorem, the new vector field $\tilde {\X}$ has the form
$$
\tilde {\X}=\left (\begin{array}{c} c+O_1(\tilde x,\tilde y)\\ a\tilde x+b\tilde y +O_2(\tilde x,\tilde y)\end{array}\right), 
$$
where $a= \frac{\partial _x \X_2(0,0)}{\Y_2(0,0)}>0$ and $c=\X_1(0,0)\neq 0$.
Now, the change of variables and time:
$$
\bar x = \frac{a}{2} \tilde x, \ \bar y = \frac{a c}{2} \tilde y , \ \tau = \frac{a c}{2} t
$$
transforms the vector field $\tilde Z$ into $\bar Z$ with $\bar {\Y}= \tilde {\Y}$ and:
$$
\bar {\X}=\left (\begin{array}{c}1+O_1(\bar x,\bar y)\\ 2\bar x+\bar b\bar y +O_2(\bar x,\bar y)\end{array}\right).
$$

To perform the last change, we observe that the second order terms in the second component of $\bar {\X}$ can be separated:
$$
O_2(\bar x,\bar y)= f_2(\bar x) + g_2(\bar x,\bar y), \quad g_2(\bar x,0)=0
$$
then, our last change is
$$
\hat x = \bar x + \frac{1}{2}f_2(\bar x).
$$
This change is well defined in a neighborhood of zero and leaves the vector field $\bar {\Y}$ invariant but changes $\bar{\X}$ into:
$$
\bar {\X}=\left(\begin{array}{c}1+O_1(\hat x,\hat y)\\ 2\hat x+\bar b\hat y +O_2(\hat x,\hat y)\end{array}\right), 
$$
but the term $O_2(\hat x,\hat y)$ vanishes at $y=0$ for any value of $\hat x$.
\endproof

This proposition  allows us to assume that we have a Filippov vector field $Z =(\X,\Y)$ where:
\begin{equation}\label{def:Xg}
\X(x,y)=\left(\begin{array}{l}
        1 + f_1(x,y)\\
        2x + by +f_2(x,y)
       \end{array}\right)
\end{equation}
where $f_i(x,y)=O_i(x,y)$ and $f_2(x,0)=0$, and
\begin{equation}\label{def:Yg}
\Y(x,y)=\left(\begin{array}{l}
        0\\
        1
       \end{array}\right)
\end{equation}
The system given by $\X$ has a visible fold at $(0,0)$ and $\Y$ is regular at this point. Therefore $(0,0)$ is a fold-regular point for $Z$. Moreover, it verifies  conditions \eqref{generalform},\eqref{cond:visiblefold}, \eqref{cond:visiblefold1}.

The regularized system \eqref{eq:regularized}  will be in the general case:
\begin{equation}\label{eq:regularizedg}
\begin{array}{rcl}
\dot x &=& 
\frac{1}{2}(1+\varphi (\frac{y}{\eps}))(1+ f_1(x,y))\\
\dot y &=& 
\frac{1+2x + by +f_2(x,y)}{2} +\frac{1}{2}\varphi (\frac{y}{\eps})(2x+by-1 + f_2(x,y)),
\end{array}
\end{equation}
and, in the variable $v=\frac{y}{\eps}$ we obtain:

\begin{equation}\label{eq:fastg}
\begin{array}{rcl}
\dot x &=& 
\frac{1+\varphi (v)}{2}(1+f_1(x, \eps v))\\
\eps \dot v &=&
 \frac{1+2x }{2} +\frac{1}{2}\varphi (v)(2x-1 ) +
\frac{1+\varphi(v)}{2}(b\eps v+ f_2(x,\eps v)).
\end{array}
\end{equation}

Observe that the slow system for $\eps =0$ is given by:
\begin{equation}\label{eq:fastg0}
\begin{array}{lcr}
\dot x &=& \frac{1+\varphi (v)}{2}(1+f_1(x, 0))\\
0 &=& \frac{1+2x }{2} +\frac{1}{2}\varphi (v)(2x-1 )
\end{array}
\end{equation}
and therefore the slow manifold $\Lambda_0$ is given in the general case by the same equation  
\eqref{SM} and the $DZ_0$ (see \eqref{eq:fast}) is exactly given by \eqref{matriunhm}.
Consequently it has the same hyperbolicity properties and Fenichel theorem \ref{thm:fenichel} can also be applied in the general case giving the existence of the invariant manifold given by $\Lambda _\eps =\{ (x,v), \ v=m(x;\eps)\}$ and also by $\Lambda _\eps =\{ (x,v), \ x=n(v;\eps)\}$ in the corresponding domains.

To study the invariant manifold near $(0,1)$ we proceed as we did in section \ref{smct01} looking for the equation of the orbits of 
$x=n(v;\eps)=n_0(v)+\eps n_1(v) +\dots$ as  Remark \ref{rem:canvivar} also applies here. 
We know that $n_0(v)$ is given by \eqref{expansio1} and easy computations give that
$$
n_1(v)= \frac{1}{2} \left( \frac{1+f_1 (n_0(v),0)}{n'_0(v)}-bv-\frac{\partial f_2}{\partial y}( n_0(v),0)v \right).
$$
Even if, in the general case,  the term $n_1(v)$ is different from \eqref{expansio2}, the behavior near $v=1$ is the same as in \eqref{n1}.
Therefore the behavior of the slow manifold near $v=1$ is also given in 
\eqref{no}, \eqref{n1} and one can easily prove  proposition \ref{blocouter} in the general fold case. 
The only thing to bear in mind is that even if $x=n_0(v)$ is no longer a isocline of zero slope, the flow in ${\bf B^-}$ also points inward $\bf{B}$. 
Moreover, to ensure that the Fenichel manifold not only enters in the block ${\bf B}$ when $v=1-\delta $
but exits it for $v=1-\eps ^{\la}$, $0<\la<\frac{1}{3}$, it is enough to see that $n(1-\delta ;\eps )>\bar n(1-\delta ;\eps )$ where $x=
\bar n(v;\eps )$ is the expression of the isocline of slope zero given by:
$$
\frac{1+2x +b\eps v+ f_2(x,\eps v)}{2} +\frac{1}{2}\varphi (v)(2x-1 +b\eps v+ f_2(x,\eps v))=0.
$$
To see this, we observe that 
$$
\bar n(v;\eps) = n_0(v)+\eps \bar n_1(v) + \OO(\eps ^2)
$$ 
with $\bar n_1(v)=-\frac{v}{2}(b+\frac{\partial f_2}{\partial y}(n_0(v),0))$, therefore:
$$
n_1(v)-\bar n_1(v)= \frac{1}{2}\frac{ 1+f_1(n_0(v),0)  }{n'_0(v)}.
$$
Now, using that $f_1(x,y)=\OO(x,y)$ and that $n_0(v)= \frac{\varphi''(1)}{8}(v-1)^2+ \OO((v-1)^3)$ near $v=1$, in a neighborhood of $(0,1)$ and that $n'_0(v)>0$ (see \eqref{expansio1}) we have that
$$
n_1(v)-\bar n_1(v)>0
$$
and  then $n(v;\eps)>\bar n(v;\eps)$.

Therefore the Fenichel manifold enters the region $\dot v>0$ and can not leave it. Also $n_1(v)>0$, and the Fenichel manifold enters inside the block ${\bf B}$ by $v=1-\delta$ and exits it at $v=1-\eps ^\la$, with $0<\la<1/3$.

When $v=1 -\OO(\eps ^{1/3})$ we proceed as usual, and the change \eqref{canviinner} transforms equations \eqref{eq:fastg} into:
\begin{equation}\label{eq:fastginner}
\begin{array}{rcl}
\eps^{-\frac{1}{3}}\dot \eta &=& 
\frac{1+\varphi (1+\eps ^{\frac{1}{3}} u )}{2}(1+f_1(\eps ^{\frac{2}{3}}\eta , \eps (1+\eps ^{\frac{1}{3}} u )))\\
\eps^{\frac{1}{3}} \dot u &=&
 \frac{1+2\eps ^{\frac{2}{3}}\eta  }{2} +\frac{1}{2}\varphi (1+\eps ^{\frac{1}{3}} u )(2\eps ^{\frac{2}{3}}\eta -1 ) +
\frac{1+\varphi(1+\eps ^{\frac{1}{3}} u )}{2}(b\eps (1+\eps ^{\frac{1}{3}} u )+ f_2(\eps ^{\frac{2}{3}}\eta ,\eps (1+\eps ^{\frac{1}{3}} u ))).
\end{array}
\end{equation}
The equation for the orbits calling $\mu = \eps ^{\frac{1}{3}}$, becomes:
\begin{eqnarray}
\frac{d \eta}{d u}=\frac{\mu ^2
(1+\varphi (1+\mu u ))(1+f_1(\mu^2\eta , \mu^3 (1+\mu u )))}
{(1+2\mu ^2\eta  ) +\varphi (1+\mu u )(2\mu ^2\eta -1 ) +
(1+\varphi(1+\mu  u ))(b\mu ^3 (1+\mu u )+ f_2(\mu^2\eta ,\mu^3 (1+\mu u )))}
\end{eqnarray}

Expanding $\eta (u) = \eta _0(u) + \mu \eta _1(u) + O(\mu ^2)$ one obtains, for $\eta _0$ the same equation \eqref{eq:edoeta0}.
For $\eta_1$, it appears a new term instead:
$$
\eta'_1=-\frac{8}{(4\eta _0-\frac{\varphi''(1)}{2}u^2)^2}\eta _1 + \frac{\varphi''(1)u^3}{3(4\eta _0-\frac{\varphi''(1)}{2}u^2)^2}
+\frac{2(b+\frac{\partial f_2}{\partial y}(0,0))}{(4\eta _0-\frac{\varphi''(1)}{2}u^2)^2}
$$
Nevertheless the asymptotic behavior at $-\infty$ is the same as \eqref{asineta1}:
$$
\eta_1\simeq \frac{\varphi''(1)}{24}u^3+\OO(u^4), \ u\to -\infty,
$$
then, Proposition \ref{prop:blocinner} also works, and we will arrive at $v=1$ having:
$$
x(1;\eps)= \eps ^{2/3}\eta _0(0)+ \OO(\eps).
$$

To see that the Fenichel manifold attracts points near $(0,1)$, concretely points of the section $\{ (x,v), \ v=1, \ -L\le x\le -\eps ^\la\}$,  $ 0<\la <2/3$, we proceed as we did in section \ref{sec:exponential} proving propositions \ref{prop:atractiogran1} and \ref{prop:atractiogran}.
The only thing to bear in mind, as Remark \ref{rem:canvivar} does, is that, in spite $v=m_0(x)$ is no longer a isocline of slope zero, the inequality
$$
m(x;\eps )<m_0(x)
$$
also is satisfied if the constant $L$ appearing in Fenichel theorem \ref{thm:fenichel} is small enough, but fixed.
The reason is that the term $m_1(x)$ in the expansion of the Fenichel manifold:
$$
m(x;\eps)=m_0(x)+\eps m_1(x)+\OO(\eps^2)
$$
is 
$$
m_1(x)=- \frac{2(1+\varphi(m_0(x)))^2(1+f_1(x,0))}{(\varphi'(m_0(x))(2x-1))^2}
-\frac{1}{2}
\frac{(1+\varphi(m_0(x)))m_0(x)}{\varphi'(m_0(x))(2x-1)}( b + \frac{\partial f_2}{\partial y}(x,0))
$$
and we know that $f_1(x,y)= \OO(x,y)$, therefore, for $x$ near $zero$, the dominant term in this expression is
$$
- \frac{2(1+\varphi(m_0(x)))}{(\varphi'(m_0(x))(2x-1))^2} <0
$$
in this region. So we can ensure that $m(x;\eps) <m_0(x)$.
On the other hand, if we consider the isocline of zero slope $v=\bar m(x;\eps)$ defined by:
$$
1+2x+\varphi(v)(2x-1)
+(1+\varphi(v))(b\eps v+f_2(x,\eps v))=0
$$
one obtains that
$$
\bar m(x;\eps) = m_0(x)-\frac{\eps}{2}\frac{(1+\varphi(m_0(x)))m_0(x)}{\varphi'(m_0(x))(2x-1)}( b + \frac{\partial f_2}{\partial y}(x,0))+ \OO(\eps^2),
$$
and therefore we also have $m(x,\eps)<\bar m(x,\eps)$.
With all these considerations, one can prove propositions
\ref{prop:atractiogran1} and \ref{prop:atractiogran} for the general fold case, obtaining the same formulas \eqref{pepsilon}
for the Poincar\'{e} map $P_\eps$ in this case.

\section*{Acknowledgements}
The authors have been partially supported by the Spanish MCyT/FEDER grand MTM2009-06973 and the Catalan SGR grant 2009SGR859. 
The authors also thank Enric Fossas, Robert Gri$\tilde{n}$\'{o} and Mike J Jeffrey and for useful discussion about Filippov systems and Toni Susin for his help with the figures.  

Finally, the first author wants to  dedicate this paper to the memory of Carmel Bonet Rev\'{e}s, 
posthumous PH D, for his exemplar tenacity in achieve his research.

\bibliography{references}

\bibliographystyle{alpha}

\end{document}